\documentclass[reqno,11pt,a4paper]{amsart}
\usepackage[utf8]{inputenc}
\usepackage{etoolbox}
\usepackage{amsmath,amsfonts,amssymb,amsthm,graphicx,color,bbm}
\usepackage{mathrsfs}
\usepackage{yhmath}
\usepackage{xcolor}
\usepackage[colorlinks=true, pdfstartview=FitV, linkcolor=blue, 
            citecolor=blue, urlcolor=blue]{hyperref}

 \allowdisplaybreaks
\numberwithin{equation}{section}

\def\be{\begin{equation}}
\def\ee{\end{equation}}
\def\bea{\begin{eqnarray}}
\def\eea{\end{eqnarray}}

\def\veps{\varepsilon}

\def\cS{{\mathcal S}} 

\def\cH{{\mathscr H}} 
\def\cM{{\mathscr M}} 
\def\cL{{\mathscr L}}

\def\R{{\mathbb R}}

\def\F{{\mathcal F}}
\def\N{{\mathbb N}}
\def\Z{{\mathbb Z}}
\def\bS{{\mathbb S}} 
\def\H{{\mathbb H}} 
\def\One{{\mathbbm 1}} 
\def\cI{{\mathcal I}}

\usepackage{mathtools}
\mathtoolsset{showonlyrefs}

\newcommand{\M}{{\mathcal M}}
\newcommand{\vol}{\mathrm{vol}} 
\newcommand{\supp}{{\rm supp}\ } 

\newcommand{\Lip}{\mathrm{Lip}}

\def\bal{\begin{aligned}}
\def\eal{\end{aligned}}

\def\proofof#1{\begin{proof}[Proof of #1]}

\def\part#1#2{\par\noindent{\underline{\it Part~#1.}}\emph{ #2}\\}

\def\({\left(}
\def\){\right)}

\DeclareMathOperator{\diam}{\rm diam}

\DeclareMathOperator{\inter}{\rm Int}

\DeclareMathOperator*{\argmin}{arg\,min}
\newcommand*\di{\mathop{}\!\mathrm{d}}

\newcommand{\mres}{\mathbin{\vrule height 1.6ex depth 0pt width
0.13ex\vrule height 0.13ex depth 0pt width 1.3ex}}

\def\XXint#1#2#3{{\setbox0=\hbox{$#1{#2#3}{\int}$} \vcenter{\vspace{-1pt}\hbox{$#2#3$}}\kern-.5\wd0}}

\theoremstyle{plain}
\newtheorem{lemma}{Lemma}[section]
\newtheorem{prop}[lemma]{Proposition}
\newtheorem{theorem}[lemma]{Theorem}
\newtheorem{example}[lemma]{Example}
\newtheorem{corol}[lemma]{Corollary}
\newtheorem{defin}[lemma]{Definition}
\newtheorem{remark}[lemma]{Remark}
\newtheorem*{remark*}{Remark}
\newtheorem*{recall*}{Recall}
\newtheorem*{notation*}{Notation}
\newtheorem*{hypothesis*}{Hypothesis}
\newtheorem*{theorem*}{Theorem}
\newtheorem*{conj*}{Conjecture}

\newcounter{mt}

\begin{document}

\title[Asymptotic quantization on metric measure spaces]{Asymptotics of the quantization problem on metric measure spaces}
\author{Ata Deniz Ayd{\i}n}
\address{ETH Z\"urich, Department of Mathematics, R\"amistrasse 101, 8092 Z\"urich, Switzerland.}
\email[Ata Deniz Ayd{\i}n]{deniz.aydin@math.ethz.ch}

\begin{abstract}
The problem of \emph{quantization of measures} looks for best approximations of probability measures on a metric space by discrete measures supported on $N$ points, where the error of approximation is measured with respect to the Wasserstein distance. 
\emph{Zador's theorem} states that, for measures on $\R^d$ or $d$-dimensional Riemannian manifolds satisfying appropriate integrability conditions, the quantization error decays to zero as $N \to \infty$ at the rate $N^{-1/d}$.

In this paper, we provide a general treatment of the asymptotics of quantization on metric measure spaces $(X, \nu)$. We show that a weaker version of Zador's theorem involving the Hausdorff densities of $\nu$ holds also in this general setting. 
We also prove Zador's theorem in full for appropriate $m$-rectifiable measures on Euclidean space, answering a conjecture by Graf and Luschgy in the affirmative.
For both results, the higher integrability conditions of Zador's theorem are replaced with a general notion of \emph{$(p,s)$-quantizability}, which follows from Pierce-type (non-asymptotic) upper bounds on the quantization error, and we also prove multiple such bounds at the level of metric measure spaces.
\end{abstract}

\maketitle

\section{Introduction}

Given a finite Borel measure $\mu$ on a complete and separable (\emph{Polish}) metric space $X$, the \emph{quantization problem} asks to find a measure supported on $N$ points which minimizes the $p$-Wasserstein distance to $\mu$:
\[ e_{N,p}(\mu) := \inf_{\mu_N = \sum_{i=1}^N a_i \delta_{x_i}} W_p(\mu, \mu_N). \]
{ This problem is meaningful for measures with finite $p$th moments, i.e.~$\int d(\cdot,x_0)^p \di\mu < \infty$ for some (equivalently all) $x_0 \in X$. }
In terms of the support $\supp\mu_N = \{x_i\}_{i=1}^N$, the problem reduces to the following minimization problem over sets of cardinality at most $N$ {(cf.~\cite[Lem.~3.4]{quantbook})}:
\begin{equation}\label{eq:setmin}
e_{N,p}(\mu) = \inf_{\# S \leq N} e_p(\mu; S); \quad e_p(\mu; S) := \|d(\cdot, S)\|_{L^p(X; \mu)} = \left\{ \begin{matrix} \left( \int_X d(x, S)^p \di\mu(x) \right)^{1/p}, & p < \infty; \\ \sup_{x \in \supp\mu} d(x, S), & p = \infty. \end{matrix} \right. 
\end{equation}

For example, suppose $X$ represents a geographical region, and $\mu$ describes the distribution of a population on the region. The points in $S$ are then chosen to be optimally reachable by the population in an $L^p$ average sense. For $p = 1$, $S$ is chosen such that the average distance to $S$ is minimized; as $p \to \infty$, $S$ is chosen such that the maximal distance to $S$ is minimized; other values of $p$ interpolate between these problems. 

{For other equivalent formulations of the quantization problem, we refer to \cite[Sect.~3]{quantbook} and \cite{Gruber2004}; while existing literature is mostly restricted to the setting $X = \R^d$, these equivalences extend readily to the metric setting.} We give a broad survey of the history and applications of the quantization problem in Section \ref{sect:hist}.

In domains such as $\R^d$ and Riemannian manifolds, under appropriate decay assumptions on $\mu$, the quantization error $e_{N,p}(\mu)$ is known to decay on the order of $N^{-1/d}$ as $N \to \infty$, where $d$ represents the dimensionality of the domain. For measures on $\R^d$, \emph{Zador's theorem} \cite{zador}, \cite{buckwise}, \cite[Thm.~6.2]{quantbook} provides the following precise expression for the asymptotics of the quantization error:
\begin{equation}\label{eq:zadorasym0}
\lim_{N \to \infty} N^{1/d} e_{N,p}(\mu) = C_{p,d} \left( \int_{\R^d} \rho(x)^{\frac{d}{d+p}} \di x \right)^{\frac{d+p}{dp}},
\end{equation}
provided the following higher integrabiity condition holds:
\begin{equation}\label{eq:momentcond0}
\int_{\R^d} |x|^{p+\delta} \di\mu(x) < \infty \quad \text{for some } \delta > 0, 
\end{equation}
where $\rho \in L^1(\R^d)$ the density of the absolutely continuous component of $\mu$ with respect to the Lebesgue measure, and $C_{p,d} \in (0,\infty)$ is a constant, in general unknown, which depends only on $p$ and $d$. 

Graf and Luschgy \cite[Thm.~7.5]{quantbook} prove moreover that for $\mu$ absolutely continuous, optimal quantizers are asymptotically distributed according to the density
\[ \frac{\rho^{\frac{d}{d+p}} }{\int_{\R^d} \rho^{\frac{d}{d+p}} \di x}. \]

The complete proof of Zador's theorem, as given by Graf and Luschgy \cite[Thm.~6.2]{quantbook}, starts from the uniform measure on the unit cube and generalizes to broader and broader classes of measures by approximation.
This argument relies on \emph{Pierce's lemma} \cite[Thm.~1]{pierce}, \cite[Lem.~6.6]{quantbook}, which gives a uniform integral upper bound on the quantization error:
\begin{equation}\label{eq:pierce}
\left( N^{1/d} e_{N,p}(\mu) \right)^p \leq C \int_{\R^d} (1 + |x|^{p+\delta}) \di \mu(x) 
\end{equation}
for arbitrary $\mu$ and $N \in \N$, with $C > 0$ depending only on $p, d$ and $\delta$. 

Similar \emph{Pierce-type} non-asymptotic upper bounds on the quantization error are key to extending Zador's theorem to more general domains. For compactly supported measures on Riemannian manifolds, Zador's theorem follows by a direct reduction to $\R^d$ along finitely many small charts \cite{Gruber2001, Gruber2004, discapprox, iacasym}, while measures of non-compact support require stronger integrability conditions, which include an additional term involving the large-scale growth of the manifold \cite[Thm.~1.4]{iacasym}, \cite[Thm.~1.7]{covgrow}. For example, on the hyperbolic space $\H^d$, Zador's theorem can only be shown to hold under an exponential moment condition $\int e^{p d(x,x_0)} \di \mu(x) < \infty$ \cite[Cor.~1.6]{iacasym}, and no polynomial moment condition is sufficient \cite[Thm.~1.7]{iacasym}. We review these results in more detail in Section \ref{sect:hist}.

The same decay rate of $N^{-1/s}$ has also been established for various classes of measures with fractal dimension $s \in (0,\infty)$, but without the existence of the limit $\lim_{N \to \infty} N^{1/s} e_{N,p}(\mu)$. 
\emph{Ahlfors regular} measures, which satisfy uniform concentration inequalities of the form $\alpha r^s \leq \mu(B_r(x)) \leq \beta r^s$, admit quantitative estimates
\[ C_1 \beta^{-1/s} \leq \liminf_{N \to \infty} N^{1/s} e_{N,p}(\mu) \leq \limsup_{N \to \infty} N^{1/s} e_{N,p}(\mu) < C_2 \alpha^{-1/s} \]
for explicit constants $C_2 > C_1 > 0$ \cite[Sect.~12]{quantbook}, \cite[Sect.~4]{discapprox}. 

For the particular case of compact rectifiable curves $\Gamma \subset \R^d$, equipped with the Hausdorff measure $\mu = \cH^1|_\Gamma$, Graf and Luschgy \cite[Thm.~13.12]{quantbook} showed that the following analogue of Zador's theorem does hold:
\[ \lim_{N \to \infty} N e_{N,p}(\cH^1|_\Gamma) = C_{p,1} \cH^1(\Gamma)^{\frac{1+p}{p}}. \]
{
In light of this result, Graf and Luschgy have made the following conjecture regarding rectifiable measures on $\R^d$:
\begin{conj*}[Graf and Luschgy {\cite[Rem.~13.13]{quantbook}}]
Let $\mu$ be an \emph{$m$-rectifiable measure} on $\R^d$, i.e.~$m \in \{1, \ldots, d\}$, $\mu$ is absolutely continuous with respect to $\cH^m$, and $\mu$ is supported on a countable union of $C^1$ manifolds. Then for all $1 \leq p \leq \infty$, the limit
\[ \lim_{N \to \infty} N^{1/m} e_{N,p}(\mu) \]
exists and lies in $(0,\infty)$.
\end{conj*}
We answer this conjecture in the case $p < \infty$ under a more general definition of rectifiability, obtaining also the explicit value of the limit under appropriate assumptions, and providing counterexamples to the conjecture when the assumptions are not satisfied.}

\subsection{Main definitions and results}

For general measures on Polish metric spaces, we characterize the asymptotics of the quantization error in terms of \emph{quantization coefficients}:

\begin{defin}
The \emph{lower} resp.~\emph{upper quantization coefficient} of $\mu$ of order $p \in [1,\infty]$ and dimension $s \in (0,\infty)$ is given by
\[ \underline Q_{p,s}(\mu) := \liminf_{N \to \infty} N^{1/s} e_{N,p}(\mu); \quad \overline Q_{p,s}(\mu) := \limsup_{N \to \infty} N^{1/s} e_{N,p}(\mu). \]
If the limit exists, it is called the \emph{quantization coefficient} and denoted by $Q_{p,s}(\mu)$.
\end{defin}

{The quantization coefficients of $\mu$ are nontrivial only for a single value of $s$, called the lower resp.~upper \emph{quantization dimension} of $\mu$ and denoted by $\underline D_p(\mu)$ resp.~$\overline D_p(\mu)$: $\underline Q_{p,s}(\mu) = \infty$ for $s < \underline D_p(\mu)$ and $\underline Q_{p,s}(\mu) = 0$ for $s > \underline D_p(\mu)$, and likewise for $\overline Q_{p,s}(\mu)$ \cite[Prop.~11.3]{quantbook}. We review quantization dimensions in Section \ref{sect:quantdims}, but will focus mostly on quantitative estimates on quantization coefficients.}

For $p < \infty$, we also introduce the following notion which captures the role of higher integrability conditions for measures on $\R^d$ or Riemannian manifolds: 

\begin{defin}
A {finite Borel measure $\mu$ with finite $p$th moments} is \emph{$(p,s)$-quantizable of order $p \in [1,\infty)$ and dimension $s \in (0,\infty)$} if $\overline{Q}_{p,s}(\mu_k) \to 0$ for any sequence of measures $\mu_k \leq \mu$ such that $\mu_k(X) \to 0$.

{Equivalently, for each $\veps > 0$ there exists $\delta > 0$ such that any measure $\nu \leq \mu$ with $\nu(X) < \delta$ satisfies $\overline{Q}_{p,s}(\nu) < \veps$.}
\end{defin}

This condition is named by analogy with the notion of O-quantizability defined by Zador \cite{zador}, which concerns instead the existence of the limit.
{
{This assumption signifies that regions of small measure have uniformly small quantization errors at an $O(N^{-1/s})$ rate, and is the minimal condition required for various approximation arguments in the proof of Zador's theorem to carry over to the metric space setting.}

Importantly, Pierce-type integral upper bounds of the form $\overline{Q}_{p,s}(\mu)^p \leq \int F \di \mu$ imply that every measure satisfying $\int F \di \mu < \infty$ is $(p,s)$-quantizable. 
Bounds of this form will be the most common tool we will use to obtain the $(p,s)$-quantizability of measures in various settings.
}

The condition of $(p,s)$-quantizability will allow us to obtain fine estimates of quantization coefficients on general metric measure spaces, and also show the existence of the limit for rectifiable measures on $\R^d$. 
These results will follow from general properties of quantization coefficients under operations such as countable sums, abstracted from proofs of Zador's theorem but valid for general Polish metric spaces, which are presented in Section \ref{sect:prel} and proven in full in Appendix \ref{app:quantcoeffs}.

\subsubsection{Density bounds in metric measure spaces}
In Section \ref{sect:densbds}, we characterize the quantization coefficients of measures in the general setting of \emph{metric measure spaces}, i.e.~a Polish metric space $X$ endowed with a locally finite Borel measure $\nu$. This includes $(\R^d, \cL^d)$ and $(M, \vol_g)$ for a Riemannian manifold $(M, g)$, but also allows for non-smooth domains. We seek to find analogues of Zador's theorem for such spaces, with the measure $\nu$ taking the place of the Lebesgue measure or the Riemannian volume form.

In this general setting, while the existence of the limit in Zador's theorem does not carry over, we can still obtain matching lower and upper bounds on quantization coefficients. We first recall the notion of upper and lower densities:

\begin{defin}
The $s$-dimensional \emph{upper} resp.~\emph{lower (Hausdorff) density} of $\nu$ at a point $x \in X$ is given by the upper resp.~lower limit
\[ \overline{\vartheta}^{(\nu)}_s(x) := \limsup_{r\to0^+} \frac{\nu(B_r(x))}{\omega_s r^s}; \quad \underline{\vartheta}^{(\nu)}_s(x) := \liminf_{r\to0^+} \frac{\nu(B_r(x))}{\omega_s r^s}, \]
where $\omega_s := \frac{\pi^{s/2}}{\Gamma(1+s/2)}$ coincides with the volume of the unit $s$-dimensional ball for $s$ integer.

For $\delta > 0$ fixed, we also consider the approximate densities
\[ \overline{\vartheta}^{(\nu)}_s(x, \delta) := \sup_{r < \delta} \frac{\nu(B_r(x))}{\omega_s r^s}; \quad \underline{\vartheta}^{(\nu)}_s(x, \delta) := \inf_{r < \delta} \frac{\nu(B_r(x))}{\omega_s r^s}, \]
so that $\overline{\vartheta}^{(\nu)}_s(x) = \lim_{\delta \to 0} \overline{\vartheta}^{(\nu)}_s(x, \delta)$ and $\underline{\vartheta}^{(\nu)}_s(x) = \lim_{\delta \to 0} \underline{\vartheta}^{(\nu)}_s(x, \delta)$.
\end{defin}

For a $d$-dimensional Riemannian manifold $(M, g)$, the $d$-dimensional Hausdorff density for the Riemannian volume form $\vol_g$ exists and is identically $1$. 
Observe that $\nu$ is Ahlfors regular of dimension $s$ iff $\nu$ admits uniform bounds of the form $\alpha \leq \underline{\vartheta}^{(\nu)}_s(\cdot, \delta) \leq \overline{\vartheta}^{(\nu)}_s(\cdot, \delta) \leq \beta$ on $\supp\nu$.

Hausdorff densities capture the dimensionality of $\nu$ as follows: if $\underline{\vartheta}^{(\nu)}_s(x) > 0$, then $\nu$ is at most as diffuse as an $s$-dimensional measure; likewise $\overline{\vartheta}^{(\nu)}_s(x) < \infty$ implies that $\nu$ is at least $s$-dimensional. This is made precise by the notion of upper resp.~lower \emph{local dimensions} for measures, cf.~\cite[Sect.~10.1]{falc97}, {which we review also in Section \ref{sect:quantdims}}. 

With this definition, we can state the precise result that we will prove in Section \ref{sect:densbdproof}:

\begin{theorem}[Hausdorff density bounds] \label{thm:densbd}
Let $(X, \nu)$ be a Polish metric measure space, $p \in [1,\infty)$, $s \in (0,\infty)$, and set $p^\prime := \frac{sp}{s+p}$.
Then for any {finite Borel measure $\mu$ on $X$ with finite $p$th moments}, 
\[ \underline{Q}_{p,s}(\mu)^{p^\prime} \geq C_1 \int \rho(x)^{\frac{s}{s+p}} \overline{\vartheta}^{(\nu)}_s(x)^{-\frac{p}{s+p}} \di\nu(x), \]
and if in addition $\mu$ is $(p,s)$-quantizable and $\underline{\vartheta}^{(\nu)}_s > 0$ $\mu$-a.e.,
\[ \overline{Q}_{p,s}(\mu)^{p^\prime} \leq C_2 \int \rho(x)^{\frac{s}{s+p}} \underline{\vartheta}^{(\nu)}_s(x)^{-\frac{p}{s+p}} \di\nu(x), \]
where $C_2 > C_1 > 0$ are explicit constants depending only on $p$ and $s$, and $\rho$ is the density of the absolutely continuous component of $\mu$ with respect to $\nu$.
\end{theorem}

In particular, if the density $\vartheta^{(\nu)}_s$ exists or is bounded above and below, the quantization coefficients of $\mu$ will be sandwiched between different multiples of the same integral. 
{
In Section \ref{sect:examples} we demonstrate this with three example settings, also discussed below, where one can fix a measure $\nu$ whose upper and lower densities are comparable and bounded from below.
The free choice of $\nu$ thus extends the applicability of the theorem even when the Hausdorff densities of $\mu$ itself are difficult to estimate.

On the other hand, by choosing $\nu = \mu$ we can deduce qualitative relations between the quantization dimensions of $\mu$ and its Hausdorff and packing dimensions. The precise statement is given in Corollary \ref{cor:quantdims}, which also leads to a compactly supported self-similar counterexample to $(p,s)$-quantizability (Ex.~ \ref{ex:selfsimcounter}). 
}

The proof of Theorem \ref{thm:densbd} starts from preexisting concentration inequalities for Ahlfors regular measures, restated in Appendix \ref{app:conc}, and relaxes to broader and broader classes of measures using the general properties presented in Section \ref{sect:prel}. The constants $C_2 > C_1 > 0$ originating from these inequalities {are explicit} but will not be precise in general. As such, consequences of the existence of the limit such as the asymptotic spatial distribution of quantizers cannot be deduced from this result.

\subsubsection{Sufficient conditions for $(p,s)$-quantizability} 
We also use Hausdorff densities to formulate general sufficient conditions for measures on $(X,\nu)$ to be $(p,s)$-quantizable. We {generalize to the metric setting the \emph{random quantizer} arguments given in} \cite[Thm.~9.1, 9.2]{quantbook}, \cite[Thm.~2]{Potz2001} and \cite[Thm.~4.2, 4.3]{lossy-2021}, where quantizers for $\mu$ are sampled independently from another probability measure $\nu$, in order to obtain the following integral upper bound:

\begin{theorem}[Random quantizer condition]\label{thm:randstab}
Let $X$ be a Polish metric space, $p \in [1,\infty)$, $s \in (0,\infty)$. Suppose there exists a Borel probability measure $\nu$ on $X$ which satisfies the decay condition
\[ 1 - \nu(B_R(x_0)) \leq (\beta R)^{-\alpha} \quad \text{for all } R \geq R_0 \]
for some $x_0 \in X$ and constants $\alpha, \beta, R_0 > 0$. 
Then there exists {$C = C(p,s,\alpha,\beta,R_0) > 0$} such that, for any {finite Borel measure $\mu$ on $X$ with finite $p$th moments}, 
{$N \in \N$} and $\delta > 0$, the following inequality holds:
\[ N^{p/s} e_{N,p}(\mu)^p \leq C \int_X \left[ 1 + d(x,x_0)^p \right] \left[ 1 + (1 + \delta^{-p}) \underline{\vartheta}_s^{(\nu)}(x, \delta)^{-p/s} \right] \di \mu(x). \]
Consequently, if 
\[ \int_X \left[ 1 + d(x,x_0)^p \right] \underline{\vartheta}_s^{(\nu)}(x, \delta)^{-p/s} \di \mu(x) < \infty \]
for some $\delta > 0$ sufficiently small, then $\mu$ is $(p,s)$-quantizable.
\end{theorem}

Note that since $\underline{\vartheta}_s^{(\nu)}(x, \delta) \nearrow \underline{\vartheta}_s^{(\nu)}(x)$, the condition becomes weaker as $\delta$ is decreased. {However, the $\underline{\vartheta}_s^{(\nu)}(x, \delta)$ term cannot be replaced with $\underline{\vartheta}_s^{(\nu)}(x)$ due to the $\delta^{-p}$ term in the upper bound. }

From this theorem, we deduce the following corollary for metric measure spaces $(X, \nu)$, where the measure $\nu$ is instead locally finite:

\begin{corol}[Volume growth condition]\label{cor:volgrowstab}
Let $X$ be a Polish metric space, $p \in [1,\infty)$, $s \in (0,\infty)$. 

Let $\nu$ be a locally finite Radon measure on $X$ such that, for some $x_0 \in \supp\nu$, $\beta > 1$ and $R_0 > 0$, there exists a nondecreasing function $V \colon [0,\infty) \to [0,\infty)$ such that
\[ \nu(B_{R+R_0}(x_0) \setminus B_{R}(x_0)) \leq V(R) \quad \text{for all } R \geq R_0. \]
{Then any finite Borel measure $\mu$ on $X$ with finite $p$th moments, which satisfies}
\[ \int_X \left[ 1 + {d(x,x_0)^p} + d(x,x_0)^{p+{(\alpha+1)} p/s} V(d(x,x_0))^{p/s} \right] \underline{\vartheta}_s^{(\nu)}(x, \delta)^{-p/s} \di \mu(x) < \infty \]
for some $\alpha, \delta > 0$, 
is $(p,s)$-quantizable.
\end{corol}

{We prove these statements in Section \ref{sect:randquant}. In both statements, we can identify two main obstacles to $(p,s)$-quantizability: when the measure decays too slowly at infinity, and when the support or the domain itself is too thin as to cause the lower density to be too close to $0$. The aforementioned Example \ref{ex:selfsimcounter} is an example of the latter case, while the former case is illustrated e.g.~by the classical counterexample \cite[Ex.~6.4]{quantbook} to Zador's theorem. } 

{We also present in Section \ref{sect:examples} three example settings with different levels of generality: Ahlfors regular domains, doubling spaces and \emph{curvature-dimension spaces} in the sense of Lott--Sturm--Villani \cite{LottVillani, sturm, sturmii}. For each setting we apply Corollary \ref{cor:volgrowstab} to deduce tractable integrability conditions for $(p,s)$-quantizability, and deduce from Theorem \ref{thm:densbd} matching lower and upper bounds on quantization coefficients.}

\subsubsection{Zador's theorem for rectifiable measures}
Under the assumption of rectifiabillity, we {can further strengthen Theorem \ref{thm:densbd}}: in Section \ref{sect:zadorrect}, we show that Zador's theorem is indeed satisfied for $(p,m)$-quantizable $m$-rectifiable measures on $\R^d$, with the Lebesgue measure replaced by the $m$-dimensional Hausdorff measure {$\cH^m$}.
This answers the conjecture posed by Graf and Luschgy \cite[Rem.~13.13]{quantbook} in the positive, up to the assumption of $(p,m)$-quantizability. The complete statement is as follows:

\begin{theorem}[Zador's theorem, rectifiable measures on $\R^d$] \label{thm:zadorrect}
{Let $m, d \in \N$ with $m \leq d$, and let $\mu$ be a countably $m$-rectifiable finite Borel measure on $\R^d$ \emph{(not necessarily $\ll \cH^m$)} such that $\mu$ has finite $p$th moments for some $p \in [1,\infty)$.}
Then 
\begin{equation}\label{eq:zadorrectineq} 
\underline{Q}_{p,m}(\mu) \geq C_{p,m} \left( \int_{\R^d} \rho^{\frac{m}{m+p}} \di\cH^m \right)^{\frac{m+p}{mp}}, 
\end{equation}
where $\rho$ is the density of the absolutely continuous component of $\mu$ with respect to $\cH^m$. 

If moreover $\mu$ is $(p,m)$-quantizable, $Q_{p,m}(\mu)$ exists and coincides with the right-hand side. Consequently, if the right-hand side is finite and positive, then for any asymptotically optimal sequence $(S_N)_{N}$ of quantizers for $\mu$, we have
\[ \frac{1}{N} \sum_{x \in S_N} \delta_x \xrightharpoonup[N \to \infty]{} \frac{\rho^{\frac{m}{m+p}}}{ \int_{\R^d} \rho^{\frac{m}{m+p}} \di\cH^m } \cH^m. \]
\end{theorem}

{This result improves upon Theorem \ref{thm:densbd} in two ways: by the presence of the exact constant $C_{p,m}$ rather than lower and upper bounds, and by accounting for the singular component of $\mu$ (see Remark \ref{rem:densbdrect}). {In relation to the original conjecture of Graf and Luschgy, the theorem also allows for rectifiable measures with singular components, provides an explicit value of the quantization coefficient, and subsumes also the case $\int_{\R^d} \rho^{\frac{m}{m+p}} \di\cH^m = \infty$. } }

Theorem \ref{thm:zadorrect} applies in particular to measures supported on $C^1$ submanifolds with boundary or corners. We note however that unlike previous results on Riemannian manifolds, Theorem \ref{thm:zadorrect} also allows for points to be selected outside the submanifold on which the measure is supported. 
Though the quantization error depends a priori on the space from which quantizers are selected, the asymptotic rate remains the same in this case.

For ambient spaces other than $\R^d$, this statement should also hold if the ambient space is a Hilbert space or Riemannian manifold, but not for general metric spaces. 
For example, if we consider $X = (\R^d, \|\cdot\|_1)$, then Zador's theorem holds on $X$ but with a different constant in place of $C_{p,d}$ (see \cite[Thm.~6.2, Rem.~8.14]{quantbook}), hence Theorem \ref{thm:zadorrect} cannot also hold for $d$-rectifiable measures on $X$. 

Heuristically, the constant $C_{p,m}$ reflects the infinitesimal Euclidean structure of the approximate tangent spaces of rectifiable sets on $\R^d$. In the general metric setting, different tangent spaces will have different norms associated to them (cf.~\cite[Thm.~9]{kirch94}), hence the constant will also be a function to be included inside the integral. Compare the remark given in \cite[Sect.~2.2]{Gruber2004} relating to Riemannian vs.~Finslerian manifolds.

For the special case $m = 1$, however, all norms on $\R$ are equivalent, and we can indeed prove the same theorem on any Polish metric space:

\begin{theorem}[Zador's theorem, $1$-rectifiable measures] \label{thm:zadorrect1}
Let $X$ be a Polish metric space.
{Let $\mu$ be a countably $1$-rectifiable measure on $X$ with finite $p$th moments for some $p \in [1,\infty)$.}
Then
\begin{equation}\label{eq:zadorrectineq1} \underline{Q}_{p,1}(\mu) \geq C_{p,1} \left( \int_{X} \rho^{\frac{1}{1+p}} \di\cH^1 \right)^{\frac{p+1}{p}}, \end{equation}
where $\rho$ is the density of the absolutely continuous component of $\mu$ with respect to $\cH^1$. 

If moreover $\mu$ is $(p,1)$-quantizable, $Q_{p,1}(\mu)$ exists and coincides with the right-hand side. Consequently, if the right-hand side is finite and positive, then for any asymptotically optimal sequence $(S_N)_N$ of quantizers for $\mu$, we have
\[ \frac{1}{N} \sum_{x \in S_N} \delta_x \rightharpoonup \frac{\rho^{\frac{1}{1+p}}}{ \int_{X} \rho^{\frac{1}{1+p}} \di\cH^1 } \cH^1. \]
\end{theorem}

The proof of the theorem relies on the existence of arc length parametrizations for rectifiable curves. 
Both proofs employ classical Lipschitz extension theorems, namely the theorems of McShane and Kirszbraun, in order to map quantizers in the ambient domain back to $\R^m$ and thus reduce to the Euclidean case.

{We also remark that verifying the $(p,m)$-quantizability of rectifiable measures generally requires an approach specific to the support of the measure. This can be observed e.g.~for measures supported on $m$-dimensional embedded submanifolds of $\R^d$ with arbitrarily negative curvature. Even when the measure is compactly supported, for $m < d$ it does not follow that the measure is $(p,m)$-quantizable or even that the integral $\int_{\R^d} \rho^{\frac{m}{m+p}} \di\cH^m$ is finite; see Example \ref{ex:rectcounter} for a counterexample. {These examples thus illustrate that in general, the finiteness of the quantization coefficient as conjectured by Graf and Luschgy may not hold if the assumption of $(p,m)$-quantizability is not satisfied.} Section \ref{sect:psstabrect} discusses the $(p,m)$-quantizability of rectifiable measures in more detail, showing in particular that measures supported on compact sets with \emph{finite Minkowski content} are $(p,m)$-quantizable (Proposition \ref{prop:minconstab}).  }

\subsection{Historical notes}\label{sect:hist}

The quantization problem has been treated independently in various fields under multiple equivalent formulations. The bulk of the classical literature on the asymptotics of quantization falls under the scope of information theory, from which the name \emph{quantization} also originates, and dates back to the 1940s for $X = \R$ (\emph{scalar quantization}) \cite{bennett, PCM, pandit, lloyd} and the 1970s for $X = \R^d$ (\emph{vector quantization}) \cite{pierce, Gersho, zador, buckwise}. We refer to \cite{gersho92vector, grayneuhoff} for more detailed surveys of quantization in this context, and also to \cite[Sect.~3]{quantbook} for the different equivalent formulations of the quantization problem.

Independently, already in the 1950s Steinhaus \cite{steinhaus} introduced an equivalent formulation of the quantization problem on $\R^d$, in terms of sums of moments, and L. Fejes T\'oth \cite{Fejes-Toth:1959vj} treated the asymptotics of this formulation for $d = 2$, proving \eqref{eq:zadorasym0} with an explicit value for $C_{2,2}$ as the normalized moment of the regular hexagon. The latter result follows from Fejes T\'oth's more general \emph{moment-sum theorem} \cite[3.8]{Fejes-Toth:2023aa} (see also \cite{newman, Gruber:1999tm, Boroczky:2010}), which holds for any monotone function of the distance in place of the $p$th power, and implies the asymptotic optimality of the hexagonal lattice for the quantization of the unit square. 
This result has also been used to construct asymptotically optimal quantizers for non-uniform measures \cite{Su:1997aa}, and admits a stability version valid also on Riemannian $2$-manifolds \cite[Thm.~1]{Gruber2001}.
We refer to \cite{Gruber2001, Gruber2004} and \cite[8.3.6]{Fejes-Toth:2023aa} for a survey of the applications and generalizations of the moment-sum theorem, in particular for the approximation of convex bodies by polyhedra and the isoperimetric problem for polyhedra with a fixed maximum number of faces.

In light of the asymptotic optimality of the hexagonal lattice for $d = 2$, Gersho \cite{Gersho} conjectured the existence of lattice quantizers attaining the optimal constant $C_{2,d}$ also for $d \geq 3$; this conjecture remains open to this day. For $d = 1$ the conjecture holds trivially, with uniform interval partitions yielding optimal quantizers of $[0,1]$ for all $p$ and $N$, with the explicit formula $C_{p,1}^p = \frac{1}{2^p (p+1)}$ \cite[Ex.~5.5]{quantbook}. See \cite{barslo, dwyer, Choksi:2020aa} for partial results in dimension $3$, \cite{Gruber2004, Zhu:2020aa} for weaker results in more general settings, and \cite{Gersho, Conway:1999aa, cvt, cvt:2005} for numerics and further discussion. 

The original (incomplete) proofs of Zador's theorem as given in \cite{zador, buckwise}, as well as prior results for $d = 1$ \cite{bennett, pandit}, were based on a \emph{compander} approach, generating quantizers for nonuniform distributions as a nonlinear distortion of uniform quantizers and then optimizing the distortion, allowing also the limiting density of quantizers to be characterized directly. As remarked already by Gersho \cite[Sect.~IX]{Gersho}, this approach is not actually valid for $d \geq 2$: the limiting point densities for generic probability distributions $\rho \di x$ cannot be generated by such a distortion unless $\log \rho$ is a harmonic function, which rules out almost all applicable distributions in practice. The compander approach has nevertheless proven useful in more recent works on the dynamical optimization of quantizers by gradient flows, both for one-dimensional nonuniform measures \cite{Caglioti_2015} and for the stability of the hexagonal lattice for uniform measures in dimension $2$ \cite{Caglioti:2018, Iacobelli:2018tz}, motivating the study of the resulting \emph{ultrafast diffusion equations} in their own right \cite{Iacobelli:2019aa, Iacobelli:2019wz}.

The complete proof of Zador's theorem due to Graf and Luschgy \cite[Thm.~6.2]{quantbook} first shows that
\[ \lim_{N \to \infty} N^{1/d} e_{N,p}(\cL^d|_{[0,1]^d}) = \inf_{N \in \N} N^{1/d} e_{N,p}(\cL^d|_{[0,1]^d}) =: C_{p,d} > 0 \]
for the uniform measure on the unit cube, then treats finite convex combinations of cubes $\mu = \sum_{i=1}^m \lambda_i \cL^d|_{C_i}$, and finally approximates general densities $\rho \in L^1(\R^d)$ by such finite combinations, controlling the approximation error using Pierce's lemma. 
Recent proofs of the theorem for compactly supported measures on Riemannian manifolds \cite[Thm.~1]{Gruber2004}, \cite[Thm.~1.2]{discapprox}, \cite[Sect.~2]{iacasym} follow a similar strategy, taking cubes contained in sufficiently small coordinate charts on which the metric tensor is approximately constant. From this point, the theorem can be extended to measures of noncompact support \cite[Thm.~1.4]{iacasym}, \cite[Thm.~1.7]{covgrow} by making use of different upper bounds in place of Pierce's lemma, which does not hold on manifolds with negative curvature.

The first such Pierce-type upper bound on Riemannian manifolds was found by Iacobelli \cite[Thm.~3.1]{iacasym}, with an additional term involving the curvature of the manifold:
\begin{equation}\label{eq:iacmomentcond}
\left( N^{1/d} e_{N,p}(\mu) \right)^p \leq C \int_{M} \left[ 1 + d(x,x_0)^{p+\delta} + A_{x_0}(d(x,x_0))^p \right] \di \mu(x), 
\end{equation}
where $x_0$ is an arbitrary point on the manifold $M$, and the function $A_{x_0}(R)$ quantifies the magnitude of the differential of the exponential map $\exp_{x_0} \colon T_{x_0} M \to M$ on the sphere $\partial B_R(0) \subset T_{x_0} M$. If the sectional curvature of $M$ is lower bounded by some $-K^2 < 0$, this amounts to an exponential moment condition $\int_M e^{p K d(x,x_0)} \di\mu(x) < \infty$, which is shown to be sharp at least for the hyperbolic plane $\H^2$ \cite[Thm.~1.7]{iacasym}.

This upper bound was refined in \cite{covgrow} using a metric notion of \emph{covering growth}: a metric space $X$ is said to be of $O(f)$ covering growth of dimension $d$ around a point $x_0$ if, for any $k \in \N$ and $R > 0$, the geodesic sphere $\partial B_R(x_0) \subset X$ can be covered by $k^{d-1}$ balls of radius $\leq C f(R)/k$. 
If $X$ is a geodesic metric space of $O(f)$ covering growth of dimension $d$ around $x_0$, 
{ \cite[Thm.~1.6]{covgrow} shows that for any probability measure $\mu$ on $X$, the same upper bound as \eqref{eq:iacmomentcond} holds with $f$ in place of $A_{x_0}$:
\begin{equation}\label{eq:covcond}
\left( N^{1/d} e_{N,p}(\mu) \right)^p \leq C \int_{M} \left[ 1 + d(x,x_0)^{p+\delta} + f(d(x,x_0))^p \right] \di \mu(x).
\end{equation}
This upper bound extends Zador's theorem to measures on Riemannian manifolds for which the right-hand side is finite \cite[Thm.~1.7]{covgrow}. In particular, Riemannian manifolds with nonnegative Ricci curvature are shown to be of $O(R)$ covering growth \cite[Cor.~4.14]{covgrow}, and manifolds on which a group of polynomial growth acts by isometries also exhibit polynomial covering growth \cite[Prop.~1.8]{covgrow}. While these results were all presented in the scope of Riemannian manifolds, the Pierce-type upper bound and the estimates on covering growth do not involve the manifold structure of the space beyond its metric and volumetric properties, motivating the question of extending Zador's theorem to non-smooth settings with similar properties.
}

We also briefly survey the existing literature on asymptotic quantization for measures not of full dimension. Zador \cite{zador} introduced the notion of \emph{quantization dimensions} for measures, which are the upper resp.~lower limits of the ratio $\frac{\log N}{- \log e_{N,p}(\mu)}$ as $N \to \infty$: if the limit exists and equals some $s \in (0,\infty)$, then $e_{N,p}(\mu)$ decays on the order of $N^{1/s}$. 
Analogously to the Hausdorff dimension, the lower quantization dimension of $\mu$ is the unique $s$ such that $\underline Q_{p,t}(\mu) = \infty$ for $t < s$ and $\underline Q_{p,t}(\mu) = 0$ for $t > s$, and likewise for the upper quantization dimension. As presented in \cite[Sect.~11]{quantbook}, the upper/lower quantization dimensions of $\mu$ are sandwiched between the Hausdorff dimension of $\mu$ (i.e.~the smallest Hausdorff dimension of sets of full $\mu$-measure) and the upper/lower box-counting dimension of $\supp\mu$. See also \cite{Potz2001, Potz2003, quantnums, Kessebohmer:2015aa, Kesseb_hmer_2023} for more recent estimates on the quantization dimensions of measures.

The first treatment of the asymptotic quantization of Ahlfors regular measures, rectifiable curves and self-similar measures on $\R^d$ appears to be by Graf and Luschgy \cite[Sect.~12-14]{quantbook}, who first treated the Cantor measure three years earlier \cite{Graf:1997aa}. See also \cite{discapprox, Zhu:2020aa, lossy-2021} for the quantization of Ahlfors regular measures, and \cite{Graf:2000aa, Graf:2002aa, Zhu:2011aa, Kessebohmer:2015aa} for a further selection of literature on the quantization of self-similar measures.
We remark that for appropriate classes of self-similar sets such as the usual Cantor set, the quantization dimension of the Hausdorff measure restricted to the set is known to exist and coincide with the Hausdorff dimension of the set for all $p$, cf.~\cite[Sect.~14]{quantbook}. However, the limit $\lim_{N\to\infty} N^{1/s} e_{N,p}(\mu)$ itself does not exist for the Cantor measure $\mu$ at least for $p \in \{ 2, \infty\}$ \cite[Prop.~14.22, Rem.~14.23]{quantbook}.

For numerical applications and algorithms for quantization, we refer to \cite{cvt, Gruber2004, pagesintro, Brigant2019uo, Brigant2019vo, Merigot:2021}. We also note the recent application of asymptotic quantization to entropy-regularized optimal transport \cite{Eckstein2024}, in which the quantization dimension and coefficient both play a role in the convergence of the regularized cost to the classical transport cost.
We lastly cite further recent works on different formulations of the quantization problem which fall outside the scope of this paper, such as more general transport costs of the form $f(\|x-y\|)$ \cite{Delattre-2004, Gruber2004}, entropy constraints and penalizations instead of size constraints \cite{Gray:2002aa, Bouchitte:2002, Bouchitte:2011tp, Bourne:2015, Bourne:2021ue}, quantization by empirical measures (\emph{optimal matching}) \cite{Talagrand:1994aa, Graf:2002ww, Dereich:2013, Fournier:2015aa, Garcia-Trillos:2015wv, chevallier_2018, Merigot:2021, Quattr2024}, and the related problem of \emph{random matching} \cite{Ambrosio:2019wy, ambrosio2019finer, Ambrosio:2019us, Benedetto:2020aa, Benedetto:2021aa}.
For further background on these different formulations, we refer to \cite[Sect.~1.2]{covgrow}. 

\subsection{Organization}

In Section \ref{sect:prel}, we establish the notation and terminology we will use throughout the paper. We first list some elementary properties of the quantization error in Section \ref{sect:quanterr}, then define quantization coefficients and $(p,s)$-quantizability in Section \ref{sect:quantcoeff},  
and lastly provide in Section \ref{sect:addquant} a list of additivity properties for quantization coefficients, whose proofs are given in Appendix \ref{app:quantcoeffs}.

Section \ref{sect:densbds} concerns the asymptotics of quantization in metric measure spaces, in particular the proof of Theorem \ref{thm:densbd}. Based on existing concentration inequalities, recalled in Appendix \ref{app:conc}, we first derive rough estimates under uniform bounds on Hausdorff densities in Section \ref{sect:unifdens}, and deduce the theorem in Section \ref{sect:densbdproof}. {Section \ref{sect:quantdims} then presents the implications of Theorem \ref{thm:densbd} on quantization dimensions, including a self-similar counterexample to $(p,s)$-quantizability.}

{Section \ref{sect:psstab} is dedicated to sufficient conditions for $(p,s)$-quantizability. Section \ref{sect:randquant} contains the proofs of Theorem \ref{thm:randstab} and Corollary \ref{cor:volgrowstab}, and Section \ref{sect:examples} applies these results to the example settings of Ahlfors regular domains, doubling spaces and curvature-dimension spaces, obtaining in each case an explicit integrability condition and statement of Theorem \ref{thm:densbd}. }

Section \ref{sect:zadorrect} finally treats the quantization of rectifiable measures. Zador's theorem for $m$-rectifiable measures on $\R^d$, stated above as Theorem \ref{thm:zadorrect}, is proven in Section \ref{sect:zadorrectproof}, and Theorem \ref{thm:zadorrect1} for $1$-rectifiable measures on Polish spaces is proven in Section \ref{sect:zadorrect1proof}. {Section \ref{sect:psstabrect} finally discusses the $(p,m)$-quantizability of rectifiable measures and presents two $1$-rectifiable counterexamples. }

\section{Preliminary definitions and properties}\label{sect:prel}

Let $X$ be a Polish metric space. We denote by $C_b(X)$ the space of bounded continuous functions $X \to \R$, equipped with the $L^\infty$ norm, and the space of finite positive Borel measures on $X$ by $\M_+(X)$, equipped with the total variation norm $\|\mu\|_{TV} = \mu(X)$. We note the duality between $C_b(X)$ and $\M_+(X)$:
\[ \left| \int_X f(x) \di\mu(x) \right| \leq \|f\|_{\infty} \|\mu\|_{TV} \quad \text{for all } f \in C_b(X), \mu \in \M_+(X). \]
We denote the restriction of a measure $\mu \in \M_+(X)$ to a subset $A \subseteq X$ by $\mu|_A$, and the support of $\mu$ by $\supp\mu$.
We write $\mu \leq \nu$ if $\mu(A) \leq \nu(A)$ for every Borel set $A \subseteq X$, and denote absolute continuity resp.~mutual singularity by $\mu \ll \nu$ resp.~$\mu \bot \nu$.

Ulam's Theorem (cf.~\cite[7.1.4]{dudley}) states that for each $\mu \in \M_+(X)$, $A \subseteq X$ $\mu$-measurable and $\veps > 0$, there exists $K_\veps \subseteq A$ compact such that $\mu(A \setminus K_\veps) < \veps$. This also implies that finite Borel measures on Polish spaces are regular.

Given a Borel map $T \colon X \to Y$ between metric spaces, we denote the \emph{pushforward} of $\mu$ along $T$ by $T_\# \mu$, which is the Borel measure on $Y$ given by
\[ (T_\# \mu)(B) := \mu(T^{-1}(B)) \quad \text{for all } B \subseteq Y \text{ Borel}, \]
or equivalently, for any $f \in C_b(X)$ or otherwise nonnegative and Borel, 
\[ \int_Y f(y) \di(T_\# \mu)(y) = \int_X f(T(x)) \di\mu(x). \]
Consequently, $\|f\|_{L^p(T_\# \mu)} = \|f \circ T\|_{L^p(\mu)}$ for all $p \in [1,\infty]$. 

For $p < \infty$, we denote by $\M_+^p(X)$ the set of finite positive Borel measures $\mu \in \M_+(X)$ with \emph{finite $p$th moments}:
\[ \int_X d(x, x_0)^p \di \mu(x) < \infty \quad \text{for some } (\Leftrightarrow \text{all) } x_0 \in X. \]
For $p = \infty$, we instead define $\M_+^\infty(X)$ to be the set of compactly supported elements of $\M_+(X)$.

We denote the cardinality of a set $S \subseteq X$ by $\# S$, and denote by $\cS_N(X)$ the set of subsets $S \subseteq X$ with at most $N$ elements:
\[ \cS_N(X) := \{ S \subseteq X \mid \# S \leq N \}. \]
We denote the set distance function of $S$ by $d(x, S) := \inf_{a \in S} d(x, a)$, and the open tubular neighborhoods of $S$ by
\[ S^r := \{ x \in X \mid d(x, S) < r \} = \bigcup_{a \in S} B_r(a), \quad r > 0. \]

\subsection{Properties of the quantization error}\label{sect:quanterr}

We first introduce the quantization error and list a few basic properties. 

\begin{defin}[{Quantization error}]
Let $X$ be a Polish metric space, $p \in [1,\infty]$, $\mu \in \M_+^p(X)$. Given a subset $\emptyset \neq S \subset X$, the \emph{quantization error} of $\mu$ of order $p$ with respect to $S$ is defined as
\[ e_p(\mu; S) := \|d(\cdot, S)\|_{L^p(\mu)} = \left\{ \begin{matrix} \left( \int_X d(x, S)^p \di \mu(x) \right)^{1/p}, & p < \infty; \\ \max_{x \in \supp\mu} d(x,S), & p = \infty. \end{matrix} \right. \]
Given $N \in \N$, the \emph{$n$th quantization error} of $\mu$ of order $p$ is the infimum over all subsets of cardinality at most $N$:
\[ e_{N,p}(\mu) := \inf_{S \in \cS_N(X)} e_p(\mu; S). \]
{
For $p < \infty$, the notation $V_{N,p}(\mu)$ is also commonly used in place of $e_{N,p}(\mu)^p$. 
}
\end{defin}

\begin{notation*}
Given a subset $A \subset X$, we also denote $e_p(\mu|_A; S)$ by $e_p^{(\mu)}(A; S)$, and likewise for $e_{N,p}$. When the metric space $X$ from which quantizers are taken needs to be emphasized, we use the notation $e^{[X]}_{N,p}(\mu)$.
\end{notation*}

The quantization error of order $\infty$ evidently depends only on $\supp\mu$, and can be defined for arbitrary totally bounded subsets $A \subseteq X$:
\[ e_\infty(A; S) := \sup_{x \in A} d(x, S) = \inf\{ r > 0 \mid d(x, S) < r \ \forall x \in A \} = \inf\{r > 0 \mid A \subseteq S^r\}. \]
That is, $e_\infty(A; S)$ is the smallest radius $r$ for which the elements of $S$ define an $r$-cover of $A$, and $e_{N,\infty}(A)$ is the smallest $r$ for which there exists an $r$-cover of $A$ with at most $N$ elements. For this reason, $e_{N,\infty}(A)$ is also called the \emph{$N$th covering radius} of $A$ \cite[Sect.~10]{quantbook}. Observe also that $e_\infty(\bar A; S) = e_\infty(A; S)$ and the maximum is indeed attained on $\bar A$.

We note the following elementary list of properties of the quantization error, cf.~\cite[Lem.~4.14]{quantbook}:

\begin{prop}[Basic properties of $e_p$]\label{prop:enpprops}
Let $p \in [1,\infty]$, $\mu \in \M^p_+(X)$.
\begin{itemize}
\item[(i)] (Additivity) if $\mu = \sum_{i=1}^m \mu_i$, then for any $S \subseteq X$,
\begin{align*}
e_p(\mu; S)^p & = \sum_{i=1}^m e_p(\mu_i; S)^p \quad \text{for} \quad p < \infty; \qquad
e_\infty(\mu; S) = \max_{i=1}^m e_\infty(\mu_i; S).
\end{align*}
In particular, given $S_1, \ldots, S_m \subseteq X$, 
\begin{align*}
e_p(\mu; \bigcup_{i=1}^m S_i)^p & = \sum_{i=1}^m e_p(\mu_i; \bigcup_{i=1}^m S_i)^p \leq \sum_{i=1}^m e_p(\mu_i; S_i)^p \quad \text{for} \quad p < \infty; \\ 
e_\infty(\mu; \bigcup_{i=1}^m S_i) & = \max_{i=1}^m e_\infty(\mu_i; \bigcup_{i=1}^m S_i) \leq \max_{i=1}^m e_\infty(\mu_i; S_i). 
\end{align*}
Consequently, given $N, N_1, \ldots, N_m \in \N$ such that $N \geq \sum_{i=1}^m N_i$, we have
\begin{align*}
\sum_{i=1}^m e_{N,p}(\mu_i)^p & \leq e_{N,p}(\mu)^p \leq \sum_{i=1}^m e_{N_i,p}(\mu_i)^p \quad \text{for} \quad p < \infty; \\ 
\max_{i=1}^m e_{N,\infty}(\mu_i) & \leq e_{N,\infty}(\mu) \leq \max_{i=1}^m e_{N_i,\infty}(\mu_i).
\end{align*}
\item[(ii)] (Scaling) given $\lambda > 0$, then $e_p(\lambda\mu; S) = \lambda^{1/p} e_p(\mu; S)$ hence $e_{N,p}(\lambda\mu) = \lambda^{1/p} e_{N,p}(\mu)$.
\item[(iii)] (Monotony) given $\nu \leq \mu$, $e_p(\nu; S) \leq e_p(\mu; S)$ hence $e_{N,p}(\nu) \leq e_{N,p}(\mu)$.
\item[(iv)] (Order) given $q > p$, $e_p(\mu; S) \leq \mu(X)^{\frac{1}{p}-\frac{1}{q}} e_q(\mu; S)$ and likewise for $e_{N,p}(\mu)$.
\item[(v)] (Pushforwards) if $T \colon X \to Y$ is $(\alpha,\beta)$-bi-Lipschitz, $0 \leq \alpha \leq \beta < \infty$, then 
\[ \alpha e_p(\mu; S) \leq e_p(T_\# \mu; T(S)) \leq \beta e_p(\mu; S), \]
hence $e_{N,p}(T_\# \mu) \leq \beta e_{N,p}(\mu)$ and $\alpha e_{N,p}(\mu) \leq e_{N,p}^{[T(X)]}(T_\# \mu) \leq \beta e_{N,p}(\mu)$.
\end{itemize}
\end{prop}
\begin{proof}\
\begin{itemize}
\item (i), (ii), (iii) follow from basic properties of $L^p$ norms.
\item (iv) is an application of H\"older's inequality: for $q < \infty$,
\[ \int_X d(\cdot, S)^p \di \mu \leq \left( \int_X 1 \di \mu \right)^{1-\frac{p}{q}} \left( \int_X (d(\cdot, S)^p)^{q/p} \di \mu \right)^{\frac{p}{q}} = \mu(X)^{1-\frac{p}{q}} \left( \int_X d(\cdot, S)^q \di \mu \right)^{\frac{p}{q}}, \]
and for $q = \infty$,
\[ \int_X d(\cdot, S)^p \di \mu \leq \left( \sup_{\supp \mu} d(\cdot, S) \right)^p \int_X 1 \di \mu = e_\infty(\mu; S)^p \mu(X). \]
\item (v): for each $x \in X$, we have $d(T(x), T(S)) = \inf_{a \in S} d(T(x), T(a))$ with
\[ \alpha d(x,S) = \alpha \inf_{a \in S} d(x, a) \leq \inf_{a \in S} d(T(x), T(a)) \leq \beta \inf_{a \in S} d(x, a) = \beta d(x, S). \]
The first statement then follows from the fact that $\|f\|_{L^p(T_\# \mu)} = \|f \circ T\|_{L^p(\mu)}$. The second statement follows by taking the infimum, noting that each finite $S^\prime \subseteq T(X)$ is the image of a set $S \subseteq X$ with $\# S = \# S^\prime$:
\[ e_{N,p}(T_\# \mu) \leq e_{N,p}^{[T(X)]}(T_\# \mu) = \inf_{\substack{ S \subseteq X \\ \# S \leq N}} e_p(T_\# \mu; T(S)) \in [\alpha,\beta] \inf_{\substack{ S \subseteq X \\ \# S \leq N}} e_p(\mu; S) = [\alpha,\beta] e_{N,p}(\mu). \qedhere \]
\end{itemize}
\end{proof}

In particular, when $\iota \colon X \to Y$ is an isometric embedding, enlarging the domain decreases the extrinsic quantization error $e_{N,p}^{[Y]}(\iota_\# \mu)$ but preserves the intrinsic quantization error $e_{N,p}^{[\iota(X)]}$. 

We also recall the basic definition of Voronoi cells and partitions, cf.~\cite[Sect.~1.1]{quantbook}.

\begin{defin}
Given a finite subset $\emptyset \neq S \subset X$, the \emph{Voronoi cell} associated to an element $a \in S$ is the closed set
\[ W(a|S) := \{ x \in X \mid d(x,a) = d(x,S) \}. \]
An indexed Borel partition $(P_a)_{a \in S}$ of $X$ is called a \emph{Voronoi partition} if each $P_a \subseteq W(a|S)$.
We also denote by
\[ S \mres A := \{a \in S \mid W(a|S) \cap A \neq \emptyset\} = \{ a \in S \mid \exists x \in A \ d(x,a) = d(x,S) \} \]
the set of all elements of $S$ whose Voronoi cells intersect $A$.
\end{defin}

Observe that $d(\cdot, S \mres A) = d(\cdot, S)$ on $A$, hence the operation $S \mres A$ removes all the redundant elements of $S$ that do not see $A$.

Proposition \ref{prop:enpprops} (i) implies in particular that, for any Voronoi partition $(P_a)_{a \in S}$ associated to $S$,
\[ e_p(\mu; S)^p = \sum_{i=1}^m e_p^{(\mu)}(P_a; S)^p = \sum_{a \in S} e_p^{(\mu)}(P_a; a)^p. \]
Thus for $p < \infty$, the quantization problem seeks to minimize the sum of the moments of each $P_a$ around $a$. This is the formulation of the quantization problem first proposed by Steinhaus \cite{steinhaus}.

For reference to further properties, such as characterizations and existence of minimizers, within the scope of Euclidean spaces, we refer to Graf and Luschgy \cite[Sect.~4]{quantbook}.
This paper will focus on the asymptotics of quantization as $N \to \infty$, and seek to avoid any reference to optimal quantizers which may not be known to exist for more general metric spaces.

As a direct generalization of \cite[Lem.~6.1]{quantbook}, we first note that the quantization error always converges to zero as $N \to \infty$ without any further assumptions on the measure $\mu$:

\begin{prop}[Convergence to zero]\label{prop:enpconv}
Let $\mu \in \M^p_+(X)$, $p \in [1,\infty]$. Then $e_{N,p}(\mu) \to 0$ as $N \to \infty$.
\end{prop}
\begin{proof}\
\begin{itemize}
\item $p = \infty$: Set $K := \supp\mu$ compact. Let $\veps > 0$. Then there exists a finite set $S_\veps = \{x_i\}_{i=1}^{N_\veps}$ of points such that $\{B_\veps(x_i)\}_{i=1}^{N_\veps}$ is a cover for $K$. Thus $d(\cdot, S) < \veps$ on $K$, hence by compactness, $e_\infty(K; S) < \veps$ and $e_{N,\infty}(K) < \veps$ for each $N \geq N_\veps$.
\item $p < \infty$: Let $\veps > 0$. Fix $x_0 \in X$, so that $\int d(\cdot, x_0)^p \di \mu < \infty$. Then by Ulam's Theorem, there exists $K_\veps \subseteq X$ compact such that
\[ \int_{X \setminus K_\veps} d(\cdot, x_0)^p \di \mu < \frac{\veps^p}{2}. \]
By Proposition \ref{prop:enpprops} (i) and (iv) we have
\[ e_{N+1,p}(\mu)^p \leq e_{N,p}^{(\mu)}(K_\veps)^p + e_{1,p}^{(\mu)}(X \setminus K_\veps)^p \leq \mu(K_\veps) e_{N,\infty}(K_\veps)^p + \int_{X \setminus K_\veps} d(\cdot, x_0)^p \di \mu. \]
From the above argument, there exists $N_\veps \in \N$ for which $e_{N,\infty}(K_\veps)^p < \frac{\veps^p}{2 \mu(K_\veps)}$ whenever $N \geq N_\veps$. We then have $e_{N,p}(\mu) < \veps$ whenever $N \geq N_\veps+1$. \qedhere
\end{itemize}
\end{proof}

We refer to sequences of quantizers whose errors converge to $0$ as \emph{admissible sequences}:

\begin{defin}
An \emph{admissible (quantizing) sequence} for $\mu$ of order $p$ is a sequence of pairs $(N_k, S_k)_k$, to be denoted henceforth by $(S_k)_{N_k}$, such that $(N_k)_k$ is an increasing sequence of positive integers, each $S_k$ is a finite subset of $X$ with at most $N_k$ elements, and $e_p(\mu; S_k) \to 0$. 
\end{defin}

\begin{notation*}
Given an admissible sequence $\cS = (S_k)_{N_k}$ and a subset $A \subseteq X$, we denote by $\cS \cap A$ the sequence $(S_k \cap A)_{\#(S_k \cap A)}$, and likewise define $\cS \mres A$ in terms of $S_k \mres A$.
\end{notation*}

The following characterization of admissible sequences is adapted from \cite[Prop.~2.2]{Delattre-2004}:

\begin{prop}\label{prop:admisseq}
Let $\mu \in \M^p_+(X)$, $p \in [1,\infty]$. Given any sequence $(S_k)_k$ of closed subsets of $X$, the following are equivalent:
\begin{itemize}
\item[(i)] $e_p(\mu; S_k) \to 0$;
\item[(ii)] $d(\cdot, S_k) \to 0$ $\mu$-a.e.;
\item[(iii)] $d(\cdot, S_k) \to 0$ pointwise on $\supp\mu$;
\item[(iv)] $d(\cdot, S_k) \to 0$ uniformly on each $K \subseteq \supp\mu$ compact.
\end{itemize}
\end{prop}
\begin{proof}
We prove equivalence only for $p < \infty$; for $p = \infty$, observe instead that $e_\infty(\mu; S_k) \to 0$ implies $e_q(\mu; S_k) \to 0$ for $q < \infty$ arbitrary, and conversely (iv) implies $e_\infty(\mu; S_k) \to 0$ with the choice $K = \supp\mu$.
\begin{itemize}
\item (i) $\Rightarrow$ (iv): Let $K \subseteq \supp\mu$. By Ulam's theorem, we can assume wlog that $\mu(K) > 0$ by taking unions of $K$ with a compact subset of large measure. We apply the Arzel\`a-Ascoli theorem to the sequence of functions $(d(\cdot, S_k))_k$. Equicontinuity follows from the fact that each function is $1$-Lipschitz. For uniform boundedness, observe that for each $k \in \N$,
\begin{align*} 
\sup_{x \in K} d(x, S_k) & = \mu(K)^{-1/p} \sup_{x \in K} \left( \int_K d(x, S_k)^p \di \mu(y) \right)^{1/p}
\\ & \leq \mu(K)^{-1/p} \sup_{x \in K} \left( \int_K \left[ d(x,y) + d(y, S_k) \right]^p \di\mu(y) \right)^{1/p} 
\\ & \leq \mu(K)^{-1/p} \left( \diam(K) \mu(K)^{1/p} + e_p(\mu; S_k) \right),
\end{align*}
where the last inequality follows from Minkowski's inequality, and the final expression is uniformly bounded in $k$ since $e_p(\mu; S_k) \to 0$.

Thus by Arzel\`a-Ascoli, every subsequence of $(d(\cdot, S_k))_k$ must admit a further subsequence which converges uniformly on $K$, say $d(\cdot, S_{k_l}) \to f \in C^0(K)$. Then $f$ must be continuous and nonnegative, and by dominated convergence,
\[ \int_K f(x)^p \di \mu(x) = \lim_{l \to \infty} \int_K d(x, S_{k_l})^p \di \mu(x) = 0, \]
so $f$ must be identically zero on $K$. This shows that $d(\cdot, S_k)$ converges uniformly to $0$ on $K$, since every subsequence has a further subsequence converging to $0$.
\item (iv) $\Rightarrow$ (iii) $\Rightarrow$ (ii): Trivial.
\item (ii) $\Rightarrow$ (i): Take $x_0 \in X$ such that $d(x_0, S_k) \to 0$, in particular $d(x_0, S_k) \leq C$ for some constant $C > 0$. Then by the triangle inequality,
\[ d(\cdot, S_k) \leq d(x_0, S_k) + d(\cdot, x_0) \leq C + d(\cdot, x_0), \]
where the right-hand side is $L^p$-integrable with respect to $\mu$ since $\mu \in \M^p_+(X)$. Thus by dominated convergence, $e_p(\mu; S_k) = \|d(\cdot, S_k)\|_{L^p(X;\mu)} \to 0$ as well. \qedhere
\end{itemize}
\end{proof}

The same argument indeed applies for any uniformly Lipschitz family of functions in place of $(d(\cdot, S_k))_k$. This characterization implies that the admissibility of a sequence of sets does not depend on the order $p \in [1,\infty]$: the same sequence $(S_k)_{N_k}$ is admissible for any $q \in [1,\infty]$ for which $\mu \in \M^q_+(X)$.

\subsection{Quantization coefficients and $(p,s)$-quantizability}\label{sect:quantcoeff}

We quantify the rate of convergence of the quantization error to zero in terms of \emph{quantization coefficients}:

\begin{defin}[{Quantization coefficients}]
Let $X$ be a Polish metric space, $p \in [1, \infty]$, and $\mu \in \M^p_+(X)$. For $s \in (0,\infty)$, the \emph{lower} resp.~\emph{upper quantization coefficient} of $\mu$ of order $p$ and dimension $s$ is given by
\[ \underline Q_{p,s}(\mu) := \liminf_{N \to \infty} N^{1/s} e_{N,p}(\mu); \quad \overline Q_{p,s}(\mu) := \limsup_{N \to \infty} N^{1/s} e_{N,p}(\mu). \]
If the lower and upper limits coincide, the limit is called the \emph{quantization coefficient} and denoted by $Q_{p,s}(\mu)$. 
{For $p = \infty$, we define $\underline Q_{\infty,s}(A)$ and $\overline Q_{\infty,s}(A)$ likewise for $A \subseteq X$ compact.}

Given an admissible quantizing sequence $\cS = (S_k)_{N_k}$, we likewise define
\[ \underline Q_{p,s}(\mu; \cS) := \liminf_{k \to \infty} N_k^{1/s} e_p(\mu; S_k); \quad \overline Q_{p,s}(\mu; \cS) := \limsup_{k \to \infty} N_k^{1/s} e_p(\mu; S_k). \]
We say that $\cS$ is \emph{asymptotically optimal} if $Q_{p,s}(\mu; \cS) = \underline Q_{p,s}(\mu)$. 
\end{defin}

\begin{remark*}
Asymptotically optimal sequences always exist: taking a subsequence $(N_k)_k$ such that $\underline Q_{p,s}(\mu) = \lim_{k \to \infty} N_k^{1/s} e_{N_k,p}(\mu)$, for each $k$ there exists $S_k \subseteq X$ with $\# S_k \leq N_k$ such that $e_p(\mu; S_k) \leq e_{N_k,p}(\mu) + 1/k$.

Observe also that for $A \subseteq X$ and $\cS = (S_k)_{N_k}$ arbitrary,
\[ Q_{p,s}^{(\mu)}(A; \cS \mres A) = \lim_{k \to \infty} \#(S_k \mres A)^{1/s} e_p^{(\mu)}(A; S_k \mres A) \leq \lim_{k \to \infty} N_k^{1/s} e_p^{(\mu)}(A; S_k) = Q_{p,s}^{(\mu)}(A; \cS), \]
where the inequality holds both for lower and upper limits. In other words, extraneous points that do not see $A$ only serve to make the sequence of quantizers less efficient.
\end{remark*}

The properties given in Proposition \ref{prop:enpprops} (ii)--(v) carry over directly to quantization coefficients: 
{
\begin{prop}\label{prop:coeffprops}
Let $\mu \in \M^p_+(X)$, $p \in [1,\infty]$. Let $\cS = (S_k)_{N_k}$ be an admissible sequence for $\mu$. 
\begin{itemize}
\item[(i)] If $T \colon X \to Y$ is $L$-Lipschitz, then 
\[ \underline Q_{p,s}(T_\# \mu; T(\cS)) \leq L \underline Q_{p,s}(\mu; \cS), \quad \overline Q_{p,s}(T_\# \mu; T(\cS)) \leq L \overline Q_{p,s}(\mu; \cS). \]
\item[(ii)] If $\nu \in \M^p_+(X)$ with $\nu \leq \lambda \mu$, then
\[ \underline Q_{p,s}(\nu; \cS) \leq \lambda^{1/p} \underline Q_{p,s}(\mu; \cS), \quad \overline Q_{p,s}(\nu; \cS) \leq \lambda^{1/p} \overline Q_{p,s}(\mu; \cS). \]
\item[(iii)] Given $1 \leq q < p$,
\[ \underline Q_{q,s}(\mu; \cS) \leq \mu(X)^{\frac{1}{q}-\frac{1}{p}} \underline Q_{p,s}(\mu; \cS), \quad \overline Q_{q,s}(\mu; \cS) \leq \mu(X)^{\frac{1}{q}-\frac{1}{p}} \overline Q_{p,s}(\mu; \cS). \]
\end{itemize}
The same inequalities hold also for $\underline{Q}_{p,s}(\mu)$ and $\overline{Q}_{p,s}(\mu)$.
\end{prop}
\begin{proof}
Follows directly by taking the $\liminf$ resp.~$\limsup$ of the inequalities
\[ e_p(T_\# \mu; T(S_k)) \leq L e_p(\mu; S_k); \quad e_p(\nu; S_k) \leq \lambda^{1/p} e_p(\mu; S_k); \quad e_q(\mu; S_k) \leq \mu(X)^{\frac1q - \frac1p} e_p(\mu; S_k) \]
and the analogous inequalities for $e_{N,p}$.
\end{proof}
}

The additivity property in Proposition \ref{prop:enpprops} (i) will instead lead to different subadditivity and superadditivity properties, which we will present below. 

We also introduce the following notion, which captures the role that Pierce's lemma and similar integral conditions play in proofs of Zador's theorem:

\begin{defin}[{$(p,s)$-quantizability}]
Let $p \in [1,\infty)$, $s \in (0,\infty)$.
\begin{itemize}
\item A measure $\mu \in \M^p_+(X)$ is \emph{$(p,s)$-quantizable} if $\overline{Q}_{p,s}(\mu_k) \to 0$ for any sequence of measures $\mu_k \leq \mu$ such that $\mu_k(X) \to 0$.
\item A Borel set $A \subseteq X$ is \emph{$(p,s)$-quantizable} with respect to a Borel measure $\nu$ if $\nu|_A \in \M^p_+(X)$ and is $(p,s)$-quantizable.
\end{itemize}
\end{defin}

This condition is not relevant for $p = \infty$ since, for any measure $\mu$ with $\overline{Q}_{\infty,s}(\mu) > 0$, we can take $\mu_k := \mu/k$ in which case $\mu_k \to 0$ but each $\overline{Q}_{\infty,s}(\mu_k) = \overline{Q}_{\infty,s}(\mu)$. 
{Likewise, a measure $\mu$ with $\overline{Q}_{p,s}(\mu) = \infty$ is trivially not $(p,s)$-quantizable.} 
In the case $p < \infty$, $(p,s)$-quantizability will yield many stronger properties such as the stability of quantization coefficients along monotonically increasing sequences of measures, summarized in Proposition \ref{prop:qpscont}, which we will employ in Sections \ref{sect:densbds} and \ref{sect:zadorrect}. 

{
From Proposition \ref{prop:coeffprops} we deduce the following basic properties of $(p,s)$-quantizability:
}

\begin{prop}[{Properties of $(p,s)$-quantizability}]\label{prop:psstabprops}
Let $\mu \in \M^p_+(X)$, $p \in [1,\infty)$, $s \in (0,\infty)$.
\begin{itemize}
\item[(i)] If $\mu \in \M^q_+(X)$ for $p < q \leq \infty$ and $\overline{Q}_{q,s}(\mu) < \infty$, then $\mu$ is $(p,s)$-quantizable.
\item[(ii)] If $\mu$ is $(p,s)$-quantizable and $\nu \leq \lambda \mu$, then $\nu$ is also $(p,s)$-quantizable. 
\item[(iii)] If $\mu$ is $(p,s)$-quantizable and $T \colon X \to Y$ is a Lipschitz map, $T_\# \mu$ is also $(p,s)$-quantizable.
\end{itemize}
\end{prop}
\begin{proof}\
\begin{itemize}
\item[(i)] Suppose $\mu_k \to 0$, $\mu_k \leq \mu$. {Then by Proposition \ref{prop:coeffprops} (iii), } 
\[ \overline{Q}_{p,s}(\mu_k) \leq \mu_k(X)^{\frac{1}{p} - \frac{1}{q}} \overline{Q}_{q,s}(\mu_k) \leq \mu_k(X)^{\frac{1}{p} - \frac{1}{q}} \overline{Q}_{q,s}(\mu). \]
If $\overline{Q}_{q,s}(\mu) < \infty$, then $\mu_k(X) \to 0$ implies that $\overline{Q}_{p,s}(\mu_k) \to 0$.
\item[(ii)] Given any sequence of measures $\nu_k \leq \nu \leq \lambda \mu$ such that $\nu_k(X) \to 0$, the measures $\mu_k := \lambda^{-1} \nu_k \leq \mu$ also converge to $0$. Proposition \ref{prop:coeffprops} (ii) and the $(p,s)$-quantizability of $\mu$ then implies 
\[ \overline Q_{p,s}(\nu_k) \leq \lambda^{1/p} \overline Q_{p,s}(\mu_k) \to 0. \]
\item[(iii)] Set $\nu := T_\# \mu$, and let $\nu_k \to 0$ with $\nu_k \leq \nu$. 
{Let $\rho_k \in L^1(Y; \nu)$ be the Radon-Nikodym derivative of $\nu_k$ with respect to $\nu$. Then each $\rho_k \leq 1$ $\nu$-a.e., which implies that $\rho_k \circ T \leq 1$ $\mu$-a.e.~since $\{\rho_k \circ T > 1\} = T^{-1}(\{\rho_k > 1\})$. Setting $\mu_k := (\rho_k \circ T) \mu$, we then have that each $\mu_k \leq \mu$, 
and moreover} $\nu_k = T_\# \mu_k$: indeed, given an arbitrary bounded continuous function $f \colon Y \to \R$,
\begin{align*} 
\int_Y f(y) \di(T_\# \mu_k)(y) & = \int_X f(T(x)) \di\mu_k(x) = \int_X f(T(x)) \rho_k(T(x)) \di\mu(x) 
\\ & = \int_X f(y) \rho_k(y) \di(T_\# \mu)(y) = \int_Y f(y) \di\nu_k(y). 
\end{align*}
Consequently, $\mu_k(X) = \nu_k(Y) \to 0$, and by the assumption that $\mu$ is $(p,s)$-quantizable,
\[ \overline{Q}_{p,s}(\nu_k) = \overline{Q}_{p,s}(T_\# \mu_k) \leq \Lip(T) \overline{Q}_{p,s}(\mu_k) \to 0. \]
Thus $\overline{Q}_{p,s}(\nu_k) \to 0$ as well. \qedhere
\end{itemize}
\end{proof}

In particular, $1$-Lipschitz inclusion maps $\iota \colon X \to Y$ preserve the property of $(p,s)$-quantizability, such as in the case of Riemannian manifolds embedded into larger Euclidean spaces. Hence in order to demonstrate that a measure supported on a {lower-dimensional submanifold} of $\R^d$ is $(p,s)$-quantizable {as a measure on $\R^d$}, it suffices to show that the measure is $(p,s)$-quantizable {as a measure on the Riemannian manifold}.

We give a list of sufficient conditions for $(p,s)$-quantizability on various domains in Section \ref{sect:psstab}. For now, we note that $(p,s)$-quantizability is guaranteed by the existence of Pierce-type integral upper bounds on quantization coefficients:

\begin{prop}\label{prop:intubstab}
Let $p \in [1,\infty)$, $s \in (0,\infty)$, $\mu \in \M^p_+(X)$. Suppose there exists a nonnegative measurable function $F \colon X \to [0,\infty]$ such that $\int F \di \mu < \infty$ and the following integral upper bound
\[ \overline{Q}_{p,s}(\nu)^p \leq \int_X F \di \nu \]
holds for all Borel measures $\nu \leq \mu$. Then $\mu$ is $(p,s)$-quantizable.
\end{prop}
\begin{proof}
Let $\mu_n \to 0$ with $\mu_n \leq \mu$, and consider the Radon-Nikodym derivatives $\rho_n := \frac{\di\mu_n}{\di\mu}$. Then each $0 \leq \rho_n \leq 1$ $\mu$-a.e., with $\rho_n \to 0$ in $L^1(X; \mu)$ thus in $\mu$-measure. By assumption, we have that
\[ 0 \leq \overline{Q}_{p,s}(\mu_n)^p \leq \int_X F \di \mu_n = \int_X F \rho_n \di \mu. \]
Since $0 \leq F \rho_n \leq F \in L^1(X; \mu)$, and $\rho_n \to 0$ in $\mu$-measure (thus $\mu$-a.e.~along a subsequence), the right-hand side converges to $0$ by dominated convergence. Thus $\overline{Q}_{p,s}(\mu_n) \to 0$ as well.
\end{proof}

\begin{remark*}
The above statement is also applicable for bounds of the form $\overline{Q}_{p,s}(\nu)^q \leq \int_X F \di \nu$ with $q \geq p$, since we only require the bound to hold for measures $\nu \leq \mu$.
\end{remark*}

In particular, by Pierce's lemma, every measure $\mu \in \M^{p+\delta}_+(\R^d)$ is $(p,d)$-quantizable with the choice $F(x) = C (1 + \|x\|^{p+\delta})$. 
More generally, we deduce in Proposition \ref{prop:covstab} that the covering growth condition given in \cite{covgrow} implies $(p,s)$-quantizability.

It is then natural to ask what kinds of domains admit such integral upper bounds and how the growth of the domain affects the integrand. To this end the random quantizer estimate given in Theorem \ref{thm:randstab} and the subsequent Corollary \ref{cor:volgrowstab} provide general sufficient conditions at the level of metric measure spaces, which though less precise than the covering growth upper bound do not require the space to be geodesically connected.

\subsection{Additivity of quantization coefficients}\label{sect:addquant}
We finally list a few general properties of quantization coefficients whose proofs are given in Appendix \ref{app:quantcoeffs}.
In Appendix \ref{app:subadd}, we prove the following subadditivity statement, which generalizes \cite[Lem.~2.5]{covgrow}:

\begin{prop}[Finite subadditivity]\label{prop:subadd}
Let $p \in [1,\infty]$, $\mu_1, \mu_2 \in \M^p_+(X)$. Set $\mu := \mu_1 + \mu_2$, and let $s \in (0,\infty)$. 

Then for $\frac{1}{p^\prime} := \frac{1}{p} + \frac{1}{s}$,
\begin{align*}
\overline Q_{p,s}(\mu)^{p^\prime} \leq \overline Q_{p,s}(\mu_1)^{p^\prime} + \overline Q_{p,s}(\mu_2)^{p^\prime}; \quad
\underline Q_{p,s}(\mu)^{p^\prime} \leq \underline Q_{p,s}(\mu_1)^{p^\prime} + \overline Q_{p,s}(\mu_2)^{p^\prime}.
\end{align*}
\end{prop}

The proof employs the same basic argument as in proofs of Zador's theorem, e.g.~\cite[Lem.~6.5]{quantbook} and \cite[Lem.~5.1]{discapprox}, quantizing each $\mu_i$ by $N_i$ points according to fixed proportions $\frac{N_i}{N_1+N_2}$ chosen optimally. 

As a consequence, we deduce the following stability statement:

\begin{prop}[Stability of quantization coefficients]\label{prop:qpscont}
Let $p \in [1,\infty)$, $s \in (0,\infty)$. 

Let $\mu \in \M^p_+(X)$ be $(p,s)$-quantizable, and let $(\mu_k)_k$ be a sequence of measures converging to $\mu$ from below; i.e.~each $\mu_k \leq \mu$ and $(\mu - \mu_k)(X) \to 0$. Then
\[ \underline Q_{p,s}(\mu) = \lim_{k \to \infty} \underline Q_{p,s}(\mu_k) \quad \text{and} \quad \overline Q_{p,s}(\mu) = \lim_{k \to \infty} \overline Q_{p,s}(\mu_k). \]
In particular, for any $A \subseteq X$ Borel,
\[ \underline{Q}_{p,s}^{(\mu)}(A) = \sup_{\substack{K \subseteq A \text{ compact}}} \underline{Q}_{p,s}^{(\mu)}(K); \quad \overline{Q}_{p,s}^{(\mu)}(A) = \sup_{\substack{K \subseteq A \text{ compact}}} \overline{Q}_{p,s}^{(\mu)}(K). \]
\end{prop}
\begin{proof}
Proposition \ref{prop:subadd} implies that
\[ \underline Q_{p,s}(\mu_k)^{p^\prime} \leq \underline Q_{p,s}(\mu)^{p^\prime} \leq \underline Q_{p,s}(\mu_k)^{p^\prime} + \overline Q_{p,s}(\mu - \mu_k)^{p^\prime}, \]
and likewise for the upper quantization coefficients. By the assumption of $(p,s)$-quantizability, we have $\overline Q_{p,s}(\mu - \mu_k) \to 0$, from which the first statement follows.

For the second statement, for each $k \in \N$, Ulam's theorem yields a compact subset $K_k \subseteq A$ such that $\mu(A \setminus K_k) < \frac{1}{k}$; setting $\mu_k := \mu|_{K_k}$ with $\mu|_A$ in place of $\mu$ then yields the second statement.
\end{proof}

In particular, if $\mu = \rho \nu$ for some (possibly infinite) Borel measure $\nu$, the statement holds for any $\mu_k = \rho_k \nu$ such that $\rho_k \to \rho$ $\nu$-a.e.~from below. The statement also follows when $(\rho_k)_k$ is not necessarily dominated by $\rho$, but by another $(p,s)$-quantizable density $\sigma \geq \rho$, but the above assumptions will be sufficient for our purposes.

{Proposition \ref{prop:subadd} also implies that the sum of finitely many $(p,s)$-quantizable measures is also $(p,s)$-quantizable; this however does not generalize to countable sums. }

In Appendix \ref{app:supadd}, we obtain the following converse superadditivity statement to Proposition \ref{prop:subadd} for measures with disjoint compact support, which also characterizes the limiting spatial distribution of quantizers:

\begin{prop}[Finite superadditivity, compact case] \label{prop:compsupadd}
Let $p \in [1,\infty]$, $\mu \in \M^p_+(X)$, $s \in (0,\infty)$. Set $\frac{1}{p^\prime} := \frac{1}{p} + \frac{1}{s}$.

Let $K_1, \ldots, K_m \subseteq X$ be disjoint compact sets. Let $\cS = (S_k)_{N_k}$ be an asymptotically optimal sequence of quantizers for $\mu$, and consider the sequence of proportions $(v_k)_k \subset [0,1]^m$ given by
\[ v_{k,i} := \frac{\#(S_k \mres K_i)}{N_k}, \quad i = 1, \ldots, m; \ k \in \N. \]
Then for any limit point $\bar v \in [0,1]^m$ of the sequence $(v_k)_k$, the following inequalities hold:
\begin{align*} 
\underline{Q}_{p,s}(\mu)^p & \geq \sum_{i=1}^m \bar v_i^{-p/s} \underline{Q}_{p,s}^{(\mu)}(K_i; \cS \mres K_i)^p \geq \left( \sum_{i=1}^m \underline{Q}_{p,s}^{(\mu)}(K_i)^{p^\prime} \right)^{p/p^\prime} \quad \text{for} \quad p < \infty, \\ 
\underline{Q}_{\infty,s}(\mu) & \geq \max_{i=1}^m \bar v_i^{-1/s} \underline{Q}_{\infty,s}^{(\mu)}(K_i; \cS \mres K_i) \geq \left( \sum_{i=1}^m \underline{Q}_{\infty,s}^{(\mu)}(K_i)^s \right)^{1/s} \quad \text{for} \quad p = \infty,
\end{align*}
where by convention {$\bar v_i^{-1/s} \underline{Q}_{p,s}^{(\mu)}(K_i; \cS \mres K_i) = 0$ whenever $\underline{Q}_{p,s}^{(\mu)}(K_i; \cS \mres K_i) = 0$, even if $\bar v_i = 0$}.
\end{prop}

In combination with Proposition \ref{prop:subadd}, we obtain equality and determine the unique limit point $\bar v$:

\begin{theorem}[Limiting spatial distribution] \label{thm:stats}
Let $p \in [1,\infty)$, $s \in (0,\infty)$, and set $\frac{1}{p^\prime} := \frac{1}{p} + \frac{1}{s}$. 

Let $\mu \in \M^p_+(X)$ be $(p,s)$-quantizable with $0 < \underline Q_{p,s}(\mu) < \infty$. Let moreover $A \subseteq X$ Borel such that $Q^{(\mu)}_{p,s}(A)$ exists and $Q^{(\mu)}_{p,s}(\partial A) = 0$. Then 
\[ \underline Q_{p,s}(\mu)^{p^\prime} = Q_{p,s}^{(\mu)}(A)^{p^\prime} + \underline Q_{p,s}^{(\mu)}(A^c)^{p^\prime}, \]
and for any asymptotically optimal sequence $\cS = (S_k)_{N_k}$ for $\mu$, we have
\[ \lim_{k \to \infty} \frac{\# (S_k \cap A)}{N_k} = \frac{Q^{(\mu)}_{p,s}(A)^{p^\prime}}{\underline Q_{p,s}(\mu)^{p^\prime}}; \quad \lim_{k \to \infty} \frac{\# (S_k \cap A^c)}{N_k} = \frac{\underline Q^{(\mu)}_{p,s}(A^c)^{p^\prime}}{\underline Q_{p,s}(\mu)^{p^\prime}}. \]
\end{theorem}

We also prove this statement in Appendix \ref{app:supadd}.
This statement generalizes the spatial distribution result of Graf and Luschgy \cite[Thm.~7.5]{quantbook} to any setting in which Zador's theorem is shown to hold, such as Riemannian manifolds under covering growth assumptions, or rectifiable measures which we will discuss below. The theorem also imposes no requirement for $\mu$ to be absolutely continuous, unlike \cite[Thm.~7.5]{quantbook}.

In the more general case of nonexistent quantization coefficients, we deduce two more statements which we will apply in the following sections:

\begin{prop}\label{prop:cntadd}
Let $p \in [1,\infty]$, $\mu \in \M^p_+(X)$, $s \in (0,\infty)$, $p^\prime := \frac{sp}{s+p}$. Let $\nu$ be a Borel measure on $X$, $(A_k)_{k=1}^\infty$ be a countable collection of disjoint Borel subsets of $X$, and set $A = \bigsqcup_{k=1}^\infty A_k$.
\begin{itemize}
\item[(i)] Suppose each $A_k$ satisfies $\underline{Q}_{p,s}^{(\mu)}(A_k)^{p^\prime} \geq \nu(A_k)$. 
Then likewise $\underline{Q}_{p,s}^{(\mu)}(A)^{p^\prime} \geq \nu(A)$.
\item[(ii)] Suppose each $A_k$ satisfies $\overline{Q}_{p,s}^{(\mu)}(A_k)^{p^\prime} \leq \nu(A_k)$, and $A$ is $(p,s)$-quantizable with respect to $\mu$ (hence necessarily $p < \infty$). 
Then likewise $\overline{Q}_{p,s}^{(\mu)}(A)^{p^\prime} \leq \nu(A)$. 
\end{itemize}
\end{prop}
\begin{proof}
We can assume wlog that each $\nu(A_k) < \infty$; otherwise, both statements hold trivially.
\begin{itemize}
\item Fix $\veps > 0$. For each $k \in \N$, since $\nu(A_k) < \infty$, we can take $K_k \subseteq A_k$ compact such that $\nu(A_k \setminus K_k) < 2^{-k} \veps$. Given $m \in \N$, applying Proposition \ref{prop:compsupadd} to $\mu|_A$ and the disjoint compact sets $(K_i)_{i=1}^m$, we obtain
\[ \underline{Q}_{p,s}^{(\mu)}(A)^{p^\prime} \geq \sum_{k=1}^m \underline{Q}_{p,s}^{(\mu)}(K_k)^{p^\prime} \geq \sum_{k=1}^m \nu(K_k) \geq \sum_{k=1}^m \nu(A_k) - \veps \sum_{k=1}^\infty 2^{-k} \geq \nu\left( \bigsqcup_{k=1}^m A_k \right) - \veps. \]
Letting $m \to \infty$ and then $\veps \to 0$ yields the desired inequality.
\item 
Given $m \in \N$, considering the finite union $B_m := \bigsqcup_{k=1}^m A_k$, we can apply Proposition \ref{prop:subadd} to obtain
\[ \overline Q_{p,s}^{(\mu)}(A)^{p^\prime} \leq \sum_{k=1}^m \overline Q_{p,s}^{(\mu)}(A_k)^{p^\prime} + \overline Q_{p,s}^{(\mu)}(A \setminus B_m)^{p^\prime} \leq \sum_{k=1}^\infty \nu(A_k) + \overline Q_{p,s}^{(\mu)}(A \setminus B_m)^{p^\prime}. \]
The statement then follows from the assumption of $(p,s)$-quantizability, noting that the measures $\mu|_{A \setminus B_m}$ are dominated by $\mu|_A$ and converge monotonically to $0$ as $m \to \infty$. \qedhere
\end{itemize}
\end{proof}

\begin{prop}\label{prop:genzadorineq}
 Let $p \in [1,\infty]$, $s \in (0,\infty)$, $p^\prime := \frac{sp}{s+p}$. Let $\nu$ be a Borel measure on $X$, $\mu = \rho\nu \in \M_+^p(X)$, $f, g \colon X \to [0,\infty]$ Borel. 
\begin{itemize}
\item[(i)] Suppose $\underline{Q}_{p,s}^{(\nu)}(A)^{p^\prime} \geq \int_A f \di \nu$ for arbitrary $A \subseteq X$ Borel such that $\nu|_A \in \M^p_+(X)$. 
Then
\[ \underline{Q}_{p,s}(\mu)^{p^\prime} \geq \int_X \rho^{\frac{s}{s+p}} f \di \nu. \]
\item[(ii)] Suppose $\overline{Q}_{p,s}^{(\nu)}(A)^{p^\prime} \leq \int_A g \di \nu$ for arbitrary $A \subseteq X$ Borel and $(p,s)$-quantizable with respect to $\nu$, and $\mu$ is also $(p,s)$-quantizable.
Then
\[ \overline{Q}_{p,s}(\mu)^{p^\prime} \leq \int_X \rho^{\frac{s}{s+p}} g \di \nu. \]
\end{itemize}
\end{prop}
\begin{proof}
First suppose $\rho$ is a simple function, say $\rho = \sum_{k=1}^\infty \lambda_k \One_{A_k}$ for $(A_k)_{k=1}^\infty$ disjoint Borel subsets of $X$ and $(\lambda_k)_{k=1}^\infty \subset (0,\infty)$. Note that each $\nu|_{A_k} \in \M^p_+(X)$ since $\nu|_{A_k} \leq \lambda_k^{-1} \mu$, and if $\mu$ is $(p,s)$-quantizable, so is each $\nu|_{A_k}$. Assumption (i) then implies
\[ \underline{Q}_{p,s}^{(\mu)}(A_k)^{p^\prime} = \lambda_k^{\frac{p^\prime}{p}} \underline{Q}_{p,s}^{(\nu)}(A_k)^{p^\prime} \geq \lambda_k^{\frac{p^\prime}{p}} \int_{A_k} f \di \nu = \int_{A_k} \rho^{\frac{s}{s+p}} f \di \nu, \]
and likewise for (ii). Applying Proposition \ref{prop:cntadd} (i) with the measure $A \mapsto \int_A \rho^{\frac{s}{s+p}} f \di \nu$, we obtain
\[ \underline{Q}_{p,s}(\mu)^{p^\prime} = \underline{Q}_{p,s}^{(\mu)}\left(\bigsqcup_{k=1}^\infty A_k \right)^{p^\prime} \geq \int_{\bigsqcup_{k=1}^\infty A_k} \rho^{\frac{s}{s+p}} f \di \nu = \int_X \rho^{\frac{s}{s+p}} f \di \nu, \]
and likewise for (ii). The statement for general $\rho \in L^1(X; \nu)$ follows likewise by approximating $\rho$ from below by simple functions, invoking Proposition \ref{prop:qpscont} for (ii).
\end{proof}

Indeed, approximating Borel sets from below by compact sets, it suffices for the assumptions to hold only for $A$ compact.

\section{Density estimates on metric measure spaces}\label{sect:densbds}

We now consider a given metric measure space $(X, \nu)$ where $\nu$ is a fixed locally finite measure on $X$. 
We first note the following measure theoretic observations, which ensure that the statement of Theorem \ref{thm:densbd} is sensible:

\begin{remark*}
Any locally finite measure $\nu$ on a separable metric space $X$ is $\sigma$-finite: for each $x \in X$, there exists $r_x > 0$ such that $\nu(B_{r_x}(x)) < \infty$, and the open cover $\{B_{r_x}(x)\}_{x \in X}$ admits a countable subcover. Every measure $\mu$ on $X$ can thus be decomposed as $\mu = \rho \nu + \mu^\bot$ with $\mu^\bot \bot \nu$.
\end{remark*}

\begin{remark*}[Hausdorff densities are measurable]
For fixed $\delta > 0$, $\underline{\vartheta}^{(\nu)}_s(\cdot, \delta)$ is upper semi-continuous: given $x_n \to x$, for arbitrary $r < \delta$ and $\alpha > 1$, we have that $B_r(x) \supseteq B_{\alpha^{-1} r}(x_n)$ for $n$ sufficiently large, hence
\[ \frac{\nu(B_r(x))}{\omega_s r^s} \geq \frac{\nu(B_{\alpha^{-1} r}(x_n))}{\omega_s r^s} \geq \alpha^{-s} \frac{\nu(B_{\alpha^{-1} r}(x_n))}{\omega_s (\alpha^{-1} r)^s} \geq \alpha^{-s} \underline{\vartheta}^{(\nu)}_s(x_n, \delta), \]
and taking the $\limsup$ of the right-hand side, minimizing the left-hand side over $r$ and letting $\alpha \to 1^-$ yields that $\underline{\vartheta}^{(\nu)}_s(x, \delta) \geq \limsup_{n\to\infty} \underline{\vartheta}^{(\nu)}_s(x_n, \delta)$. 

The lower semi-continuity of $\overline{\vartheta}^{(\nu)}_s(\cdot, \delta)$ follows analogously by applying instead the inclusion $B_r(x) \subseteq B_{\alpha r}(x_n)$. Since the family of Borel functions is closed under monotone limits, it follows that both Hausdorff densities and their approximate versions are Borel functions.
\end{remark*}

\begin{remark*}
If $\underline{\vartheta}_s^{(\nu)}(x) > 0$, then $\underline{\vartheta}_s^{(\nu)}(x, \delta) > 0$ for all $\delta > 0$. Taking a minimizing sequence $(r_k)_k$ such that $\underline{\vartheta}_s^{(\nu)}(x, \delta) = \lim_{k \to \infty} \frac{\nu(B_{r_k}(x))}{\omega_s r_k^s}$, either 
\begin{itemize}
\item each $r_k \geq \delta_0 > 0$: then $\underline{\vartheta}_s^{(\nu)}(x, \delta) = \lim_{k \to \infty} \frac{\nu(B_{r_k}(x))}{\omega_s r_k^s} \geq \frac{\nu(B_{\delta_0}(x))}{\omega_s \delta^s} > 0$ since $x \in \supp\nu$;
\item $r_k \to 0$ along a subsequence: then for arbitrary $\delta^\prime < \delta$, eventually $r_k < \delta^\prime$ hence $\frac{\nu(B_{r_k}(x))}{\omega_s r_k^s} \geq \underline{\vartheta}_s^{(\nu)}(x, \delta^\prime)$, which implies that $\underline{\vartheta}_s^{(\nu)}(x, \delta) \geq \underline{\vartheta}_s^{(\nu)}(x, \delta^\prime)$. Letting $\delta^\prime \searrow 0$ yields $\underline{\vartheta}_s^{(\nu)}(x, \delta) \geq \underline{\vartheta}_s^{(\nu)}(x) > 0$.
\end{itemize}
\end{remark*}

We now proceed with the proof of Theorem \ref{thm:densbd}, first for measures of the form $\mu := \nu|_A$ and then for general measures. 

The proof starts from preexisting estimates, presented in Appendix \ref{app:conc}, which control the quantization errors of sets on which concentration inequalities of the form $\nu(B_r(x)) > \alpha r^s$ or $\nu(B_r(x)) < \beta r^s$ hold, i.e.~sets of the form $\{ \underline\vartheta_s(x, \delta) > \alpha \}$ resp.~$\{ \overline\vartheta_s(x, \delta) < \beta \}$. 

These estimates are then refined further and further using the general properties given in Appendix \ref{app:quantcoeffs}, first by letting $\delta \to 0$, then proving the theorem for measures of the form $\mu = \nu|_A$ by partitioning $A$ into countably many subsets on which $\underline\vartheta_s$ resp.~$\overline\vartheta_s$ is approximately constant, then generalizing to $\mu \ll \nu$ using Proposition \ref{prop:genzadorineq}.

\subsection{Estimates for uniformly bounded densities}\label{sect:unifdens}

We first cover the case when the Hausdorff densities of $\nu$ are uniformly bounded on $A$. 
The following statement is obtained from the concentration inequalities given in Lemmas \ref{lem:concub} and \ref{lem:conclb}:

\begin{lemma}\label{lem:densbdd}
Let $p \in [1,\infty)$, $s \in (0,\infty)$. Set $p^\prime := \frac{sp}{s+p}$, $C_1 := 2^{-p^\prime} \omega_s^{-\frac{p}{s+p}} \left( \frac{s}{s+p} \right)^{\frac{s}{s+p}}$ and $C_2 := 2^{p^\prime} \omega_s^{-\frac{p}{s+p}}$.

Let $A \subseteq X$ be Borel such that $\nu|_A \in \M^p_+(X)$ (i.e.~$\int_A d(\cdot, x_0)^p \di \nu < \infty$ for some $x_0 \in X$), $\vartheta \in (0,\infty)$.
\begin{itemize}
\item[(i)] If $\overline\vartheta_s(x) < \vartheta$ for $\nu$-a.e.~$x \in A$, then
\[ \underline{Q}^{(\nu)}_{p,s}(A)^{p^\prime} \geq C_1 \vartheta^{-\frac{p}{s+p}} \nu(A). \]
\item[(ii)] If $\underline\vartheta_s(x) > \vartheta$ for $\nu$-a.e.~$x \in A$, and $A$ is $(p,s)$-quantizable with respect to $\nu$, then
\[ \overline{Q}^{(\nu)}_{p,s}(A)^{p^\prime} \leq C_2 \vartheta^{-\frac{p}{s+p}} \nu(A). \]
\end{itemize}
\end{lemma}
\begin{proof}
Since subsets of zero $\nu$-measure do not affect either side of the above inequalities, we can assume wlog that the given inequalities hold for all $x \in A$.

We first prove the statements under the assumption that the density inequalities hold with $\vartheta_s(x, \delta)$ in place of $\vartheta_s(x)$. For the lower bound, note that if $\overline\vartheta_s(\cdot, \delta) < \vartheta$ on $A$, this implies that 
\[ \nu(A \cap B_r(y)) \leq 2^s \omega_s \vartheta r^s \quad \text{for all } y \in X,\ r < \frac{\delta}{2}. \]
Indeed, either $d(y, A) \geq r$, in which case $\nu(A \cap B_r(y)) = \nu(\emptyset) = 0$, or there exists $x \in A$ such that $d(x,y) < r$, in which case $\nu(A \cap B_r(y)) \leq \nu(B_{2r}(x)) \leq \vartheta \omega_s (2r)^s$.
Lemma \ref{lem:conclb} then implies that
\[ \underline{Q}^{(\nu)}_{p,s}(A)^{p^\prime} \geq \left( \frac{s}{s+p} \right)^{\frac{s}{s+p}} (2^s \omega_s \vartheta)^{-\frac{p}{s+p}} \nu(A) = C_1 \vartheta^{-\frac{p}{s+p}} \nu(A). \]
For the upper bound, by Proposition \ref{prop:qpscont} we can assume wlog that $A$ is compact. If $\overline\vartheta_s(\cdot, \delta) \leq \vartheta$ on $A$, $\nu(B_r(x)) \geq \vartheta \omega_s r^s$ for all $x \in A, r < \delta$. 
Lemma \ref{lem:concub} then implies
\[ \overline{Q}^{(\nu)}_{p,s}(A)^{p^\prime} \leq 2^{\frac{sp}{s+p}} \omega_s^{-\frac{p}{s+p}} \vartheta^{-\frac{p}{s+p}} \nu(A) = C_2 \vartheta^{-\frac{p}{s+p}} \nu(A). \]

Now for the general case, for arbitrary $\delta > 0$ we can set 
\[ \overline A_\delta := \{ x \in A \mid \overline\vartheta_s(x, \delta) > \vartheta \}, \quad \underline A_\delta := \{ x \in A \mid \underline\vartheta_s(x, \delta) < \vartheta \}, \]
so that $A = \bigcup_{\delta > 0} \overline A_\delta = \bigcup_{\delta > 0} \underline A_\delta$ as a monotone limit, with
\begin{align*}
 \underline{Q}^{(\nu)}_{p,s}(A)^{p^\prime} \geq \underline{Q}^{(\nu)}_{p,s}(\overline A_\delta)^{p^\prime} & \geq C_1 \vartheta^{-\frac{p}{s+p}} \nu(\overline A_\delta); \\
 \overline{Q}^{(\nu)}_{p,s}(\underline A_\delta)^{p^\prime} & \leq C_2 \vartheta^{-\frac{p}{s+p}} \nu(\underline A_\delta) \leq C_2 \vartheta^{-\frac{p}{s+p}} \nu(A). 
\end{align*}
We can then take $\delta \to 0$, apply $\nu(\overline A_\delta) \nearrow \nu(A)$ in the first inequality, and invoke Proposition \ref{prop:qpscont} under the assumption of $(p,s)$-quantizability in the second inequality, obtaining the desired statements.
\end{proof}

We also deduce the following edge cases:

\begin{lemma}\label{lem:a0inf}
Let $p \in [1,\infty)$ and $A \subseteq X$ be Borel with $\nu|_A \in \M^p_+(X)$.
\begin{itemize}
\item[(i)] Suppose $\nu(A) > 0$ and $\overline\vartheta_s(x) = 0$ for $\nu$-a.e.~$x \in A$. 
Then $\underline{Q}_{p,s}^{(\nu)}(A) = \infty$.
\item[(ii)] Suppose $A$ is $(p,s)$-quantizable with respect to $\nu$ and $\underline\vartheta_s(x) = \infty$ for $\nu$-a.e.~$x \in A$. 
Then $\overline{Q}_{p,s}^{(\nu)}(A) = 0$.
\end{itemize}
\end{lemma}
\begin{proof}
Let $\vartheta \in (0,\infty)$ be arbitrary. 
\begin{itemize}
\item[(i)] We have that $\overline\vartheta_s(x) < \vartheta$ for $\nu$-a.e.~$x \in A$, hence by Lemma \ref{lem:densbdd} (i), $\underline{Q}^{(\nu)}_{p,s}(A)^{p^\prime} \geq C_1 \vartheta^{-\frac{p}{s+p}} \nu(A)$.

Since $\nu(A) > 0$, letting $\vartheta \to 0^+$ then shows that $\underline{Q}_{p,s}^{(\nu)}(A) = \infty$.
\item[(ii)] We have that $\overline\vartheta_s(x) > \vartheta$ for $\nu$-a.e.~$x \in A$, hence by Lemma \ref{lem:densbdd} (ii), $\overline{Q}^{(\nu)}_{p,s}(A)^{p^\prime} \leq C_1 \vartheta^{-\frac{p}{s+p}} \nu(A)$.

Since $\nu(A) < \infty$, letting $\vartheta \to \infty$ then shows that $\overline{Q}_{p,s}^{(\nu)}(A) = 0$. \qedhere
\end{itemize}
\end{proof}

{The cases (i) and (ii) correspond to sets of dimension larger resp.~smaller than $s$, and also imply lower resp.~upper bounds on the quantization dimension; this is made precise in Corollary \ref{cor:quantdims}.}

\subsection{Proof of Theorem \ref{thm:densbd}}\label{sect:densbdproof}

Using the above inequalities and the additivity properties presented in Section \ref{sect:addquant}, we first prove Theorem \ref{thm:densbd} for the case $\mu = \nu|_A$.

\begin{lemma}\label{lem:nua}
Let $A \subseteq X$ be Borel with $\nu|_A \in \M^p_+(X)$.
Then for each $p \in [1,\infty]$,
\[ \underline{Q}_{p,s}^{(\nu)}(A)^{p^\prime} \geq C_1 \int_A \overline{\vartheta}_s(x)^{-\frac{p}{s+p}} \di\nu(x), \]
and if $A$ is $(p,s)$-quantizable with respect to $\nu$,
\[ \overline{Q}_{p,s}^{(\nu)}(A)^{p^\prime} \leq C_2 \int_A \underline{\vartheta}_s(x)^{-\frac{p}{s+p}} \di\nu(x). \]
\end{lemma}
\begin{proof}
By Lemma \ref{lem:a0inf}, we can assume wlog that $0 < \underline{\vartheta}_s(x) \leq \overline{\vartheta}_s(x) < \infty$ for all $x \in A$. Indeed, if the set $\overline A_0 := \{x \in A \mid \overline{\vartheta}_s(x) = 0\}$ has nonzero measure, then both integrals on the right-hand side are infinite, and we also have $\underline{Q}_{p,s}^{(\nu)}(A) \geq \underline{Q}_{p,s}^{(\nu)}(\overline A_0) = \infty$.
Likewise, setting $\underline A_\infty := \{x \in A \mid \underline{\vartheta}_s(x) = \infty\}$, $\overline{Q}_{p,s}^{(\nu)}(\underline A_\infty) = 0$ so $\overline{Q}_{p,s}^{(\nu)}(A) = \overline{Q}_{p,s}^{(\nu)}(A \setminus \underline A_\infty)$ by Proposition \ref{prop:subadd}. 

We first prove the lower bound. Fix $\alpha > 1$ and consider the measurable sets
\[ \overline A_l := \{ x \in A \mid \alpha^{l-1} \leq \overline\vartheta_s(x) < \alpha^l \}, \quad l \in \Z, \]
so that $A = \bigsqcup_{l \in \Z} \overline A_l$. Applying Lemma \ref{lem:densbdd} (i) with $\vartheta = \alpha^l$ and using that $\overline\vartheta_s \geq \alpha^{l-1}$ on $\overline A_l$, we obtain
\[ \underline{Q}^{(\nu)}_{p,s}(\overline A_l)^{p^\prime} \geq C_1 (\alpha^l)^{-\frac{p}{s+p}} \nu(\overline A_l) = C_1 \alpha^{-\frac{p}{s+p}} (\alpha^{l-1})^{-\frac{p}{s+p}} \nu(\overline A_l) \geq C_1 \alpha^{-\frac{p}{s+p}} \int_{\overline A_l} \overline{\vartheta}_s(x)^{-\frac{p}{s+p}} \di\nu(x). \]
Applying Proposition \ref{prop:cntadd} to the countable disjoint union $A = \bigsqcup_{l \in \Z} \overline A_l$ yields
\[ \underline{Q}^{(\nu)}_{p,s}(A)^{p^\prime} \geq C_1 \alpha^{-\frac{p}{s+p}} \int_{A} \overline{\vartheta}_s(x)^{-\frac{p}{s+p}} \di\nu(x). \]
Letting $\alpha \to 1^-$ yields the lower bound.
Now for the upper bound, again fix $\alpha > 1$ and similarly define
\[ \underline A_l := \{ x \in A \mid \alpha^{l-1} < \underline\vartheta_s(x) \leq \alpha^l \}, \quad l \in \Z, \]
so that $A = \bigsqcup_{l \in \Z} \underline A_l$. By Lemma \ref{lem:densbdd} (ii), we have
\[ \overline{Q}_{p,s}^{(\nu)}(\underline A_l)^{\frac{sp}{s+p}} \leq C_1 (\alpha^{l-1})^{-\frac{p}{s+p}} \nu(\underline A_l) \leq C_1 \alpha^{\frac{p}{s+p}} \int_{\underline A_l} \underline{\vartheta}_s(x)^{-\frac{p}{s+p}} \di\nu(x). \]
Then Proposition \ref{prop:cntadd} and the $(p,s)$-quantizability of $A$ likewise yields
\[ \overline{Q}_{p,s}^{(\nu)}(A)^{\frac{sp}{s+p}} \leq C_1 \alpha^{\frac{p}{s+p}} \int_A \underline{\vartheta}_s(x)^{-\frac{p}{s+p}} \di\nu(x), \]
which gives the desired upper bound after letting $\alpha \to 1^-$.
\end{proof}

This proves Theorem \ref{thm:densbd} for $\mu := \nu|_A$; we now prove the general case.

\begin{proof}[Proof of Theorem \ref{thm:densbd}]
We have $\mu = \rho\nu + \mu_s$ where $\rho \in L^1(X; \nu)$ and $\mu_s \bot \nu$. Then {by Proposition \ref{prop:coeffprops} (ii) and \ref{prop:subadd},}
\begin{align*}
\underline{Q}_{p,s}(\mu)^{p^\prime} & \geq \underline{Q}_{p,s}(\rho\nu)^{p^\prime}; \\
\overline{Q}_{p,s}(\mu)^{p^\prime} & \leq \overline{Q}_{p,s}(\rho\nu)^{p^\prime} + \overline{Q}_{p,s}(\mu_s)^{p^\prime}. 
\end{align*}
We argue separately for the absolutely continuous and singular components.
\begin{itemize}
\item (a.c.~component) By Lemma \ref{lem:nua}, Theorem \ref{thm:densbd} holds for arbitrary restrictions $\nu|_A \in \M^p_+(X)$. The statement for $\mu = \rho\nu$ then follows from Proposition \ref{prop:genzadorineq}, with the choices $f = C_1 \overline{\vartheta}_s(x)^{-\frac{p}{s+p}}$ and $g = C_2 \underline{\vartheta}_s(x)^{-\frac{p}{s+p}}$.

\item (singular component) We now show that, under the assumptions that $\mu$ (hence $\mu_s$) is $(p,s)$-quantizable and $\underline{\vartheta}_s > 0$ $\mu$-a.e., we have $\overline{Q}_{p,s}(\mu_s) = 0$. Since $\mu_s \bot \nu$, there exists $A \subseteq X$ Borel such that $\mu_s(X \setminus A) = \nu(A) = 0$. Assume wlog that $\underline{\vartheta}_s(x) > 0$ for all $x \in A$. We prove $\overline{Q}_{p,s}(\mu_s) = 0$ by approximating $A$ from below and applying the assumption that $\mu$, thus $\mu_s$, is $(p,s)$-quantizable.

Given $\vartheta, \delta > 0$ write
\begin{align*} 
A_{\vartheta, \delta} &:= \{ x \in A \mid \underline{\vartheta}_s(x, \delta) > \vartheta \}; \quad
A_{\vartheta} := \{ x \in A \mid \underline{\vartheta}_s(x) > \vartheta \}.
\end{align*}
By construction, each $A_{\vartheta} = \bigcup_{\delta > 0} A_{\vartheta,\delta}$ and $A = \bigcup_{\vartheta > 0} A_\vartheta$ with both unions being monotone limits. Let $\vartheta, \delta > 0$ and take $K \subseteq A_{\vartheta,\delta}$ compact. Then Lemma \ref{lem:conclb} for $p = \infty$ yields
\[ \overline{Q}_{\infty,s}(K)^{s} \leq 2^{s} \vartheta^{-1} \nu(K) = 0, \]
thus also $\overline{Q}_{p,s}^{(\mu)}(K) = 0$. Then apply Proposition \ref{prop:qpscont} to deduce first that each $\overline{Q}_{p,s}^{(\mu)}(A_{\vartheta,\delta}) = 0$, and then $\overline{Q}_{p,s}(\mu_s) = \overline{Q}_{p,s}^{(\mu)}(A) = 0$. \qedhere
\end{itemize}
\end{proof}

{
\subsection{Applications to quantization dimensions}\label{sect:quantdims}

With the choice $\nu = \mu$, we can use Theorem \ref{thm:densbd} to deduce qualitative statements about the quantization dimensions of $\mu$. We first review the definitions of local, Hausdorff and packing dimensions of measures, following \cite[Sect.~10.1]{falc97}:

\begin{recall*}[Dimensions of measures]
The \emph{lower} resp.~\emph{upper local dimension} of a finite Borel measure $\mu$ on a metric space $X$ at a point $x \in X$ is
\begin{align*}
\underline\dim_{loc} \mu(x) := \liminf_{r \to 0} \frac{\log \mu(B_r(x))}{\log r} 
; \quad 
\overline\dim_{loc} \mu(x) := \limsup_{r \to 0} \frac{\log \mu(B_r(x))}{\log r} 
. 
\end{align*}
Local dimensions control Hausdorff densities as such (note the interchange of upper and lower limits):
\begin{align*}
\underline\dim_{loc} \mu(x) > s & \implies \overline \vartheta^{(\mu)}_s(x) = 0; \quad\,\,\, \overline\dim_{loc} \mu(x) > s \implies \underline \vartheta^{(\mu)}_s(x) = 0; \\
\underline\dim_{loc} \mu(x) < s & \implies \overline \vartheta^{(\mu)}_s(x) = \infty; \quad \overline\dim_{loc} \mu(x) < s \implies \underline \vartheta^{(\mu)}_s(x) = \infty.
\end{align*}
These follow as in \cite[Prop.~2.3]{falc97}: along any subsequence $(r_k)_k$ converging to zero, we have
\[ \frac{\log \mu(B_{r_k}(x))}{\log r_k} > s + \varepsilon \quad \bigl( < s - \varepsilon \bigr) \quad \implies \quad \lim_{k \to \infty} \frac{\mu(B_{r_k}(x))}{r_k^s} = 0 \quad \bigl( = \infty \bigr). \]

The \emph{lower} resp.~\emph{upper Hausdorff} and \emph{packing dimensions} of $\mu$ are the essential infimum resp.~supremum of the local dimensions of $\mu$:
\begin{align*}
\dim_H \mu & = \sup\{s \mid \underline\dim_{loc} \mu \geq s \ \mu\text{-a.e.} \} 
, \quad 
\dim_P \mu = \sup\{s \mid \overline\dim_{loc} \mu \geq s \ \mu\text{-a.e.} \} 
; \\ 
\dim_H^\ast \mu & = \inf\{s \mid \underline\dim_{loc} \mu \leq s \ \mu\text{-a.e.} \} 
, \quad\ 
\dim_P^\ast \mu = \inf\{s \mid \overline\dim_{loc} \mu \leq s \ \mu\text{-a.e.} \} 
. 
\end{align*}
On $\R^d$, these notions can be defined alternatively as the minimal resp.~maximal Hausdorff and packing dimensions of sets with positive $\mu$-measure, hence the nomenclature. 

We also note the precise definitions of \emph{lower} resp.~\emph{upper quantization dimensions} of $\mu$ of order $p$, cf.~\cite[Sect.~11.1]{quantbook} and \cite{Potz2001}:
\begin{align*}
\underline D_p(\mu) & = \liminf_{N \to \infty} \; \frac{\log N}{-\log e_{N,p}(\mu)} 
 = \sup\{s \mid \underline Q_{p,s}(\mu) = \infty\} = \inf\{s \mid \underline Q_{p,s}(\mu) = 0\}; \\
\overline D_p(\mu) & = \limsup_{N \to \infty} \frac{\log N}{-\log e_{N,p}(\mu)} 
= \sup\{s \mid \overline Q_{p,s}(\mu) = \infty\} = \inf\{s \mid \overline Q_{p,s}(\mu) = 0\}.
\end{align*}
The equivalences follow as in \cite[Prop.~11.3]{quantbook}. 
\end{recall*}

The following statements then follow directly from Theorem \ref{thm:densbd} (specifically from Lemma \ref{lem:a0inf}):

\begin{corol}\label{cor:quantdims}
Let $\mu \in \M^p_+(X)$, $p < \infty$. Then
\begin{itemize}
\item[(i)] $\underline D_p(\mu) \geq \dim_H^\ast \mu$;
\item[(ii)] $\underline D_p(\mu|_A) \geq \dim_H \mu$ for any $A \subseteq X$ such that $\mu(A) > 0$;
\item[(iii)] If $\mu$ is $(p,s)$-quantizable for some $s > 0$ (necessarily $s \geq \overline D_p(\mu)$), then either $s \leq \dim_P^\ast \mu$ or $\overline Q_{p,s}(\mu) = 0$; 
\item[(iv)] If $\mu$ is $(p,s)$-quantizable for some $s > 0$, then either $s \leq \dim_P \mu$ or $\overline Q_{p,s}(\mu|_A) = 0$ for some $A \subseteq X$ such that $\mu(A) > 0$.
\end{itemize}
\end{corol}
\begin{proof}\
\begin{itemize}
\item[(i)] For any $s < \dim_H^\ast \mu$, the set $A = \{\underline\dim_{loc} \mu > s\}$ has positive $\mu$-measure. Then $\overline \vartheta^{(\mu)}_s = 0$ on $A$. Applying Theorem \ref{thm:densbd} with $\nu = \mu$ then yields
\[ \underline Q_{p,s}(\mu)^{p^\prime} \geq C_1 \int_X \overline \vartheta^{(\mu)}_s(x)^{-\frac{p}{s+p}} \di\mu(x) \geq C_1 \int_A \overline \vartheta^{(\mu)}_s(x)^{-\frac{p}{s+p}} \di\mu(x) = \infty, \]
so that $\underline Q_{p,s}(\mu) = \infty$ and $s \leq \underline D_p \mu$. 
\item[(ii)] For any $s < \dim_H \mu$, $\underline\dim_{loc} \mu > s$ hence $\overline \vartheta^{(\mu)}_s = 0$ $\mu$-a.e. Then by Theorem \ref{thm:densbd}, any subset $A \subseteq X$ with $\mu(A) > 0$ satisfies
\[ \underline Q_{p,s}^{(\mu)}(A)^{p^\prime} \geq C_1 \int_A \overline \vartheta^{(\mu)}_s(x)^{-\frac{p}{s+p}} \di\mu(x) = \infty, \]
so that $s \leq \underline D_p (\mu|_A)$.
\item[(iii)] Suppose $s > \dim_P^\ast \mu$, so that $\overline\dim_{loc} \mu < s$ hence $\overline \vartheta^{(\mu)}_{s} = \infty$ $\mu$-a.e. Then {by the $(p,s)$-quantizability of $\mu$}, Theorem \ref{thm:densbd} implies
\[ \overline Q_{p,s}(\mu)^{p^\prime} \leq C_1 \int_X \underline \vartheta^{(\mu)}_s(x)^{-\frac{p}{s+p}} \di\mu(x) = 0. \]
\item[(iv)] Suppose $s > \dim_P \mu$, so that the set $A = \{\underline\dim_{loc} \mu > s\}$ has positive $\mu$-measure. Then $\overline \vartheta^{(\mu)}_s(x) = \infty$ for each $x \in A$, hence Theorem \ref{thm:densbd} yields
\[ \overline Q_{p,s}(\mu|_A)^{p^\prime} \leq C_1 \int_A \underline \vartheta^{(\mu)}_s(x)^{-\frac{p}{s+p}} \di\mu(x) = 0. \qedhere \]
\end{itemize}
\end{proof}

The inequality (i) is known for compactly supported measures on $\R^d$ \cite[Thm.~1]{Potz2001}, see also \cite[Thm.~11.6]{quantbook}. The lower bound in Theorem \ref{thm:densbd} is thus a quantitative version of this inequality, generalized to arbitrary spaces and made more precise by the possibility to choose $\nu \neq \mu$. In particular, on any set $A$ of finite Hausdorff measure, for $\nu = \cH^s|_A$ we have that $\overline \vartheta^{(\nu)}_s(x) \leq 1$  for $\cH^m$-a.e.~$x \in A$ (cf.~\cite[2.10.19]{Federer1996}), so that
\[ \underline Q_{p,s}(\mu|_A)^{p^\prime} \geq C_1 \int_A \rho^{\frac{s}{p+s}} \di \cH^m \quad \text{for} \quad \mu|_A = \rho \cH^s|_A + \mu^\bot. \]
Inequalities (iii) and (iv) instead give partial converses to the existing bound $\overline D_p(\mu) \geq \dim_P^\ast \mu$, known for compactly supported measures on $\R^d$ \cite[Thm.~1]{Potz2001}, and imply restrictions on the notion of $(p,s)$-quantizability: either $\overline D_p(\mu) = \dim_P^\ast \mu$ or $\mu$ is only $(p,s)$-quantizable in the trivial case $\overline Q_{p,s}(\mu) = 0$. We remark also that the sufficient conditions (A1) and (A2) given in \cite[Thm.~1.2]{Potz2001}, which imply $\overline D_p(\mu) = \dim_P^\ast \mu$ resp.~$\overline Q_{p,s}(\mu) < \infty$, are both special cases of Theorem \ref{thm:randstab} hence guarantee $(p,s)$-quantizability.

This motivates the following self-similar counterexample, which illustrates that for $s < d$, $(p,s)$-quantizability may fail even for compactly supported measures on $\R^d$:

\begin{example}[Self-similar counterexample]\label{ex:selfsimcounter}
Let $\mu$ be a non-uniform self-similar measure on $[0,1]$, given by 
\[ \mu = (1-\lambda) (F_0)_\# \mu + \lambda (F_1)_\# \mu, \quad 0 < \lambda < \frac12, \]
where $F_0 \colon [0,1] \to [0,\frac12]$ and $F_1 \colon [0,1] \to [\frac12,1]$ are the similarity transformations $t \mapsto \frac12 t$ and $t \mapsto \frac12 (t+1)$. Such measures are considered in \cite[Prop.~10.4]{falc97}, which shows that
\[ \dim_H \mu = \dim_H^\ast \mu = \dim_P \mu = \dim_P^\ast \mu = - \frac{\lambda \log \lambda + (1-\lambda) \log (1-\lambda)}{\log 2} =: D_0, \] 
which lies strictly between $0$ (attained at $\lambda \in \{0,1\}$ in which case $\mu = \delta_\lambda$) and $1$ (attained at $\lambda = \frac12$ in which case $\mu = \cL^1|_{[0,1]}$). Moreover, there is a Borel set $E$ with Hausdorff dimension $D_0$ such that $\mu(E^c) = 0$. In particular, $\mu \bot \cH^s$ for $s > D_0$. 

On the other hand, since the maps $F_i$ satisfy the \emph{open set condition} with $U = (0,1)$ (i.e.~the images $F_i(U)$ are disjoint and contained in $U$), \cite[Thm.~3.1]{Graf:2000aa} implies that the quantization dimension $D_p$ of $\mu$ exists and solves an implicit equation, and moreover
\[ 0 < \underline{Q}_{p,D_p}(\mu) \leq \overline{Q}_{p,D_p}(\mu) < \infty. \] 
Moreover, by \cite[Lem.~14.16]{quantbook} $D_p$ is strictly increasing with $p$ and converges as $p \to \infty$ to the box-counting dimension of $\supp\mu$, which is $1$ in this case. 

Hence by Corollary \ref{cor:quantdims} (i) and (iii), $\mu$ cannot be $(p,s)$-quantizable at the exact dimension $s = D_p$, since $D_p > \dim_H^\ast \mu = \dim_P^\ast \mu$ and $\overline{Q}_{p,D_p}(\mu) > 0$. Theorem \ref{thm:densbd} does not provide any quantitative information for the exact dimension $s = D_p$: taking $\nu = \cH^{D_p}$ gives only a vacuous lower bound $\underline Q_{p,D_p}(\mu) \geq 0$ since $\cH^{D_p} \bot \mu$, and likewise for the choice $\nu = \mu$ since $\vartheta^{(\mu)}_s = \infty$ $\mu$-a.e.~for $s > D_0$. Hence $\mu$ is also a counterexample to Lemma \ref{lem:a0inf} (ii), as $\overline{Q}_{p,D_p}(\mu) \neq 0$.
\end{example}
}

\section{Sufficient conditions for $(p,s)$-quantizability}\label{sect:psstab}

{
We now prove various integrability conditions for $(p,s)$-quantizability, especially the random quantizer bounds Theorem \ref{thm:randstab} and Corollary \ref{cor:volgrowstab}, and apply them to different classes of domains. 
}

{
We first note that geodesic spaces with $O(f)$ \emph{covering growth}, as defined in \cite{covgrow}, admit the following integral condition for $(p,s)$-quantizability: 
}

\begin{prop}\label{prop:covstab}
Let $X$ be a geodesic space with $s$-dimensional $O(f)$ covering growth around a point $x_0 \in X$, and $p \in [1,\infty)$. Then any measure $\mu \in \M^p_+(X)$, such that
\[ \int_X \left[ d(x,x_0)^{p+\delta} + f(d(x,x_0))^p \right] \di \mu(x) < \infty \]
for some $\delta > 0$, is $(p,s)$-quantizable.
\end{prop}
\begin{proof}
By \cite[Thm.~1.6]{covgrow}, there exists $C > 0$ such that
\[ \overline{Q}_{p,s}(\mu)^p \leq \int_X \left[ 1 + d(x,x_0)^{p+\delta} + f(d(x,x_0))^p \right] \di \mu(x) \]
holds for arbitrary $\mu \in \M^p_+(X)$. The statement then follows from Proposition \ref{prop:intubstab}.
\end{proof}

{This subsumes all existing integrability conditions on $\R^d$ and Riemannian manifolds presented in Section \ref{sect:hist}.}

\subsection{Random quantizer estimates}\label{sect:randquant}

We prove Theorem \ref{thm:randstab} and deduce Corollary \ref{cor:volgrowstab} by a random quantizer argument, estimating the quantization error of $\mu$ by sampling $N$ points independently from another measure $\nu$.
For clarity, we give the general form of random quantizer arguments below:

\begin{lemma}\label{lem:randquant}
Let $X$ be a Polish space, $p \in [1,\infty)$, $s \in (0,\infty)$. Let $\nu$ be a Borel probability measure on $X$ and $\mu \in \M^p_+(X)$. Then for all $N \in \N$, the following upper bound holds:
\[ N^{p/s} e_{N,p}(\mu)^p \leq \int_X F_{\nu, N}(x) \di \mu(x), \]
where
\[ F_{\nu, N}(x) := p \int_0^\infty N^{p/s} \left[1 - \nu(B_r(x))\right]^N r^{p-1} \di r. \]
\end{lemma}
\begin{proof}
Denoting by $\nu^N$ the $N$-fold product measure $\nu \otimes \ldots \otimes \nu$, we can apply Tonelli's Theorem to bound
\begin{align*}
N^{p/s} e_{N,p}(\mu)^p & \leq \int_{X^N} N^{p/s} e_p(\mu; S)^p \di \nu^N(S)
 = \int_X N^{p/s} \left[ \int_{X^N} d(x,  S)^p \di \nu^N(S) \right] \di \mu(x) =: \int_X F_{\nu, N}(x) \di \mu(x).
\end{align*}
where by slight abuse of notation we use $S$ to denote both $N$-tuples of points $(y_1, \ldots, y_N)$ and the associated sets $\{y_1, \ldots, y_N\}$.
For fixed $x \in X$, we apply the layer cake representation to $F_{\nu, N}(x)$:
\begin{align*}
 F_{\nu, N}(x) & = N^{p/s} \int_{X^N} d(x,  S)^p \di \nu^N(S)
 \\ & = N^{p/s} p \int_0^\infty r^{p-1} \nu^N(\{S \mid d(x, S) \geq r\}) \di r
 \\ & = p \int_0^\infty r^{p-1} N^{p/s} \nu^N(\{ (y_1, \ldots, y_N) \mid \min_{i=1}^N d(x, y_i) \geq r\}) \di r
 \\ & = p \int_0^\infty r^{p-1} N^{p/s} \left( 1 - \nu(B_r(x)) \right)^N \di r,
\end{align*}
noting that
\[ \{ (y_1, \ldots, y_N) \mid \min_{i=1}^N d(x, y_i) \geq r\} = \prod_{i=1}^N \{ y_i \mid d(x, y_i) \geq r \} = \left[ X \setminus B_r(x) \right]^N. \qedhere \]
\end{proof}

The main difficulty lies in choosing the measure $\nu$ appropriately in order to control the functions $F_{\nu, N}(x)$ independently of $N$. If there exists a positive measurable function $F_\nu$ which dominates each $F_{\nu,N}$, then $\int F_\nu \di \mu < \infty$ is sufficient to ensure the $(p,s)$-quantizability of $\mu$; {in that case, Fatou's lemma in fact implies the tighter bound} 
\[ \overline Q_{p,s}(\mu)^p \leq \int_X \limsup_{N \to \infty} F_{\nu,N}(x) \di \mu(x). \]
To that end, we note the following identities:
\begin{itemize}
\item For $t \in \R$ arbitrary, $1 - t \leq e^{-t}$ hence $(1 - t)^N \leq e^{-N t}$.
\item For $A > 0$ arbitrary,
\[ p \int_0^\infty r^{p-1} e^{-A r^s} \di r \underset{t = A r^s}{=} \frac{p}{s} A^{-p/s} \int_0^\infty t^{\frac{p}{s}-1} e^{-t} \di t = \frac{p}{s} \Gamma\left(\frac{p}{s}\right) A^{-p/s} =: C_1 A^{-p/s}. \]
\item For $A > 0$ arbitrary,
\[ \sup_{M > 0} M^{p/s} e^{-AM} \underset{t = A M}{=} A^{-p/s} \sup_{t > 0} t^{p/s} e^{-t} =: C_2 A^{-p/s}, \]
noting that $-t+\frac{p}{s} \log t$ is a concave function and attains its maximum at $t = \frac{p}{s}$.
\end{itemize}
We now use these identities to control $F_{\nu, N}$ uniformly in $N$ under the assumptions of Theorem \ref{thm:randstab}.

\begin{proof}[Proof of Theorem \ref{thm:randstab}]
Let $\mu \in \M^p_+(X)$, $N \in \N$ and $\delta > 0$ be arbitrary. If $\underline{\vartheta}_s^{(\nu)}(x, \delta) = 0$ on a set of nonzero $\mu$-measure, the upper bound is vacuously true, so suppose $\underline{\vartheta}_s^{(\nu)}(x, \delta) > 0$ for $\mu$-a.e.~$x \in X$. Without loss of generality, we can also take $R_0$ arbitrarily large so that $\beta R_0 > 1$.

Fix $x \in X$ such that $\underline{\vartheta}_s^{(\nu)}(x, \delta) > 0$.
We bound $N^{p/s} \left( 1 - \nu(B_r(x)) \right)^N$ uniformly in $N$ by considering the following three cases separately:
\begin{itemize}
\item[(i)] $r \leq \delta$: Since $\nu(B_r(x)) \geq \underline{\vartheta}_s^{(\nu)}(x, \delta) r^s$, we have
\[ \left( 1 - \nu(B_r(x)) \right)^N \leq \left( 1 - \underline{\vartheta}_s^{(\nu)}(x, \delta) r^s \right)^N \leq \exp\left( - N \underline{\vartheta}_s^{(\nu)}(x, \delta) r^s\right). \]
Integrating,
\[ p \int_0^\delta r^{p-1} \left( 1 - \nu(B_r(x)) \right)^N \di r \leq p \int_0^\infty r^{p-1} \exp\left( - N \underline{\vartheta}_s^{(\nu)}(x, \delta) r^s\right) \di r = C_1 N^{-p/s} \underline{\vartheta}_s^{(\nu)}(x, \delta)^{-p/s}. \]
\item[(ii)] $\delta \leq r \leq R_0 + d(x, x_0)$: We apply $\nu(B_r(x)) \geq \nu(B_\delta(x)) \geq \underline{\vartheta}_s^{(\nu)}(x, \delta) \delta^s$ to obtain
\[ N^{p/s} \left( 1 - \nu(B_r(x)) \right)^N \leq N^{p/s} \left( 1 - \underline{\vartheta}_s^{(\nu)}(x, \delta) \delta^s \right)^N \leq N^{p/s} \exp\left( - N \underline{\vartheta}_s^{(\nu)}(x, \delta) \delta^s\right) \leq C_2 \underline{\vartheta}_s^{(\nu)}(x, \delta)^{-p/s} \delta^{-p}. \]
Integrating,
\begin{align*} 
p \int_\delta^{R_0+d(x,x_0)} r^{p-1} N^{p/s} \left( 1 - \nu(B_r(x)) \right)^N \di r 
& \leq C_2 \underline{\vartheta}_s^{(\nu)}(x, \delta)^{-p/s} \delta^p \cdot p \int_\delta^{R_0+d(x,x_0)} r^{p-1} \di r 
\\ & = C_2 \underline{\vartheta}_s^{(\nu)}(x, \delta)^{-p/s} \left[ \delta^{-p} (R_0+d(x,x_0))^p - 1 \right].
\end{align*}
\item[(iii)] $r \geq R_0 + d(x,x_0)$: In this case, $B_r(x) \supseteq B_{r - d(x,x_0)}(x_0)$ hence
\[ \left( 1 - \nu(B_r(x)) \right)^N \leq \left( 1 - \nu(B_{r - d(x,x_0)}(x_0)) \right)^N \leq \left[ \beta (r - d(x,x_0)) \right]^{-N \alpha}. \]
Integrating over $r$, for every $N > p/\alpha$ we have
\begin{align*} 
p \int_{R_0+d(x,x_0)}^\infty r^{p-1} \left( 1 - \nu(B_r(x)) \right)^N \di r 
& \leq p \int_{R_0+d(x,x_0)}^\infty r^{p-1} \left[ \beta (r - d(x,x_0)) \right]^{-N \alpha} \di r
\\ & = \frac{p}{\beta^{p-1}} \int_{\beta R_0}^\infty (t+d(x,x_0))^{p-1} t^{-N \alpha} \di t
\\ & \leq \frac{2^{p-2} p}{\beta^{p-1}} \int_{\beta R_0}^\infty (t^{p-1} +d(x,x_0))^{p-1}) t^{-N \alpha} \di t
\\ & \leq \frac{2^{p-2} p}{\beta^{p-1}} \left[ \frac{(\beta R_0)^{p-N \alpha}}{N \alpha - p} + d(x,x_0)^{p-1} \frac{(\beta R_0)^{1-N \alpha}}{N \alpha - 1} \right]
\\ & \leq C_3 \left[ 1 + d(x,x_0)^{p-1} \right] (\beta R_0)^{-N \alpha}
\end{align*}
for some $C_3 > 0$ dependent on $p$, $\alpha$, $\beta$ and $R_0$. Multiplying by $N^{p/s}$ and taking the supremum over $N$ yields
\[ p N^{p/s} \int_{R_0+d(x,x_0)}^\infty r^{p-1} \left( 1 - \nu(B_r(x)) \right)^N \di r \leq C_3^\prime \left[ 1 + d(x,x_0)^{p-1} \right], \]
where we use the assumption that $\beta R_0 > 1$ and note that $\sup_{N \in \N} N^{p/s} (\beta R_0)^{-N \alpha} = C (\alpha \log(\beta R_0))^{-p/s}$.
\end{itemize}
Summing up these three cases, we obtain
\begin{align*}
  F_{\nu,N}(x) & = p \left[ \int_0^\delta + \int_\delta^{R_0+d(x,x_0)} + \int_{R_0+d(x,x_0)}^\infty \right] r^{p-1} N^{p/s} \left( 1 - \nu(B_r(x)) \right)^N \di r
 \\ & \leq (C_1 - C_2) \underline{\vartheta}_s^{(\nu)}(x, \delta)^{-p/s} + C_2 \underline{\vartheta}_s^{(\nu)}(x, \delta)^{-p/s} \delta^{-p} (R_0 + d(x, x_0))^p + C_3 (1 + d(x,x_0)^{p-1})
 \\ & \leq C (1+\delta^{-p}) (1 + d(x,x_0)^p) \underline{\vartheta}_s^{(\nu)}(x, \delta)^{-p/s} + 2 C_3 (1 + d(x,x_0)^p).
\end{align*}
Integrating with respect to $x$ yields the desired inequality, {from which the sufficient condition for $(p,s)$-quantizability follows by Proposition \ref{prop:intubstab}}.
\end{proof}

In the proof of Theorem \ref{thm:randstab}, it is the mid-range regime $\delta \leq r \leq R_0 + d(x, x_0)$ that gives rise to the largest imprecision in the upper bound. For particular domains, one may choose $\delta$ dependent on $x$ in order to eliminate this case:

\begin{example}[Pierce's lemma on $\R$]
Proofs of Pierce's lemma on the real line (cf.~\cite[Lem.~6.6]{quantbook}) employ a Pareto distribution $\nu$ on $[1,\infty)$ given by the density $\rho_\nu(t) = \beta t^{-\beta-1}$, $\beta := \alpha/p$, $\alpha > 0$ arbitrary. This measure satisfies the power law decay condition with $x_0 = 1$ and $R_0 = 0$, and moreover, for any $x \in (1,\infty)$ and $r < x-1$, 
\[ \frac{\nu(B_r(x))}{2r} = \frac{1}{2r} \int_{x-r}^{x+r} \beta t^{-\beta-1} \di t \geq \beta x^{-\beta-1} \]
by the convexity of $\rho_\nu$. Thus instead of picking a fixed $\delta$ for each $x$, one may pick $\delta(x) = x-1 = R_0 + d(x,1)$, so that
\[ \underline{\vartheta}_1^{(\nu)}(x, \delta(x))^{-p/1} \geq (\beta x^{-\beta-1})^{-p} = \beta^{-p} x^{p+\alpha}. \] 
Arguing similarly to the above proof while excluding case (ii) yields the same upper bound as in Pierce's lemma, while case (ii) would yield an extraneous $x^{2p+\alpha}$ term if one were instead to fix $\delta$ independently of $x$.
\end{example} 

We now prove Corollary \ref{cor:volgrowstab} by constructing an appropriately decaying probability measure $\hat\nu$ out of a {locally finite} measure $\nu$:

{
\begin{proof}[Proof of Corollary \ref{cor:volgrowstab}]
We cover $X$ with overlapping annuli:
\[ A_0 := B_{2 R_0}(x_0); \qquad
 A_k := B_{\frac{k+3}{2} R_0}(x_0) \setminus B_{\frac{k+1}{2} R_0}(x_0), \quad k \geq 1.
\]
Since $x_0 \in \supp\nu$, we have that $0 < \nu(A_0) < \infty$. 
For each $k$, we set $\hat \nu_k := \frac{1}{\nu(A_k)} \nu|_{A_k}$ if $\nu(A_k) > 0$, and $\hat \nu_k := \hat \nu_0$ if $\nu(A_k) = 0$. 
We then define the probability measure
\[ \hat \nu := \sum_{k=0}^\infty \lambda_k \hat \nu_k, \qquad \lambda_k := \frac{1}{(k+1)^\alpha} - \frac{1}{(k+2)^\alpha}, \]
noting that $\sum_{k=0}^\infty \lambda_k = 1$. 
\begin{itemize}
\item Let $R \geq 2 R_0$. Then for the minimal integer $l \geq 1$ such that $R < \frac{l+3}{2} R_0$, we have that $B_R(x_0) \supseteq \bigcup_{k=0}^{l-1} A_k$.
Thus $\hat\nu_k(B_R(x_0)) = 1$ for all $k \leq l-1$, and
\[ 1 - \hat \nu(B_R(x_0)) = \sum_{k=0}^\infty \lambda_k (1 - \hat\nu_k(B_R(x_0))) \leq \sum_{k=l}^\infty \lambda_k = (l+1)^{-\alpha} \leq \left(\frac{R}{R_0}\right)^{-\alpha}, \]
where we apply that $\frac{R}{R_0} < \frac{l+3}{2} \leq l+1$ since $l \geq 1$. 
\item Since the integrability condition only becomes more and more general as $\delta \to 0$, we can assume wlog that $\delta \leq R_0/4$. 

Then for any $x \in X$ with $d(x,x_0) = R$, either $R < \frac74 R_0$ so that $B_\delta(x) \subseteq A_0$, or there exists a positive integer $k$ such that 
\[ (R - \delta, R + \delta) \subseteq \left[ \frac{k+1}{2} R_0, \frac{k+3}{2} R_0 \right) \quad \implies \quad B_\delta(x) \subseteq A_k, \]
since the intervals on the right-hand side intersect on intervals of width $R_0/2$.

In the former case, we have that
\[ \underline \vartheta^{(\hat \nu)}_s(x,\delta) \geq \frac{\lambda_0}{\nu(A_0)} \underline \vartheta^{(\nu)}_s(x,\delta) =: C_0 \underline \vartheta^{(\nu)}_s(x,\delta), \]
while for the latter case, either $\nu(A_k) = 0$ in which case $\underline \vartheta^{(\hat \nu)}_s(x,\delta) = \underline \vartheta^{(\nu)}_s(x,\delta) = 0$, or 
\[ \underline \vartheta^{(\hat \nu)}_s(x,\delta) \geq \frac{\lambda_k}{\nu(A_k)} \underline \vartheta^{(\nu)}_s(x,\delta), \]
with
\begin{align*}
\nu(A_k) & = \nu\left(B_{\frac{k+3}{2} R_0}(x_0) \setminus B_{\frac{k+1}{2} R_0}(x_0)\right) \leq V\left(\frac{k+1}{2} R_0 \right) \leq V(R); \\ 
\lambda_k & = \int_{k+1}^{k+2} \frac{\alpha}{t^{\alpha+1}} \di t \geq \frac{\alpha}{(k+2)^{\alpha+1}} \geq \alpha \left(\frac{3R}{R_0} \right)^{-\alpha-1}
\end{align*} 
noting that $R \geq \frac{k+1}{2} R_0 \geq \frac{k+2}{3} R_0$.
\item Putting everything together, we have
\begin{align*}
& \quad \left[ 1 + d(x,x_0)^p \right] \underline \vartheta^{(\hat \nu)}_s(x,\delta)^{-p/s}
\\ & \leq C \left[ 1 + d(x,x_0)^p \right] \max\left\{1, V(d(x,x_0))^{p/s} d(x,x_0)^{(\alpha+1)p/s}  \right\} \underline \vartheta^{(\nu)}_s(x,\delta)^{-p/s}
\\ & \leq C \left[ 1 + d(x,x_0)^p \right] \left[ 1 + V(d(x,x_0))^{p/s} d(x,x_0)^{(\alpha+1)p/s} \right] \underline \vartheta^{(\nu)}_s(x,\delta)^{-p/s},
\end{align*}
so that the $\mu$-integrability of the last expression implies the $\mu$-integrability of the first, hence by Theorem \ref{thm:randstab} applied to $\hat\nu$, the $(p,s)$-quantizability of $\mu$. \qedhere
\end{itemize}
\end{proof}
}

We can interpret $\hat\nu$ as a mixture distribution, which first draws a random radius from a {power law} distribution, and then samples from an annulus around that radius according to $\nu$. 

For the special case $\nu = \cL^d$ on $\R^d$, this recovers Pierce's lemma only for exponents $q > p+\frac{p}{d}$. We also compare this bound with the covering growth condition for nonnegative vs.~negative Ricci curvature below in Example \ref{ex:cdkn}. 

{The following statement for compactly supported measures, which is analogous to existing results such as \cite[Thm.~1.2]{Potz2001} and \cite[Thm.~4.2, 4.3]{lossy-2021}, also follows directly from Theorem \ref{thm:randstab}:}

\begin{corol}\label{cor:compstab}
Let $\mu \in \M^\infty_+(X)$, and suppose there exists a Borel probability measure $\nu$ on $X$ and $s, \alpha, \delta > 0$ such that
\[ \int_X \underline \vartheta^{(\nu)}_s(x, \delta)^{-\alpha} \di \mu(x) < \infty. \]
Then $\mu$ is $(p,s)$-quantizable for any $p \in [1,\infty)$.
\end{corol}

In particular, if $\nu(B_r(x)) \geq \lambda r^s$ for $x \in \supp\mu$ and $r < \delta$, $\underline \vartheta^{(\nu)}_s(\cdot, \delta) \geq \lambda$ uniformly and the above condition is automatically satisfied; {this can also be obtained directly from Lemma \ref{lem:concub}.}

{{Conversely, if a measure $\mu$ supported on a compact set $K$ fails to be $(p,s)$-quantizable, $K$ admits no probability measure $\nu$ such that $\underline \vartheta^{(\nu)}_s(\cdot, \delta)$ is uniformly lower bounded on $K$. This holds in particular when $\cH^s(K) = \infty$. Consider e.g.~the measure $\mu$ constructed in Example \ref{ex:selfsimcounter}, which is supported on the $1$-dimensional set $K = [0,1]$ but has quantization dimension $s = D_p(\mu)$ strictly less than $1$. If there existed a measure $\nu$ supported on $K$ such that $\underline \vartheta^{(\nu)}_s(x, \delta) \geq \lambda > 0$ for all $x \in K$, by \cite[Prop.~2.2]{falc97} it would follow that $\cH^s(K) < \infty$, which is impossible since $\dim_H K = 1 > s$. }}

{

\subsection{Example settings}\label{sect:examples}

We now present a few example settings of metric measure spaces $(X, \nu)$ for which the Hausdorff densities of $\nu$ can be controlled, allowing for a more precise statement of Theorem \ref{thm:densbd}. 

\begin{example}[Ahlfors regular domains]
Suppose a metric space $X$ admits an Ahlfors regular measure $\nu$ of dimension $s > 0$, i.e.~there exist constants $\alpha, \beta, \delta > 0$ such that 
\[ \alpha r^s \leq \nu(B_r(x)) \leq \beta r^s \quad \text{for all} \quad x \in \supp\nu, \ 0 < r < \delta. \]
This implies directly that
\[ \alpha \leq \underline \vartheta^{(\nu)}(\cdot, \delta) \leq \underline \vartheta^{(\nu)} \leq \overline \vartheta^{(\nu)} \leq \overline \vartheta^{(\nu)}(\cdot, \delta) \leq \beta \]
uniformly on $\supp\nu$. This in fact implies that $\nu$ is sandwiched between multiples of $\cH^s$ (cf.~\cite[Prop.~2.2]{falc97}), hence $\cH^s$ itself is Ahlfors regular and Ahlfors regularity is inherent to the domain $X$.

Taking a suitable volume function $V$ for $\nu$, as defined in Corollary \ref{cor:volgrowstab}, we can conclude the following: any measure $\mu \in \M^p_+(X)$ which satisfies the higher integrability condition
\begin{equation}\label{eq:ahlcond} \int_X d(x,x_0)^{p+p/s+\alpha} V(d(x,x_0))^{p/s} \di\mu(x) < \infty \end{equation}
for some $\alpha > 0$ is $(p,s)$-quantizable, and satisfies
the following weak version of Zador's theorem:
\[ C_1 \beta^{-\frac{p}{s+p}} \int \rho^{\frac{s}{s+p}} \di \nu \leq \underline Q_{p,s}(\mu)^{p^\prime} \leq \overline Q_{p,s}(\mu)^{p^\prime} \leq C_2 \alpha^{-\frac{p}{s+p}} \int \rho^{\frac{s}{s+p}} \di \nu, \]
where $\rho$ is the density of the a.c.~component of $\mu$ with respect to $\nu$. 
In particular, any measure on an Ahlfors regular domain satisfying \eqref{eq:ahlcond} has quantization dimension at most $s$, with equality unless it is singular to $\cH^s$.
\end{example}

This is a significant improvement over existing results by Graf and Luschgy \cite[Sect.~12]{quantbook}, restated in Lemmas \ref{lem:conclb} and \ref{lem:concub}, which only capture the case of $\mu$ itself being Ahlfors regular and compactly supported. We note that while the integrability condition can still be improved by alternative constructions in Corollary \ref{cor:volgrowstab}, the volume function is still necessary and cannot be controlled by Ahlfors regularity alone. Indeed the model spaces $\bS^d$, $\R^d$ and $\H^d$ are all Ahlfors regular, but have constant, polynomial and exponential volume growth respectively.

Even more generally, we can consider measures on doubling spaces:

\begin{example}[Doubling spaces]\label{ex:doubling}
Let $(X,\nu)$ be a metric measure space, admitting constants $C, R_0 > 0$ such that $0 < \nu(B_{2r}(x)) \leq C \nu(B_r(x)) < \infty$ for all $x \in \supp\nu$ and $r \in (0,R_0]$. 

For $s = \log_2 C$ the \emph{doubling dimension} of $(X,\nu)$, the following volume comparison property holds:
\begin{equation}\label{eq:doubcomp} \frac{\nu(B_R(x))}{\omega_s R^s} \leq C \frac{\nu(B_r(x))}{\omega_s r^s} \quad \text{for all} \quad x \in X, \ 0 < r \leq R \leq R_0. \end{equation}
Indeed, taking $k \in \N$ such that $2^{k-1} r \leq R < 2^k r$, we can bound
\[ \nu(B_R(x)) \leq \nu(B_{2 \cdot 2^{k-1} r}(x)) \leq C^k \nu(B_r(x)) = C \left(\frac{2^{k-1}r}{r}\right)^s \nu(B_r(x)) \leq C \left(\frac{R}{r}\right)^s \nu(B_r(x)). \]
Then for each $\delta \leq R_0$, \eqref{eq:doubcomp} implies
\begin{equation}\label{eq:doublb} \frac{\nu(B_\delta(x))}{\omega_s \delta^s} \leq C \inf_{r < \delta} \frac{\nu(B_r(x))}{\omega_s r^s} = C \underline \vartheta_s(x,\delta), \end{equation}
thus $\underline \vartheta_s(x,\delta) > 0$ for each $x \in \supp\nu$, and taking $\delta \to 0$ then yields
\[ \overline \vartheta_s(x) = \limsup_{\delta \to 0} \frac{\nu(B_\delta(x))}{\omega_s \delta^s} \leq C \lim_{\delta \to 0} \inf_{r < \delta} \frac{\nu(B_r(x))}{\omega_s r^s} = C \underline \vartheta_s(x). \]
It also follows from \eqref{eq:doublb} that $\underline \vartheta_s(x,\delta)$ is uniformly lower bounded on each $K \subseteq \supp\nu$ compact:
\[ \inf_{x \in K} \underline \vartheta_s(x,\delta) \geq \frac{1}{C \omega_s \delta^s} \inf_{x \in K} \nu(B_\delta(x)) > 0, \]
noting that any minimizing sequence for the latter infimum admits a convergent subsequence $x_k \to \bar x \in K$, such that $B_{\delta}(x_k) \supseteq B_{\delta/2}(\bar x)$ for $k$ sufficiently large, with $\bar x \in \supp\nu$ hence $\nu(B_{\delta/2}(\bar x)) > 0$. 

Then likewise applying Corollary \ref{cor:volgrowstab} and Theorem \ref{thm:densbd}, we conclude that any measure $\mu \in \M^p_+(X)$ which is supported on $\supp\nu$ and 
satisfies the higher integrability condition
\begin{equation}\label{eq:doubcond} \int_X \left[1 + d(x,x_0)^{p+p/s+\alpha} V(d(x,x_0))^{p/s} \right] \underline \vartheta_s(x,\delta)^{-p/s} \di\mu(x) < \infty \end{equation}
for some $\alpha, \delta > 0$ is $(p,s)$-quantizable, and satisfies
\[ C_1 C^{-\frac{p}{s+p}} \int \rho^{\frac{s}{s+p}} \underline \vartheta_s(x)^{-\frac{p}{s+p}} \di \nu \leq \underline Q_{p,s}(\mu)^{p^\prime} \leq \overline Q_{p,s}(\mu)^{p^\prime} \leq C_2 \int \rho^{\frac{s}{s+p}} \underline \vartheta_s(x)^{-\frac{p}{s+p}} \di \nu. \]
In particular, any compactly supported measure on $\supp\nu$ is $(p,s)$-quantizable and has upper quantization dimension at most $s$. 

Note that usually $s$ is an imprecise upper bound on the Hausdorff dimension of $\nu$ (e.g.~when $R_0$ is taken too large), in which case $\underline \vartheta_s = \infty$ $\nu$-a.e. We can still conclude that any compactly supported measure on $\supp\nu$ is $(p,s)$-quantizable, and every $(p,s)$-quantizable measure on $X$ has vanishing $s$-dimensional quantization coefficients, hence upper quantization dimension at most $s$.

When the doubling property does hold for all radii, we can strengthen \eqref{eq:doubcond} even further: fixing $x_0 \in \supp\nu$ and $\delta > 0$, for any $x \in B_R(x_0)$ and $0 < r < \delta$ we have
\[ \frac{\nu(B_r(x))}{\omega_s r^s} \geq \frac{\nu(B_{R+\delta}(x))}{C \omega_s(R+\delta)^s} \geq \frac{\nu(B_{\delta}(x_0))}{C \omega_s(R+\delta)^s}, \]
since $B_{\delta}(x_0) \subseteq B_{R+\delta}(x)$ (cf.~\cite[Lem.~3.8]{covgrow} for a similar bound). Thus $\underline \vartheta_s(x, \delta) \geq C^\prime (d(x,x_0)+1)^s$ for $C^\prime > 0$ depending on $x_0$ and $\delta$. Likewise
\[ V(R) = \nu(B_{R+R_0}(x_0) \setminus B_R(x_0)) \leq C \frac{\nu(B_\delta(x_0))}{\omega_s \delta^s} R^s. \]
Putting these together, we obtain the integrability condition
\begin{equation}\label{eq:doubcond2} \int_X \left[1 + d(x,x_0)^{2p+p/s+\alpha} \right] (1 + d(x,x_0))^p \di\mu(x) < \infty,\end{equation}
which is equivalent to the integrability of $d(x,x_0)^{3p+p/s+\alpha}$.
\end{example}

Compare \eqref{eq:doubcomp} with the \emph{Bishop--Gromov volume comparison theorem} for Riemannian manifolds. 
This theorem also holds on \emph{curvature-dimension spaces} in the sense of Lott--Sturm--Villani \cite{LottVillani, sturm, sturmii}. On such spaces, we can conclude the following:

\begin{example}[Curvature-dimension spaces]\label{ex:cdkn}
Let $(X,\nu)$ be a metric measure space, which satisfies the curvature-dimension condition of Lott--Sturm--Villani with curvature lower bound $\kappa \in \R$ and dimension upper bound $s \in [1,\infty)$. This condition generalizes Riemannian manifolds with dimension at most $s$ and Ricci curvature $\mathrm{Ric} \geq (s-1) \kappa g$.

The Bishop--Gromov theorem on $(X,\nu)$ \cite[Thm.~2.3]{sturmii} states that
\begin{equation}\label{eq:bgcomp} \frac{\nu(B_R(x))}{\vol_\kappa^s(R)} \leq \frac{\nu(B_r(x))}{\vol_\kappa^s(r)} \quad \text{for all } x \in \supp\nu, \ 0 < r \leq R \leq \diam(X), \end{equation}
where $\vol_\kappa^s(R)$ is (at least for integer $s$) is the volume of balls of radius $R$ in the model $s$-dimensional Riemannian manifold $M_\kappa^s$ with constant sectional curvature $\kappa$. For small $r$, $\frac{\vol_\kappa^s(r)}{\omega_s r^s} \to 1$ for all $\kappa$, while for large $R$, $\vol_\kappa^s(R) = O(R^s)$ for $\kappa \geq 0$ and $\vol_\kappa^s(R) = O(e^{s \sqrt{-\kappa} R})$ for $\kappa < 0$. Moreover, for $\kappa > 0$ we have that $\diam(X) \leq \diam(M_\kappa)$ and that $\supp\nu$ is compact.

From this it follows that the $s$-dimensional Hausdorff density of $\nu$ exists and is positive on $\supp\nu$:
\[ \vartheta^{(\nu)}_s(x) = \lim_{r \to 0} \frac{\nu(B_r(x))}{\omega_s r^s} = \lim_{r \to 0} \frac{\nu(B_r(x))}{\vol_\kappa^s(r)} = \sup_{r > 0} \frac{\nu(B_r(x))}{\vol_K^s(r)}. \]
As above, $\underline \vartheta^{(\nu)}_s(x, \delta)$ is uniformly lower bounded on compact subsets. Indeed since the Bishop-Gromov inequality holds for arbitrary $R = d(x,x_0)$, arguing as above we obtain
\begin{equation}\label{eq:bgdenslb} \underline \vartheta_s(x, \delta) \geq \frac{\nu(B_{R+\delta}(x))}{\vol_\kappa^s(R+\delta)} \geq \frac{\nu(B_{\delta}(x_0))}{\vol_\kappa^s(R+\delta)} \geq C \left\{ \begin{matrix} R^s, & \kappa \geq 0; \\ e^{s \sqrt{-\kappa} R}, & \kappa < 0. \end{matrix} \right. \end{equation}
Furthermore, arguing as in \cite[Lem.~4.9, 4.10]{covgrow}, $\nu$ satisfies
\[ \nu(B_{R+R_0}(x_0) \setminus B_R(x_0)) \leq \frac{C_1+C_2 \sqrt{-\kappa} R}{R} \nu(B_{R+R_0}(x_0)) \leq \frac{C_1+C_2 \sqrt{-\kappa} R}{R} C \vol_\kappa^s(R+R_0). \]
We can then apply Corollary \ref{cor:volgrowstab} with
\[ V(R) := \left\{ \begin{matrix} C_0 R^{s-1}, & \kappa \geq 0 \\ C_\kappa e^{s \sqrt{-\kappa} R}, & \kappa < 0 \end{matrix} \right. \]
in order to obtain the $(p,s)$-quantizability conditions
\begin{align*}
(\kappa = 0) \qquad & \int_X \underbrace{d(x,x_0)^{p+(\alpha+1)p/s} d(x,x_0)^{(s-1)p/s} d(x,x_0)^{s p/s}}_{d(x,x_0)^{3p+\alpha p/s}} \di\mu(x) < \infty; \\
(\kappa < 0) \qquad & \int_X d(x,x_0)^{p+(\alpha+1)p/s} e^{2 \sqrt{-\kappa} p d(x,x_0)} \di\mu(x) < \infty.
\end{align*}
In either case, Theorem \ref{thm:densbd} implies that 
\[ C_1 \int \rho^{\frac{s}{s+p}} \vartheta_s(x)^{-\frac{p}{s+p}} \di \nu \leq \underline Q_{p,s}(\mu)^{p^\prime} \leq \overline Q_{p,s}(\mu)^{p^\prime} \leq C_2 \int \rho^{\frac{s}{s+p}} \vartheta_s(x)^{-\frac{p}{s+p}} \di \nu. \]
We contrast these conditions with the covering growth estimates given in \cite[Sect.~4.2]{covgrow}, which are presented for Riemannian manifolds but depend solely on the Bishop--Gromov theorem, hence in principle extend also to curvature-dimension spaces. These estimates, which are based on controlling the volumes of thin annuli, give instead the following integrability conditions:
\begin{align*}
(\kappa = 0) \qquad & \int_X d(x,x_0)^{p+\alpha} \di\mu(x) < \infty; &
(\kappa < 0) \qquad \int_X e^{\frac{5\sqrt{-\kappa}}{2} p d(x,x_0)} \di\mu(x) < \infty.
\end{align*}
This is more general for $\kappa = 0$ but less general for $\kappa < 0$. Combining the two estimates, $(p,s)$-quantizability is guaranteed by the integrability of $d(x,x_0)^{p+\alpha}$ for $\kappa = 0$ resp.~$e^{(2+\veps) \sqrt{-\kappa} p d(x,x_0)}$ for $\kappa < 0$. If moreover $\nu$ is known to satisfy $\nu(B_r(x)) \geq \vartheta r^s$ for all $x \in X$ and $r < \delta$, the bound \eqref{eq:bgdenslb} will not be necessary, relaxing the latter to $e^{(1+\veps) \sqrt{-\kappa} p d(x,x_0)}$.
\end{example}
}

\section{Zador's theorem for rectifiable measures}\label{sect:zadorrect}

We now apply the general formalism given in Section \ref{sect:quantcoeff} in order to extend Zador's theorem to rectifiable measures. 
We first review a few definitions and facts from geometric measure theory.

We use the following convention for the $s$-dimensional Hausdorff measure $\cH^s$, $s > 0$:
\[ \cH^s(A) := \sup_{\delta > 0} \cH^s_\delta(A); \quad \cH^s_\delta(A) := \inf\left\{ \sum_{i=1}^\infty \omega_s \left(\frac{\diam B_i}{2}\right)^s \mid A \subseteq \bigcup_{i=1}^\infty B_i, \ \diam B_i < \delta \right\}. \]
The $\omega_s$ factor in the definition ensures that for integer $m \geq 1$, $\cH^m$ coincides exactly with the Lebesgue measure $\cL^m$ on $\R^m$.
Given $F \colon X \to Y$ $\lambda$-Lipschitz and $A \subseteq X$ $\cH^s$-measurable, the image $F(A)$ is also $\cH^s$-measurable with $\cH^s(F(A)) \leq \lambda^s \cH^s(A)$, and likewise if $F$ is $(\alpha,\beta)$-bi-Lipschitz, $\alpha^s \cH^s(A) \leq \cH^s(F(A)) \leq \beta^s \cH^s(A)$. 

{We fix the following nomenclature for rectifiable sets and measures on metric spaces, cf.~\cite{kirch94, Ambrosio:2000aa}:}

\begin{defin}[Rectifiable sets and measures]
{
Let $m$ be a positive integer. 

A subset $E$ of a metric space $X$ is said to be 
\begin{itemize}
\item \emph{$m$-rectifiable} if $E = F(A)$ for a Borel set $A \subseteq \R^m$ and a Lipschitz map $F \colon A \to X$; 
\item \emph{countably $m$-rectifiable} if 
$\cH^m\left(E \setminus \bigcup_{i=1}^\infty F_i(A_i) \right) = 0$
for a countable family of nonempty Borel subsets $A_i \subseteq \R^m$ and Lipschitz maps $F_i \colon A_i \to X$.
\end{itemize}
A Borel measure $\mu$ on $X$ is said to be \emph{(countably) $m$-rectifiable} if likewise
$\mu\left(X \setminus \bigcup_{i=1}^\infty F_i(A_i) \right) = 0$
for a countable family of nonempty Borel subsets $A_i \subseteq \R^m$ and Lipschitz maps $F_i \colon A_i \to X$.
}
\end{defin}

{Note that contrary to usual convention, we do not require rectifiable measures to be absolutely continuous with respect to $\cH^m$.
Observe that $\cH^m$ restricted to sets of the form $\bigcup_{i=1}^\infty F_i(A_i)$ is always $\sigma$-finite, hence the Radon-Nikodym theorem is applicable to the absolutely continuous components of $m$-rectifiable measures. 
}

When $X$ is a Banach space, in particular $\R^d$, we can assume wlog that each Lipschitz map $F_i$ is defined on the entirety of $\R^m$. Otherwise, one may always embed $X$ isometrically into a Banach space $\bar X$:

\begin{recall*}[Kuratowski embedding]
Let $X$ be a metric space, $x_0 \in X$ arbitrary. Then the map $\Phi \colon X \to C_b(X)$ given by $\Phi(x)(y) := d(x,y) - d(x_0,y)$ is an isometric embedding.
When $X$ is separable, fixing a countable dense subset $\{y_k\}_k \subset X$, the map $\phi \colon X \to \ell^\infty$ given by $\phi(x)_k := \Phi(x)(y_k)$ is likewise an isometric embedding.
\end{recall*}

\begin{recall*}[Lipschitz extension theorems]
The McShane extension theorem states that, given an arbitrary subset $A$ of a metric space $X$, any $L$-Lipschitz map $f \colon A \to \R$ admits an $L$-Lipschitz extension $\bar f \colon X \to \R$, given e.g.~by
\[ \bar f(x) = \inf \{ f(y) + L d(x,y) \mid y \in A \}, \quad x \in X. \]
From this it follows that $L$-Lipschitz maps $A \to \R^d$ extend to $\sqrt{d} L$-Lipschitz maps, and $L$-Lipschitz maps $A \to \ell^\infty$ extend to $L$-Lipschitz maps.

The Kirszbraun extension theorem \cite{Kirszbraun1934}, \cite[Thm.~2.10.43]{Federer1996} states that, if $X$ and $Y$ are Hilbert spaces, any $L$-Lipschitz map $f \colon A \to Y$ defined on an arbitrary subset $A \subset X$ admits an $L$-Lipschitz extension $\bar f \colon X \to Y$. Kirszbraun's theorem also generalizes to geodesic spaces of bounded sectional curvature \cite{Lang:1997aa}.

If instead $X$ and $Y$ are Banach spaces and $\dim X = m < \infty$, there exists a dimensional constant $\alpha_m$ such that any $L$-Lipschitz map $f \colon A \to Y$ with $A \subset X$ admits an $\alpha_m L$-Lipschitz extension $\bar f \colon X \to Y$; cf.~\cite{Johnson:1986aa}.
\end{recall*}

The existence of Lipschitz extensions will guarantee that arbitrary points on $X$, not necessarily on a Lipschitz image, can be mapped back to $\R^m$ in order to reduce to Zador's theorem on $\R^m$. For the absolutely continuous case, we will also employ the following partitioning property:

\begin{recall*}[Bi-Lipschitz parametrization]
Given $E \subseteq X$ $m$-rectifiable and $\lambda > 1$ arbitrary, there exist a countable family of compact sets $K_i \subset \R^m$, norms $\rho_i$ on $\R^m$, and $(\lambda^{-1},\lambda)$-bi-Lipschitz maps $F_i \colon (K_i, \rho_i) \to E$ such that the images $F_i(K_i)$ are pairwise disjoint and cover $E$ $\cH^m$-a.e. 

When $m = 1$ or $X = \R^d$ equipped with the Euclidean norm (or by extension, a Riemannian manifold), each $\rho_i$ can be assumed wlog to be the Euclidean norm, the former case following simply from the fact that all norms on $\R$ are isometric. 
The Euclidean case is proven in \cite[Lem.~3.2.18]{Federer1996}; the metric case can be deduced from \cite[Lem.~4]{kirch94} with the same argument as in \cite[Lem.~4.1]{AmbKir00}.
\end{recall*}

In order to obtain equality in Zador's theorem, we will need each $\rho_i$ to be the Euclidean norm and the statement to hold for arbitrary $\lambda$. For general metric spaces, since all norms on $\R^m$ are equivalent, it still holds that $F_i \colon K_i \to E$ are $(\lambda_0^{-1}, \lambda_0)$-bi-Lipschitz for \emph{some} $\lambda_0 > 1$, but this will only yield lower and upper bounds on quantization coefficients. This is the main obstacle to extending Zador's theorem to more general domains. We note that the Euclidean result does indeed hold for metric measure spaces satisfying Riemannian curvature-dimension conditions \cite{Mondino2019, Brue:2021aa}.

{
\begin{remark}[Existence of densities]\label{rem:densbdrect}
If $E$ is countably $m$-rectifiable with $0 < \cH^s(E) < \infty$, the $m$-dimensional Hausdorff density of $\cH^m|_E$ exists and equals $1$ for $\cH^m$-a.e.~$x \in E$; cf.~\cite[Thm.~16.2]{Mattila_1995} for $X = \R^d$ and \cite[Thm.~9]{kirch94} for the general case. 

Thus already from Theorem \ref{thm:densbd}, we can deduce that any $(p,m)$-quantizable $m$-rectifiable measure $\mu \ll \cH^m$ with $\mu(E^c) = 0$ satisfies
\[ C_1 \int_E \rho^{\frac{m}{m+p}} \di\cH^m \leq \underline Q_{p,m}(\mu)^{\frac{m+p}{mp}} \leq \overline Q_{p,m}(\mu)^{\frac{m+p}{mp}} \leq C_2 \int_E \rho^{\frac{m}{m+p}} \di\cH^m. \]
We can further eliminate the condition $\cH^m(E) < \infty$ by subdividing $E$ into countably many subsets with finite measure and invoking Proposition \ref{prop:cntadd}, so that the above weak form of Zador's theorem holds for a.c.~rectifiable measures on arbitrary metric spaces. Theorems \ref{thm:zadorrect} and \ref{thm:zadorrect1} improve on this result by demonstrating equality and allowing for singular components, but at the expense of restrictions on the ambient space or the dimension $m$.
\end{remark}
}

We now proceed with the proofs of Zador's theorem for $m$-rectifiable measures on $\R^d$ and $1$-rectifiable measures on metric spaces. In both settings, the absolutely continuous case will follow from the application of Proposition \ref{prop:cntadd} to bi-Lipschitz parametrizations and reduction to Zador's theorem on $\R^m$, while the singular case will require different arguments specific to the setting. 

Given an $m$-rectifiable measure $\mu = \rho \cH^m + \mu^\bot$ on $X$, $\mu^\bot \bot \cH^m$, we set
\[ \F_{p,m}[\mu] := \left( \int_X \rho^{\frac{m}{m+p}} \di\cH^m \right)^{\frac{m+p}{mp}}. \]
By construction, the $p^\prime = \frac{mp}{m+p}$th power of $\F_{p,m}$ is countably subadditive, since the function $t \mapsto t^{\frac{m}{m+p}}$ is subadditive on $[0,\infty]$, and additive for mutually singular measures. 
We note that $\F_{p,m}$ is preserved along isometric embeddings; more generally, we have the following statement:

\begin{lemma}\label{lem:fpmpush}
Let $F \colon X \to Y$ be $(\alpha, \beta)$-bi-Lipschitz, $0 < \alpha \leq \beta < \infty$. Then for any finite $m$-rectifiable Borel measure $\nu$ on $X$,
\[ \alpha \F_{p,m}[\nu] \leq \F_{p,m}[F_\# \nu] \leq \beta \F_{p,m}[\nu]. \]
\end{lemma}
\begin{proof}
Observe first that $F$ preserves the Hahn-Jordan decomposition: if $\nu$ decomposes as $\nu = \nu_{ac} + \nu_s$, $\nu_{ac} \ll \cH^m$ and $\nu_s \bot \cH^m$, then $F_\# \nu_{ac} \ll \cH^m$ and $F_\# \nu_s \bot \cH^m$ as well. Indeed for any $B \subseteq Y$ such that $\cH^m(B) = 0$, $\cH^m(F^{-1}(B)) \leq \alpha^m \cH^m(B) = 0$ thus $(F_\# \nu_{ac})(B) = \nu_{ac}(F^{-1}(B)) = 0$. Conversely, if $\nu_s(X \setminus A) = \cH^m(A) = 0$ for some $A \subseteq X$, then $\cH^m(F(A)) \leq \beta^m \cH^m(A) = 0$ and $(F_\# \nu_s)(Y \setminus F(A)) = \nu_s(X \setminus F^{-1}(F(A))) = 0$.

We can then assume wlog that both $\nu$ and $\mu := F_\# \nu$ are absolutely continuous with respect to $\cH^m$. First consider the case $\nu = \cH^m|_A$; we show that $\beta^{-m} \cH^m|_{F(A)} \leq \mu \leq \alpha^{-m} \cH^m|_{F(A)}$. Indeed for $B \subseteq Y$ Borel, setting $C = F^{-1}(B) \cap A \subseteq X$, we have that $F(C) = B \cap F(A)$ and
\[ \mu(B) = \nu(F^{-1}(B)) = \cH^m(C) \in [\beta^{-m}, \alpha^{-m}] \cH^m(F(C)). \]
Then denoting by $\rho$ the density of $\mu$, we have that $\beta^{-m} \One_{F(A)} \leq \rho \leq \alpha^{-m} \One_{F(A)}$, hence
\[ \F_{p,m}[\mu]^{p^\prime} = \int_Y \rho^{-\frac{p}{m+p}} \di\mu \leq \beta^{\frac{mp}{m+p}} \int_Y \di\mu = \beta^{p^\prime} \int_X \di\nu = \beta^{p^\prime} \F_{p,m}[\nu]^{p^\prime} \]
and likewise for the converse inequality.

We can then deduce the statement for the case $\nu = \sum_{i=1}^\infty \lambda_i \cH^m|_{A_i}$ with $(A_i)_i$ disjoint, noting that the images $F(A_i)$ are also disjoint by the injectivity of $F$, and observing that $\F_{p,m}[\nu]^{p^\prime} = \sum_{i=1}^\infty \lambda_i^{\frac{m}{m+p}} \F_{p,m}[\cH^m|_{A_i}]^{p^\prime}$ and likewise for $\mu$. 

The statement for general $\nu \ll \cH^m$ then follows by approximating $\nu$ from below by measures of the above form, noting that for any monotonically convergent sequence $\nu_k \nearrow \nu$, $F_\# \nu_k \nearrow F_\# \nu$ as well and the convergence of $\F_{p,m}$ follows from the monotone convergence theorem.
\end{proof}

\subsection{Proof of Theorem \ref{thm:zadorrect}}\label{sect:zadorrectproof}

We now prove Zador's theorem for $m$-rectifiable measures on $\R^d$.
We first consider measures $\mu \ll \cH^m$ supported on the image of a single bi-Lipschitz map $F \colon K \to \R^d$. 

\begin{lemma}\label{lem:zadorrectineq}
Let $K \subset \R^m$ be compact, $F \colon \R^m \to \R^d$ be Lipschitz and $(\lambda^{-1}, \lambda)$-bi-Lipschitz on $K$. 

Let $\mu \ll \cH^m$ be supported on $F(K)$. Then 
\[ \lambda^{-2} C_{p,m} \F_{p,m}[\mu] \leq \underline Q_{p,m}(\mu) \leq \overline Q_{p,m}(\mu) \leq \lambda^2 C_{p,m} \F_{p,m}[\mu]. \]
\end{lemma}
\begin{proof}
Since $F$ is bi-Lipschitz on $K$, $F \colon K \to F(K)$ is invertible and $F^{-1} \colon F(K) \to K$ is also $(\lambda^{-1}, \lambda)$-bi-Lipschitz. By the Kirszbraun extension theorem, there exist $\lambda$-Lipschitz extensions $G \colon \R^m \to \R^d$ and $H \colon \R^d \to \R^m$ of $F|_K$ and $F^{-1}|_{F(K)}$ respectively. 

Consider the measure $\nu = H_\# \mu$ on $\R^m$. Then $\mu = G_\# \nu$: since the map $G \circ H$ restricts to the identity on $F(K)$, for each $A \subseteq \R^d$ we have
\[ (G_\# \nu)(A) = ((G \circ H)_\# \mu)(A) = \mu((G \circ H)^{-1}(A)) = \mu((G \circ H)^{-1}(A) \cap F(K)) = \mu(A \cap F(K)) = \mu(A). \]
Thus $\overline Q_{p,m}(\mu) \leq \lambda \overline Q_{p,m}(\nu)$ and $\underline Q_{p,m}(\nu) \leq \lambda \underline Q_{p,m}(\mu)$ {by Proposition \ref{prop:coeffprops} (i)}.
Moreover, $\supp\nu \subseteq H(\supp\mu) \subseteq H(F(K)) = K$ compact, thus by Zador's theorem, 
\[ Q_{p,m}(\nu) = C_{p,m} \F_{p,m}[\nu] \quad \implies \quad \lambda^{-1} C_{p,m} \F_{p,m}[\nu] \leq {\underline{Q}}_{p,m}(\mu) \leq \overline{{Q}}_{p,m}(\mu) \leq \lambda C_{p,m} \F_{p,m}[\nu]. \]
The statement then follows from Proposition \ref{prop:psstabprops} (iii) and Lemma \ref{lem:fpmpush}.
\end{proof}

Here Lipschitz extensions play the following role: the extension $G \colon \R^m \to \R^d$ allows us to upper bound $\overline Q_{p,m}(\mu)$ by mapping quantizers on $\R^m$ onto $\R^d$, and the extension $H \colon \R^d \to \R^m$ yields the converse lower bound by projecting arbitrary quantizers of $\mu$ on $\R^d$ down to $\R^m$. 
For $1$-rectifiable measures on metric spaces, we will still have the latter extension by McShane, but not the former extension.

The singular case instead follows from the connection between Minkowski contents and $\infty$-quantization coefficients, cf.~Appendix \ref{app:conc}:

\begin{lemma}\label{lem:rectsing}
Let $\mu \in \M^p_+(\R^d)$ be countably $m$-rectifiable and $(p,m)$-quantizable, such that $\mu \bot \cH^m$. Then 
\[ Q_{p,m}(\mu) = 0. \]
\end{lemma}
\begin{proof}
First suppose $\mu(\R^d \setminus F(K)) = 0$ with $F \colon \R^m \to \R^d$ Lipschitz, $K \subset \R^m$ compact. By assumption, there exists $B \subseteq F(K)$ Borel such that $\mu(\R^d \setminus B) = 0$ and $\cH^m(B) = 0$. 

Let $E$ be an arbitrary compact subset of $B \subseteq F(K)$. Then by Kneser's Theorem \cite{Kneser:1955aa}, \cite[Thm.~3.2.39]{Federer1996}, we have that $\cM^m(E) = \cH^m(E) = 0$. Applying Lemma \ref{lem:packub} with the locally finite measure $\nu = \cL^d$, which satisfies $\cL^d(B_r(x)) = \omega_d r^d$ for all $x$ and $r$, then yields
\[ \overline{Q}_{\infty,m}(E)^m \leq \frac{2^m \omega_{d-m}}{\omega_d} \overline\cM^m(E) = 0. \]
In particular, $\overline{Q}^{(\mu)}_{p,m}(E) \leq \mu(E)^{1/p} \overline{Q}_{\infty,m}(E) = 0$ {by Proposition \ref{prop:coeffprops} (iii). This implies that} $Q_{p,m}(\mu) = Q_{p,m}^{(\mu)}(B) = 0$ by Proposition \ref{prop:qpscont}.

If instead $\mu(\R^d \setminus F(\R^m)) = 0$ for some $F \colon \R^m \to \R^d$ Lipschitz, we can simply restrict $\mu$ to $F(\overline{B_n(0)})$ for each $n \in \N$ and invoke Proposition \ref{prop:qpscont} again. 

{
Finally, if $\mu$ is countably $m$-rectifiable, say $\mu(\R^d \setminus \bigcup_{i=1}^\infty F_i(\R^m)) = 0$, then we can write 
\[ \mu = \sum_{i=1}^\infty \mu_i \quad \text{where} \quad \mu_i := \mu|_{F_i(\R^m) \setminus \bigcup_{j < i} F_j(\R^m)}. \]
Then each $Q_{p,m}(\mu_i) = 0$ by the above argument, and since $\mu$ is $(p,m)$-quantizable, Proposition \ref{prop:cntadd} (ii) implies that
\[ \overline Q_{p,m}(\mu)^{p^\prime} \leq \sum_{i=1}^\infty \overline Q_{p,m}(\mu_i)^{p^\prime} = 0. \qedhere \]
}
\end{proof}

The general case then follows by taking bi-Lipschitz parametrizations and applying Proposition \ref{prop:cntadd}:

\begin{proof}[Proof of Theorem \ref{thm:zadorrect}]
Let $\mu \in \M^p_+(\R^d)$ be $m$-rectifiable. By Lemma \ref{lem:rectsing}, we can assume wlog that $\mu \ll \cH^m$ with density $\rho$.

Let $E \subset \R^d$ be $\cH^m$-measurable and {countably} $m$-rectifiable such that $\mu(\R^d \setminus E) = 0$. 
Let $\lambda > 1$. By \cite[Lem.~3.2.18]{Federer1996}, there exist countably many Lipschitz maps $F_i \colon \R^m \to \R^d$ and compact subsets $K_i \subset \R^d$ such that each $F_i$ is $(\lambda^{-1}, \lambda)$-bi-Lipschitz on $K_i$, and the sets $E_i := F_i(K_i)$ are disjoint and cover $E$ $\cH^m$-a.e., thus also $\mu$-a.e. 

We can then write $\mu = \sum_{i=1}^\infty \mu|_{E_i}$. By Lemma \ref{lem:zadorrectineq}, we have
\[ \lambda^{-2} C_{p,m} \F_{p,m}[\mu|_{E_i}] \leq {\underline{Q}}_{p,m}^{(\mu)}(E_i) \leq \overline{{Q}}_{p,m}^{(\mu)}(E_i) \leq \lambda^2 C_{p,m} \F_{p,m}[\mu|_{E_i}]. \]
Then applying Proposition \ref{prop:cntadd} with
\[ \nu_1(A) := \lambda^{-2p^\prime} \F_{p,m}[\mu|_A]^{p^\prime} = \lambda^{-2p^\prime} \int_A \rho^{\frac{m}{m+p}} \di \cH^m \text{ and } \nu_2(A) := \lambda^{2p^\prime} \F_{p,m}[\mu|_A]^{p^\prime} = \lambda^{2p^\prime} \int_A \rho^{\frac{m}{m+p}} \di \cH^m, \]
we obtain that 
$\underline Q_{p,m}(\mu) \geq \lambda^{-2} C_{p,m} \F_{p,m}[\mu]$,
and if $\mu$ is $(p,m)$-quantizable, 
$\overline Q_{p,m}(\mu) \leq \lambda^2 C_{p,m} \F_{p,m}[\mu]$. 
We conclude by letting $\lambda \to 1^+$, and deducing the weak convergence statement directly from Theorem \ref{thm:stats}.
\end{proof}

\subsection{Proof of Theorem \ref{thm:zadorrect1}}\label{sect:zadorrect1proof}

We now prove Zador's theorem for $1$-rectifiable measures on Polish spaces.

We prove the singular case by recourse to arc length parametrizations of rectifiable curves. We first review the properties of rectifiable curves on metric spaces; for reference, see e.g.~\cite[Sect.~1.1]{AGS05}.

\begin{recall*}[Lipschitz curves]
Given a Lipschitz curve $\gamma \colon (a,b) \to X$, $-\infty \leq a < b \leq \infty$, the \emph{metric derivative}
\[ |\dot\gamma|(t) := \lim_{t^\prime \to t} \frac{d(\gamma(t^\prime), \gamma(t))}{|t^\prime - t|} \]
exists for $\cL^1$-a.e.~$t \in [a,b]$, with $|\dot\gamma| \in L^\infty[a,b]$, being upper bounded by $\Lip(\gamma)$. We note that $|\dot\gamma|$ coincides also with the metric differential as defined by Kirchheim \cite{kirch94}.

Any such curve admits an arc length reparametrization $s \colon [0,L] \to X$, $L = \int_a^b |\dot\gamma|(t) \di t \leq \infty$, such that $|\dot s| = 1$ $\cL^1$-a.e.~and $\gamma = s \circ \varphi$ for some increasing homeomorphism $\varphi \colon [a,b] \to [0,L]$. The curve $s$ satisfies
\[ d(s(t_0), s(t_1)) \leq \int_{t_0}^{t_1} |\dot s|(t) \di t = t_1 - t_0 \quad \text{for all } 0 \leq t_0 \leq t_1 \leq L, \]
and is thus $1$-Lipschitz.
\end{recall*}

Note that $1$-rectifiable subsets of $X$ are defined instead by images of Borel sets, not intervals; we remedy this by embedding $X$ into a Banach space $\bar X$.  

\begin{lemma}\label{lem:rectsing1}
Suppose $\mu \in \M^p_+(X)$ is $1$-rectifiable and $(p,1)$-quantizable, such that $\mu \bot \cH^1$. Then $Q_{p,1}(\mu) = 0$. 
\end{lemma}
\begin{proof}
{Applying Proposition \ref{prop:cntadd} (ii) as in Lemma \ref{lem:rectsing}, it suffices to prove the statement 
under the assumption that} $\mu(X \setminus \gamma(A)) = 0$ for some nonempty $A \subset \R$ Borel, $\gamma \colon A \to X$ Lipschitz. Subdividing $\R$ into intervals, we can also assume wlog that $A \subseteq [a,b]$ for some $-\infty < a < b < \infty$. 

Along an isometric embedding $\iota \colon X \to \bar X$ Banach, $\iota \circ \gamma$ admits a Lipschitz extension $\bar \gamma \colon [a,b] \to \bar X$. Take an arc length parametrization $s \colon [0,L] \to \bar X$ of $\bar \gamma$. If $L = 0$, then $\bar\gamma([a,b]) \supseteq \gamma(A)$ is a singleton, $\mu$ is a Dirac measure and $Q_{p,1}(\mu) = 0$, so we can assume wlog that $L > 0$.
Consider the finite Borel measure $\nu = s_\# (\cL^1|_{[0,L]})$ on $\bar X$. For any $x = s(t) \in s([0,L])$ and $0 < r \leq L/2$, we have
\[ \nu(B_r(x)) := \cL^1([0,L] \cap \bar s^{-1}(B_r(x))) \geq \cL^1([0,L] \cap (t-r, t+r)) \geq r, \]
since $ s$ is $1$-Lipschitz and the interval $(t-r,t+r) \cap [0,L]$ has length at least $r$.

Moreover, by the area formula (cf.~\cite[Thm.~7]{kirch94}), for any $B \subseteq \bar X$ Borel,
\[ \nu(B) = \int_{ s^{-1}(B)} \di t = \int_{ s^{-1}(B)} |\dot s|(t) \di t = \int_\Gamma \#( s^{-1}(x) \cap s^{-1}(B)) \di \cH^1(x) = \int_B \#( s^{-1}(x)) \di \cH^1(x). \]
Since $\nu(\bar X) = L < \infty$, it follows from the above that $\#( s^{-1}(x)) < \infty$ $\cH^1$-a.e., and therefore $\nu \ll \cH^1$.

By assumption, there exists $B \subseteq \gamma(A)$ Borel such that $\mu(X \setminus B) = \cH^1(B) = 0$. Then $\cH^1(\iota(B)) = 0$ thus $\nu(\iota(B)) = 0$ as well. Let $K \subseteq B$ be an arbitrary compact subset. Then Lemma \ref{lem:packub}, applied to $\nu$ with $s=m=1$, implies that
\[ \overline{Q}^{[X]}_{\infty,1}(K) \leq \overline{Q}^{[K]}_{\infty,1}(K) = \overline{Q}^{[\iota(K)]}_{\infty,1}(\iota(K)) \leq 2 \nu(\iota(K)) = 0, \]
thus $\overline{Q}^{(\mu)}_{p,m}(K) \leq \mu(K)^{1/p} \overline{Q}_{\infty,m}(K) = 0$ {by Proposition \ref{prop:coeffprops} (iii)}. Proposition \ref{prop:qpscont} then implies that $Q_{p,m}(\mu) = Q_{p,m}^{(\mu)}(B) = 0$ as well.
\end{proof}

\begin{remark*}
This argument does not extend to the $m$-dimensional case: constructing such a measure $\nu$ from a general Lipschitz map $F \colon \R^m \to \R^d$ instead of an arc length parametrization
will not always yield a measure a.c.~with respect to $\cH^m$, and the implication $\mu \bot \cH^m \implies \mu \bot \nu$ will not hold in general. For example, a constant map $F \colon \R^m \to \{x_0\} \subset \R^d$ will yield $\nu = C \delta_{x_0}$. In this case, we only have $\nu \ll \cH^1|_\Gamma$ thanks to the existence of arc length parametrizations, for which $|\dot s| = 1$ $\cL^1$-a.e.

For general $m$-rectifiable subsets $B \subseteq X$, we still have that $\underline\vartheta^{(\cH^m)}_m(x) > 0$ for $\cH^m$-a.e.~$x \in B$ by Preiss' theorem, but this does not provide any information for singular measures supported on sets of zero $\cH^m$-measure. This is what necessitates the application of Kneser's theorem in Lemma \ref{lem:rectsing}.
\end{remark*}

The argument for the absolutely continuous case is analogous to the previous section, taking countably many disjoint images of bi-Lipschitz maps $F_i \colon K_i \to X$. In this case, we can invoke the McShane extension theorem for $F_i^{-1} \colon F_i(K_i) \to K_i \subset \R$, but $F_i$ itself cannot be extended to a map $\R \to X$ in general. To make up for this, we will need the following technical statement about restricting the quantization problem on $\R$ to the support of the measure, which seems to be new and unique to dimension $1$ (or $2$, using hexagonal partitions):

\begin{lemma}\label{lem:zador1}
Let $K \subset \R$ be compact with $\cL^1(K) > 0$, $\nu \ll \cL^1|_K$ supported on $K$. Then the quantization coefficient $Q^{[K]}_{p,1}(\nu)$ of $\nu$ on the metric space $X = K$, i.e.
\[ Q^{[K]}_{p,1}(\nu) = \lim_{N \to \infty} N^p e^{[K]}_{N,p}(\nu), \quad e^{[K]}_{N,p}(\nu) := \inf \{ e_p(\nu; S) \mid S \subseteq K, \ \# S \leq N\}, \] 
also exists and coincides with that on $\R$:
\[ Q^{[K]}_{p,1}(\nu) = Q^{[\R]}_{p,1}(\nu) = C_{p,1} \F_{p,1}[\nu]. \]
\end{lemma}
\begin{proof}
The inequality $\underline Q^{[K]}_{p,1}(\nu) \geq Q^{[\R]}_{p,1}(\nu)$ is trivial; it suffices to prove that $\overline Q^{[K]}_{p,1}(\nu) \leq Q^{[\R]}_{p,1}(\nu)$. 

We prove the statement for $\nu = \cL^1|_K$. The general case follows by applying the same argument to $\nu = \cL^1|_{K^\prime}$ for $K^\prime \subseteq K$ compact, then generalizing to $\cL^1|_B$ for $B \subseteq K$ Borel by Proposition \ref{prop:qpscont}, and finally deducing the statement for general $\nu \ll \cL^1|_K$ using Proposition \ref{prop:genzadorineq}, noting that any measure $\nu$ supported on $K \subset \R$ compact is $(p,1)$-quantizable on $K$ since $\overline Q^{[K]}_{\infty,1}(K) \leq C \cL^1(K) < \infty$ by Lemma \ref{lem:packub} applied to $\cL^1$.

Up to scaling, we may also assume wlog that $K \subseteq [0,1]$. For each $N \in \N$, let $\cI_N = \{[\frac{k-1}{N}, \frac{k}{N}]\}_{k = 1}^N$ be the uniform interval partition of $[0,1]$ into $N$ subintervals, and set
\[ \cI_N(K) := \{ I \in \cI_N \mid |K \cap I| > 0 \}. \]
Since every such $K \cap I$ is compact and nonempty, we can take quantizers of the following form:
\[ S_N := \{ a_I \mid I \in \cI_N(K) \}; \quad a_I \in \argmin_{a \in K \cap I} \int_I |x-a|^p \di x. \]
The key observation is to control $\int_I |x - a_I|^p \di x$ based on $\frac{|K \cap I|}{|I|}$. Note firstly that $K \subseteq \bigcup \cI_N(K) \subseteq K^{1/N}$, with $\# S_N = \# \cI_N(K) = N |\bigcup \cI_N(K)|$, hence
\[ |K| \leq \frac{\# S_N}{N} \leq |K^{1/N}|. \]
Since $K$ is compact, $|K^{1/N}| \searrow |K|$ as $N \to \infty$, thus $\frac{\# S_N}{N} \to |K|$ as well.

For the unit interval $[0,1]$, the function $[0,1] \ni a \mapsto \int_0^1 |x-a|^p \di x$ is evidently continuous, and attains a maximum of $\frac{1}{p+1}$ at $a \in \{0,1\}$ and the minimum $\frac{1}{2^p(p+1)} = C_{p,1}^p$ (cf.~\cite[Ex.~5.5]{quantbook}) at $a = \frac{1}{2}$. For any interval $I = [a,b]$ with midpoint $m_I := \frac{a+b}{2}$, rescaling then gives
\[ C_{p,1}^p |I|^{p+1} = \int_I |x-m_I|^p \di x \leq \max_{a \in I} \int_I |x-a|^p \di x = \frac{1}{p+1} |I|^{p+1}, \]
and for any $\veps > 0$, there exists $\delta > 0$ (wlog $\delta < \frac{1}{2}$) such that 
\[ |a - m_I| < \delta |I| \quad \implies \quad \int_I |x-a|^p \di x < (C_{p,1}+\veps)^p |I|^{p+1}. \]
Now for each $I \in \cI_N(K)$, since $B_{\delta |I|}(m_I) \subset I$ has measure $2 \delta |I|$, $K \cap B_{\delta |I|}(m_I)$ must be nonempty whenever $|K \cap I| > (1-2\delta) |I|$. In that case,
\[ \int_{K \cap I} |x-a_I|^p \di x \leq \int_I |x-a_I|^p \di x = \min_{a \in K \cap I} \int_I |x-a|^p \di x < (C_{p,1}+\veps)^p |I|^{p+1}, \]
and otherwise,
\[ \int_{K \cap I} |x-a_I|^p \di x \leq \max_{a \in I} \int_I |x-a|^p \di x = \frac{1}{p+1} |I|^{p+1}. \]
Note also that $|I| = \frac{1}{N}$ for each $I \in \cI_N(K)$.
Defining $\rho_N := \sum_{I \in \cI_N(K)} \frac{|K \cap I|}{|I|} \One_I$, we can then bound
\begin{align*} 
e_p(\nu; S_N)^p & \leq \sum_{I \in \cI_N(K)} \int_{K \cap I} |x-a_I|^p \di x 
\leq (C_{p,1}+\veps)^p N^{-p} \Big( \sum_{\substack{I \in \cI_N(K) \\ \frac{|K \cap I|}{|I|} > 1 - 2\delta }}  |I| \Big) + \frac{1}{p+1} \Big( \sum_{\substack{I \in \cI_N(K) \\ \frac{|K \cap I|}{|I|} \leq 1 - 2\delta }}  |I| \Big)
\\ & \leq (C_{p,1}+\veps)^p N^{-p} |K^{1/N} \cap \{ \rho_N > 1 - 2\delta\}| + \frac{1}{p+1} N^{-p} |K^{1/N} \cap \{ \rho_N \leq 1 - 2\delta\}|.
\end{align*}
We now show that $\rho_N \to \One_K$ in measure. On $L^1[0,1]$, we can define the local averaging operators
\[ J_N f := \sum_{I \in \cI} \left( \frac{1}{|I|} \int_I f \di x \right) \One_I, \quad f \in L^1[0,1], \]
so that $\rho_N = J_N (\One_K)$. Since $\|J_N f\|_{L^1} \leq \|f\|_{L^1}$ uniformly in $N$ and $f$, and $J_N f \to f$ uniformly for $f \in C[0,1]$, it follows that $J_N f \to f$ in $L^1$, thus in measure, for each $f \in L^1[0,1]$, in particular for $f = \One_K$. 

This implies that $|K \cap \{ \rho_N > 1 - 2 \delta\}| \to |K|$ as $N \to \infty$, and since $|K^{1/N} \setminus K| \to 0$ as well, we deduce
\begin{align*} 
\lim_{N \to \infty} (\# S_N)^p e_p(\nu; S_N)^p & \leq \left( \lim_{N \to \infty} \frac{\# S_N}{N} \right)^p \left[ (C_{p,1}+\veps)^p |K| \right] = (C_{p,1}+\veps)^p |K|^{p+1}. 
\end{align*}
We can then let $\veps \to 0$ in the right-hand side. We now conclude by taking appropriate $N_k \in \N$ such that $\# S_{N_k} \sim k$ for each $k \in \N$. For each $\veps > 0$, we have $|K^{1/N}| \leq |K| + \veps$ thus $|K| N \leq \# S_N \leq (|K|+\veps) N$ for each $N$ sufficiently large. Taking $N_k = \lfloor \frac{k}{|K|+\veps} \rfloor$ then yields $\# S_{N_k} \leq k$ for $k$ sufficiently large, with 
\[ \lim_{k \to \infty} \frac{\# S_{N_k}}{k} = \lim_{k \to \infty} \frac{\# S_{N_k}}{N_k} \lim_{k \to \infty} \frac{N_k}{k} = \frac{|K|}{|K|+\veps}. \]
Then for each $k$, $S_{N_k} \subset K$ is a feasible set of quantizers with at most $k$ elements, yielding
\[ \overline Q^{[K]}_{p,1}(\nu) \leq \lim_{k \to \infty} k^p e_p(\nu; S_{N_k})^p \leq \left( \lim_{k \to \infty} \frac{\# S_{N_k}}{k} \right)^p \lim_{k \to \infty} (\# S_{N_k})^p e_p(\nu; S_{N_k})^p \leq C_{p,1}^p |K| (|K|+\veps)^p. \]
Letting $\veps \to 0$ again finally gives the desired inequality. \qedhere
\end{proof}

We can then deduce the following analogue of Lemma \ref{lem:zadorrectineq}:

\begin{lemma}\label{lem:zadorrectineq1}
Let $K \subset \R$ be compact, $F \colon K \to X$ $(\lambda^{-1}, \lambda)$-bi-Lipschitz, $\mu \ll \cH^1$ supported on $F(K)$. Then 
\[ \lambda^{-2} C_{p,1} \F_{p,1}[\mu] \leq \underline Q_{p,1}(\mu) \leq \overline Q_{p,1}(\mu) \leq \lambda^2 C_{p,1} \F_{p,1}[\mu]. \]
\end{lemma}
\begin{proof}
We can assume that $\cL^1(K) > 0$ since otherwise $\cH^1(F(K)) = 0$ hence $\mu = 0$.
Since $F$ is bi-Lipschitz on $K$, $F \colon K \to F(K)$ is invertible and $F^{-1}|_{F(K)} \colon F(K) \to K \subset \R$ is also $(\lambda^{-1}, \lambda)$-bi-Lipschitz. By the McShane extension theorem, there exists a $\lambda$-Lipschitz extension $H \colon X \to \R$ of $F^{-1}|_{F(K)}$, so that $H$ sends $F(K)$ to $K$ and the composition $F \circ H$ is well-defined and equal to the identity on $F(K)$. Then the measure $\nu := H_\# \mu$ is supported on $K$, with $\mu = F_\# (\nu|_K)$. Indeed, for any $B \subseteq X$ Borel,
\[ \nu(F^{-1}(B)) = \mu(H^{-1}(F^{-1}(B))) = \mu(F(K) \cap (F \circ H)^{-1}(B)) = \mu(F(K) \cap B) = \mu(B) \]
since for any $x \in F(K)$, $(F \circ H)(x) \in B$ iff $x \in B$. 

{By Proposition \ref{prop:coeffprops} (i), we have that} $Q^{[\R]}_{p,1}(\nu) \leq \lambda \underline Q_{p,1}(\mu)$ {since} $H \colon X \to \R$ is $\lambda$-Lipschitz, and conversely $\overline{Q}_{p,1}(\mu) \leq \lambda Q^{[K]}_{p,1}(\nu)$ since $F \colon K \to X$ is $\lambda$-Lipschitz. Since $\supp\nu \subseteq K$, Lemma \ref{lem:zador1} implies that
\[ Q^{[K]}_{p,1}(\nu) = Q^{[\R]}_{p,1}(\nu) = C_{p,1} \F_{p,1}[\nu]. \]
The desired statement then follows from Lemma \ref{lem:fpmpush}.
\end{proof}

Zador's theorem for the absolutely continuous case follows again from bi-Lipschitz parametrizations.

\begin{proof}[Proof of Theorem \ref{thm:zadorrect1}]
Let $\mu \in \M^p_+(X)$ be countably $1$-rectifiable, wlog $\mu \ll \cH^1$ by Lemma \ref{lem:rectsing1}.

Let $E \subset X$ be $\cH^1$-measurable and $1$-rectifiable such that $\mu(X \setminus E) = 0$. 
Then again for $\lambda > 1$ arbitrary, there exist countably many compact subsets $K_i \subset \R$ and $(\lambda^{-1},\lambda)$-bi-Lipschitz maps $F_i \colon K_i \to X$ such that the sets $E_i := F_i(K_i)$ are disjoint and cover $E$ $\cH^1$-a.e., so that $\mu = \sum_{i=1}^\infty \mu|_{E_i}$. 
On each $E_i$, Lemma \ref{lem:zadorrectineq1} yields
\[ \lambda^{-2} C_{p,1} \F_{p,1}[\mu|_{E_i}] \leq {\underline{Q}}_{p,1}^{(\mu)}(E_i) \leq \overline{{Q}}_{p,1}^{(\mu)}(E_i) \leq \lambda^2 C_{p,1} \F_{p,1}[\mu|_{E_i}]. \]
From this, the statement follows analogously from Proposition \ref{prop:cntadd} and Theorem \ref{thm:stats}.
\end{proof}

\subsection{Conditions for $(p,m)$-quantizability of rectifiable measures}\label{sect:psstabrect}

{
Unlike the settings presented in Section \ref{sect:examples}, it is difficult to formulate general conditions for the $(p,m)$-quantizability of rectifiable measures. In this section we present a few domain-specific conditions as well as counterexamples.

Suppose $\mu(X \setminus E) = 0$ for $E$ countably $m$-rectifiable. Since the $m$-dimensional density of $\nu = \cH^m|_E$ exists and equals $1$ for $\cH^m$-a.e.~$x \in E$ (cf.~\cite[Thm.~9]{kirch94} on metric spaces), the lower densities $\underline\vartheta^{(\nu)}_m(\cdot, \delta)$ must also be positive for arbitrary $\delta > 0$. 
Hence in principle, for measures $\mu \ll \cH^m$ one could deduce $(p,m)$-quantizability from Corollary \ref{cor:volgrowstab}, but this would not be useful in general due to the intractability of controlling lower densities. 
Nonetheless, if $(E, \cH^m)$ falls under any of the example categories treated in Section \ref{sect:examples} (e.g.~if $E$ is uniformly rectifiable hence by definition Ahlfors regular), the same conditions presented there for $(p,m)$-quantizability on $E$ even for singular measures will carry over to the ambient space $X$ thanks to Proposition \ref{prop:psstabprops} (iii). 

Otherwise, one general condition we can impose on $E$ compact is the finiteness of its upper Minkowski content, defined with respect to a higher-dimensional measure $\nu$.
We recall the definition of Minkowski contents on metric measure spaces:

\begin{defin}\label{def:mincon}
Let $(X, \nu)$ be a metric measure space, $0 < m \leq s < \infty$. The \emph{$m$-dimensional lower} and \emph{upper Minkowski contents} of a compact set $A \subseteq X$ are
\[ \underline \cM^m_{(\nu)}(A) := \liminf_{r \to 0^+} \frac{\nu(A^r)}{\omega_{s-m} r^{s-m}}, \quad \overline \cM^m_{(\nu)}(A) := \limsup_{r \to 0^+} \frac{\nu(A^r)}{\omega_{s-m} r^{s-m}}. \]
If the limit exists, it is denoted by $\cM^m_{(\nu)}$.
\end{defin}

This corresponds to the usual definition of the Minkowski content $\cM^m$ on $\R^d$ when $s = d$, $\nu = \cL^d$ and $m$ is an integer.

\begin{remark*}
Since $\nu$ is locally finite, $\nu(A^r) < \infty$ for $r$ sufficiently small. Indeed, for each $x \in X$, there exists $r_x > 0$ such that $\nu(B_{r_x}(x)) < \infty$, and the open cover $\{B_{r_x}(x)\}_{x \in A}$ of $A$ admits a finite open subcover $\{B_{r_i}(x_i)\}_{i=1}^m$. Then by the Lebesgue number lemma, there exists $\delta > 0$ such that $A^\delta \subseteq \bigcup_{i=1}^m B_{r_i}(x_i)$, which has finite $\nu$-measure. 

In particular, for $m = s$, since $A = \bigcap_{r > 0} A^r$ as a monotone limit, the limit $\lim_{r \to 0^+} \nu(A^r)$ exists and coincides with $\nu(A)$. 
\end{remark*}
}

{We deduce that any measure supported on a compact set with finite upper Minkowski content is $(p,m)$-quantizable, as long as the $s$-dimensional density of $\nu$ is bounded from below:}

\begin{prop}\label{prop:minconstab}
Let $(X, \nu)$ be a metric measure space, $E \subseteq X$ compact.
Suppose that there exist dimensions $0 < m \leq s$ such that $\overline\cM^m_{(\nu)}(E) < \infty$ and constants $\delta,\lambda > 0$ such that $\underline \vartheta^{(\nu)}_s(\cdot, \delta) \geq \lambda > 0$ on $E$.

Then any measure $\mu \in \M^\infty_+(X)$ supported on $E$ is $(p,m)$-quantizable for all $p < \infty$.
\end{prop}
\begin{proof}
Applying Lemma \ref{lem:packub} to $E$ yields
\[ \overline{Q}_{\infty,m}(E)^m \leq \frac{2^m \omega_{s-m}}{\lambda} \overline\cM^m_{(\nu)}(E) < \infty. \]
The desired statement then follows from Proposition \ref{prop:psstabprops} (i).
\end{proof}

{In particular, this statement holds for any measure on $\R^d$ or on any of the example spaces discussed in Section \ref{sect:examples}, whose support is compact and has finite $m$-dimensional upper Minkowski content.}

\begin{remark*}
A compact set $E \subset \R^d$ is called \emph{Minkowski $m$-regular} if $\cM^m(E) = \cH^m(E)$.
Kneser's Theorem \cite{Kneser:1955aa}, \cite[Thm.~3.2.39]{Federer1996} states that $m$-rectifiable sets of the form $F(K)$, $K \subset \R^m$ compact, $F \colon \R^m \to \R^d$ Lipschitz, are Minkowski $m$-regular. Ambrosio, Fusco and Pallara \cite[Thm.~2.106]{AFP00} also prove that if $E$ is a {countably $m$-rectifiable compact} subset of $\R^d$, and moreover there exists a Radon measure $\nu \ll \cH^m$ on $\R^d$ and constants $\lambda, \delta > 0$ such that
\[ \nu(B_r(x)) \geq \lambda r^m \quad \text{for all } x \in E \text{ and } r < \delta, \]
then $E$ is again Minkowski $m$-regular. {Note that this is stronger than the condition on $\nu$ in Proposition \ref{prop:minconstab}, and is also a sufficient condition for Corollary \ref{cor:compstab} to yield $(p,m)$-quantizability. }
\end{remark*}

{
For non-compact support, we can also apply Corollary \ref{cor:volgrowstab} under very domain-specific assumptions. 

\begin{example} 
Let $E = \bigcup_{i=1}^\infty F_i(\R^m)$ for a countable family of Lipschitz maps $F_i \colon \R^m \to X$. Then without loss of generality, we can write $E = \bigcup_{j=1}^\infty G_j([0,1]^m)$ with each $G_j$ $1$-Lipschitz. 

Suppose that every ball in $X$ intersects only finitely many sets $E_j := G_j([0,1]^m)$. Then the measure
\[ \nu := \sum_{j=1}^\infty (G_j)_\# (\cL^m|_{[0,1]^m}) \]
is locally finite (n.b. $\nu \neq \cH^m|_E$). For any $x \in E$ and $r < 1$, say $x = G_j(y)$ with $y \in [0,1]^m$, we have
\[ \nu(B_r(x)) \geq \cL^m([0,1]^m \cap G_j^{-1}(B_r(x))) \geq \cL^m([0,1]^m \cap B_r(y)) \geq 2^{-m} \omega_m r^m. \]
Fixing $x_0 \in X$ and $R_0 > 0$, defining the counting function
\[ N(R) := \#\{j \in \N \mid E_j \cap B_R(x_0) \neq \varnothing\} = \#\{j \in \N \mid [0,1]^m \cap G_j^{-1}(B_R(x_0)) \neq \varnothing\}, \]
we have that
\[ \nu(B_{R+R_0}(x_0) \setminus B_R(x_0)) \leq \nu(B_{R+R_0}(x_0)) = \sum_{j=1}^\infty \cL^m([0,1]^m \cap G_j^{-1}(B_{R+R_0}(x_0))) \leq N(R+R_0). \]
Then taking $V(R) := N(R+R_0)$, Corollary \ref{cor:volgrowstab} implies that any measure $\mu$ such that $\mu(X \setminus E) = 0$, which satisfies
\[ \int_X \left[ 1+d(x,x_0)^p + d(x,x_0)^{p+(\alpha+1) p/m} V(d(x,x_0))^{p/m} \right] \di\mu(x) < \infty \]
for some $\alpha > 0$, is $(p,m)$-quantizable on $X$. 
\end{example}

The local finiteness assumption on $E$ is fulfilled e.g.~by locally finite unions of submanifolds, but is intractable for general countably rectifiable sets especially when the maps $F_i$ are obtained from Lipschitz extension theorems. In the case $E = F(\R^m)$, one sufficient condition for local finiteness is that $F$ be a proper map, so that $F^{-1}(B_R(x_0))$ is always contained in a finite union of cubes. 
}

{For measures defined on a single $m$-rectifiable set $E = F(\R^m)$,}
we can also formulate the following Pierce-type integral condition:

\begin{prop}\label{prop:singrectpierce}
Let $F \colon \R^m \to X$ be a Lipschitz map, $\mu \in \M^p_+(X)$ such that $\mu(X \setminus F(\R^m)) = 0$. Assume that $\mu = \rho \cH^m$ and
\[ \int_{\R^m} \left(1 + |x|^{p+\delta} \right) JF(x) \rho(F(x)) \di x < \infty \quad \text{for some } \delta > 0, \]
where $JF$ is the Jacobian of $F$ in the sense of Kirchheim \cite{kirch94}.
Then $\mu$ is $(p,m)$-quantizable.
\end{prop}
\begin{proof}
Consider the a.c.~measure $\nu$ on $\R^m$ with density $JF(x) \rho(F(x))$, which lies in $L^1(\R^m)$ by assumption. By Pierce's lemma, $\nu$ is $(p,m)$-quantizable, thus $F_\# \nu$ is also $(p,m)$-quantizable by Proposition \ref{prop:psstabprops} (iii).

Now by the area formula \cite[Thm.~7]{kirch94}, for $A \subseteq X$ Borel,
\begin{align*}
(F_\# \nu)(A) = \nu(F^{-1}(A)) & = \int_{F^{-1}(A)} JF(x) \rho(F(x)) \di x 
\\ & = \int_X \left[ \sum_{x \in F^{-1}(y)} \rho(F(x)) \One_{F^{-1}(A)}(x) \right] \di \cH^m(y) 
\\ & = \int_X \#(F^{-1}(y)) \rho(y) \One_A(y) \di\cH^m(y)
\\ & \geq \int_{A \cap F(\R^m)} \rho(y) \di\cH^m(y) = \mu(A). 
\end{align*}
Thus $\mu \leq F_\# \nu$, and $\mu$ is also $(p,m)$-quantizable by Proposition \ref{prop:psstabprops} (ii).
\end{proof}

{
We leave open the possibility of obtaining further conditions for the $(p,m)$-quantizability of rectifiable measures under different assumptions.
}

{
Lastly, we give two $1$-rectifiable counterexamples to $(p,1)$-quantizability, in the vein of \cite[Ex.~6.4]{quantbook} and \cite[Thm.~1.7]{iacasym}. Indeed \cite[Ex.~6.4]{quantbook} is itself a discrete, hence trivially $1$-rectifiable measure, but can be modified to be absolutely continuous with respect to $\cH^1$. We construct two $1$-rectifiable measures $\mu \ll \cH^1$, which both satisfy $Q_{p,1}(\mu) = \infty$ hence cannot be $(p,1)$-quantizable: one with non-compact support which still satisfies $\int \rho^{\frac{1}{p+1}} \di\cH^1 < \infty$, and one with compact support for which instead $\int \rho^{\frac{1}{p+1}} \di\cH^1 = \infty$. This is in contrast to the case $s = d$, where the finiteness of higher moments implies the finiteness of $\int \rho^{\frac{d}{p+d}} \di x$ \cite[Rem.~6.3]{quantbook}. {These examples thus also present counterexamples to the original conjecture of Graf and Luschgy \cite[Rem.~13.13]{quantbook}, which posited the finiteness of the quantization coefficient.}

\begin{example}[$1$-rectifiable counterexamples]\label{ex:rectcounter}
Fix $p \in [1,\infty)$. Closely following \cite[Ex.~6.4]{quantbook}, for $k \geq 2$ we set $R_k := 3 \cdot 2^{k-1}$ and $\alpha_k := \frac{1}{2^{kp} k \log^2 k}$, and we consider the $1$-rectifiable measure $\mu \ll \cH^1$ on $\R^2$ given by
\[ \mu = \sum_{k=2}^\infty \alpha_k \frac{1}{\cH^1(\Gamma_k)} \cH^1|_{\Gamma_k}, \]
where $\Gamma_k$ is an arc on the circle $\partial B_{R_k}(0)$ with radius $R_k$, chosen so that $\cH^1(\Gamma_k) \leq \alpha_k$.
Under these assumptions,
\[ \int |x|^p \di\mu(x) = \sum_{k=2}^\infty \alpha_k R_k^p = \sum_{k=2}^\infty \frac{3}{k \log^2 k} < \infty, \]
and for $\rho = \sum_{k=2}^\infty \alpha_k \frac{1}{\cH^1(\Gamma_k)} \One_{\Gamma_k}$ the density of $\mu$ wrt $\cH^1$,
\[ \int \rho^{\frac{1}{p+1}} \di \cH^1 = \sum_{k=2}^\infty \alpha_k^{\frac{1}{p+1}} \cH^1(\Gamma_k)^{\frac{p}{p+1}} \leq \sum_{k=2}^\infty \alpha_k < \infty. \]
Now take an arbitrary $S \subset \R^2$ with $\# S = N$. Define the set of indices
\[ I = \left\{k \geq 2 \mid S \cap \left(B_{2^{k+1}}(0) \setminus B_{2^k}(0) \right) = \varnothing \right\}, \]
so that for any $k \in I$ and $x \in \partial B_{R_k}(0)$,
\[ d(x, S) \geq d(x, \left(B_{2^{k+1}}(0) \setminus B_{2^k}(0) \right)^c) = \min\{|R_k - 2^{k+1}|, |R_k - 2^k|\} = 2^{k-1}. \]
Therefore
\begin{align*}
\int d(x,S)^p \di \mu(x) & = \sum_{k=2}^\infty \alpha_k \frac{1}{\cH^1(\Gamma_k)} \int_{\Gamma_k} d(x,S)^p \di \cH^1(x) 
\\ & \geq \sum_{k \in I} \alpha_k 2^{(k-1)p} \geq \sum_{k = N+2} \alpha_k 2^{(k-1)p} = 2^{-p} \sum_{k = N+2} \frac{1}{k \log^2 k},
\end{align*}
where the second inequality follows from the fact that $\alpha_k 2^{(k-1)p} = \frac{2^p}{k \log^2 k}$ is decreasing. Indeed, $\# I^c \leq N$ by the pigeonhole principle, hence $\#([1,N] \cap I) \geq \#([N+1,\infty) \setminus I)$, and swapping the former set with the latter can only decrease the sum. We observe finally that the last sum is lower bounded by $\frac{1}{\log(N+2)}$ as in \cite[Ex.~6.4]{quantbook}, so that $Q_{p,1}(\mu) = \infty$.

We can also modify this counterexample to produce a $1$-rectifiable measure on $\R^3$ with compact support, for which the integral in Theorem \ref{thm:zadorrect} is infinite. This differs from the case of absolutely continuous measures on $\R^d$, for which the finiteness of higher moments implies the finiteness of the integral functional in Zador's theorem.

Set instead $R_k \equiv 1$ and $\alpha_k = 2^{-kp}$, and let the sets $\Gamma_k$ now be disjoint rectifiable curves on the sphere $\partial B_1(0)$ with length $2^k$. This is possible now that we take the ambient space to be $\R^3$ instead of $\R^2$. Then defining $\mu \ll \cH^1$ as above, $\mu$ is compactly supported hence all the moments of $\mu$ are finite, but 
\[ \int \rho^{\frac{1}{p+1}} \di \cH^1 = \sum_{k=1}^\infty 2^{-\frac{kp}{p+1}} 2^{\frac{kp}{p+1}} = \infty. \] 
Taking $\mu_k := \alpha_k \frac{1}{\cH^1(\Gamma_k)} \cH^1|_{\Gamma_k}$ also gives a direct counterexample to the $(p,1)$-quantizability of $\mu$, since $\mu_k(\R^3) = \alpha_k \to 0$ while 
\[ Q_{p,1}(\mu_k) = C_{p,1} \alpha_k^{\frac{1}{p+1}} \cH^1(\Gamma_k)^{\frac{p}{p+1}}  = C_{p,1} 2^{-\frac{kp}{p+1}} 2^{\frac{kp}{p+1}} \not\to 0. \]
This is also a counterexample to Proposition \ref{prop:minconstab}, since the $1$-dimensional Minkowski content of $E = \bigcup_{k=1}^\infty \Gamma_k$ is infinite by construction.
Moreover, since $\cH^1(E) = \infty$, there exists no probability measure $\nu$ on $E$ such that $\underline \vartheta^{(\nu)}_1(\cdot, \delta)$ is uniformly lower bounded on $E$, cf. Corollary \ref{cor:compstab}.
\end{example}
}

\appendix

\section{Proofs of additivity properties}\label{app:quantcoeffs}

{
Here we prove Propositions \ref{prop:subadd} and \ref{prop:compsupadd} and deduce Theorem \ref{thm:stats}.

For convenience and consistency between the cases $p < \infty$ and $p = \infty$, we introduce the following notation for \emph{$p$-sums}: given a tuple or sequence $(a_i)_{i=1}^m$ of nonnegative real numbers, $m \leq \infty$, we write
\[ {}^p \sum_{i=1}^m a_i := \| (a_i)_i \|_{\ell^p} = \left\{ \begin{matrix} \left(\sum_{i=1}^m a_i^p\right)^{1/p}, & p < \infty; \\ \sup_{i=1}^m a_i, & p = \infty. \end{matrix} \right. \]
For $m = 2$, we write $a_1 +_p a_2$ in place of ${}^p \sum_{i=1}^2 a_i$. We will use this notation mainly for Proposition \ref{prop:enpprops} (i), which can be written compactly as
\[ \mu = \sum_{i=1}^m \mu_i \implies e_p(\mu; S) = {}^p \sum_{i=1}^m e_p(\mu_i; S); \qquad N \geq \sum_{i=1}^m N_i \implies e_{N,p}(\mu) \leq {}^p \sum_{i=1}^m e_{N_i,p}(\mu_i). \]
}

Propositions \ref{prop:subadd} and \ref{prop:compsupadd} both rely on the minimization of functionals of the form
\[ F_\alpha[\beta] := {}^p \sum_{i=1}^m \beta_i^{-1/s} \alpha_i, \quad \alpha, \beta \in [0,\infty)^m, \]
where by convention we set $\beta_i^{-1/s} \alpha_i \equiv 0$ for each index $i$ such that $\alpha_i = 0$. We write $\alpha \ll \beta$ if $\alpha_i = 0$ for each index $i$ such that $\beta_i = 0$.

The following lemma generalizes \cite[Lem.~6.8]{quantbook} to the case $p \leq \infty$:

\begin{lemma}\label{lem:optpsum}
Let $p \in [1,\infty]$, $s \in (0,\infty)$, and set $\frac{1}{p^\prime} := \frac{1}{p} + \frac{1}{s}$. Then for any tuple $\alpha = (\alpha_i)_{i=1}^m \in [0,\infty)^m$ of nonnegative real numbers, 
\[ \min \left\{ F_\alpha[\beta] \mid \beta = (\beta_i)_i \in [0,1]^m, \sum_{i=1}^m \beta_i \leq 1 \right\} = {}^{p^\prime} \sum_{i=1}^m \alpha_i. \]
If $\alpha$ is not identically zero, the unique minimizer is given by
\[ \bar \beta_i := \frac{\alpha_i^{p^\prime}}{\sum_{i=1}^m \alpha_i^{p^\prime}}; \quad i = 1, \ldots, m. \]
\end{lemma}
\begin{proof}
The first equality is trivial for $\alpha_i \equiv 0$, so assume wlog that $\alpha$ is not identically zero. Note also that the sum on the left-hand side is infinite if $\alpha_i > 0$ for any $i$ such that $\beta_i = 0$, so we can assume wlog that $\alpha \ll \beta$.

Set $\gamma_i := \beta_i^{-1/s} \alpha_i$. By the convention used, we always have that $\gamma_i < \infty$ and $\alpha_i = \gamma_i \beta_i^{1/s}$.
Then by the generalized H\"older inequality,
\begin{align*} 
{}^{p^\prime} \sum_{i=1}^m \alpha_i 
= {}^{p^\prime} \sum_{i=1}^m \gamma_i \beta_i^{1/s} 
\leq \left( {}^p \sum_{i=1}^m \gamma_i \right) \left( {}^s \sum_{i=1}^m \beta_i^{1/s} \right)
= \left( {}^p \sum_{i=1}^m \gamma_i \right) \left( \sum_{i=1}^m \beta_i \right)^{1/s}
\leq {}^p \sum_{i=1}^m \gamma_i.
\end{align*}
Since the left-hand side is positive by assumption on $\alpha$, the second inequality is strict unless $\sum_{i=1}^m \beta_i = 1$. For $p < \infty$, the first inequality is attained iff
\[ \gamma_i^p = C (\beta_i^{1/s})^s = C \beta_i \]
for some constant $C > 0$, which implies that 
\[ \beta_i^{1+p/s} = C^{-1} \gamma_i^p \beta_i^{p/s} = C^{-1} \alpha_i^p \quad \iff \quad \beta_i^{1/p+1/s} = C^{-1/p} \alpha_i \quad \iff \quad \beta_i = C^{-p^\prime/p} \alpha_i^{p^\prime}, \]
thus $\beta_i \equiv \bar \beta_i$ by normalization. 

Instead for $p = \infty$, we have $p^\prime = s$ and the first inequality reduces to
\[ \sum_{i=1}^m \gamma_i^s \beta_i \leq \left( \max_{i=1}^m \gamma_i \right)^s \sum_{i=1}^m \beta_i. \]
This inequality is attained iff $\gamma_i = \max_{j=1}^m \gamma_j =: C$ for each $i$ such that $\beta_i > 0$. This itself is equivalent to the condition that $\alpha_i = C \beta^{1/s}$ for each $i$: either $\beta_i = 0$ in which case $\alpha_i = \gamma_i \beta_i^{1/s} = 0$, or $\beta_i > 0$ in which case $\gamma_i = C$ hence $\alpha_i = \gamma_i \beta_i^{1/s}$). Thus the first inequality is again attained iff $\beta_i = C \alpha_i^s$ for each $i = 1, \ldots, M$.
\end{proof}

We also note the following continuity properties of $p$-sums:

\begin{lemma}\label{lem:psumlim}
Let $(a_n)_n$, $(b_n)_n$ be nonnegative real sequences. Then for $p \in [1,\infty]$,
\begin{align*}
\liminf_{n\to\infty} a_n +_p \liminf_{n\to\infty} b_n & \leq \liminf_{n\to\infty} (a_n +_p b_n) \leq \liminf_{n\to\infty} a_n +_p \limsup_{n\to\infty} b_n; \\
\liminf_{n\to\infty} a_n +_p \limsup_{n\to\infty} b_n & \leq \limsup_{n\to\infty} (a_n +_p b_n) \leq
\limsup_{n\to\infty} a_n +_p \limsup_{n\to\infty} b_n.
\end{align*}
Moreover, for $p = \infty$, the last inequality is an equality.
\end{lemma}
\begin{proof}
For $p < \infty$, the inequalities follow from the super-/subadditivity of $\liminf$ and $\limsup$ applied to the $p$th powers. 
Take $p = \infty$. Then
\begin{align*}
 \limsup_{n \to \infty} \max\{a_n, b_n\} & = \lim_{n_0 \to \infty} \sup_{n \geq n_0} \max\{ a_n, b_n \} 
 = \lim_{n_0 \to \infty} \max\{ \sup_{n \geq n_0} a_n, \sup_{n \geq n_0} b_n \} 
 \\ & = \max\{ \lim_{n_0 \to \infty} \sup_{n \geq n_0} a_n, \lim_{n_0 \to \infty} \sup_{n \geq n_0} b_n \} 
 \\ & = \max\{ \limsup_{n \to \infty} a_n, \limsup_{n \to \infty} b_n \},
\end{align*}
noting the continuity of $\max \colon \R \times \R \to \R$. This proves the last inequality as an equality, immediately implying also the second-to-last inequality.
The first inequality
\[ \max\{ \liminf_{n\to\infty} a_n, \liminf_{n\to\infty} b_n \} \leq \liminf_{n\to\infty} \max\{a_n, b_n\} \]
also follows simply by the monotonicity of the $\liminf$. 

It remains to show the second inequality. 
Take a subsequence $(a_{n_k})_k$ such that $\lim_{k\to\infty} a_{n_k} = \liminf_{n\to\infty} a_n$. 
Then applying the last inequality,
\begin{align*}
 \liminf_{n\to\infty} \max\{a_n, b_n\} & \leq \limsup_{k \to \infty} \max\{a_{n_k}, b_{n_k}\} 
\\ & = \max\{ \lim_{k\to\infty} a_{n_k}, \limsup_{k\to\infty} b_{n_k} \}
 \leq \max\{ \liminf_{n\to\infty} a_n, \limsup_{n\to\infty} b_n \}. \qedhere
\end{align*}
\end{proof}

\subsection{Finite subadditivity}\label{app:subadd}

We now prove Propositions \ref{prop:subadd} and \ref{prop:compsupadd}.
We first prove the following preliminary bound:

\begin{lemma}\label{lem:subadd}
Let $p \in [1,\infty]$, $\mu_1, \mu_2 \in \M^p_+(X)$. Set $\mu := \mu_1 + \mu_2$, and let $s \in (0,\infty)$. Then for any choice of proportions $t_1, t_2 \in (0,1)$ such that $t_1 + t_2 = 1$, 
\begin{align*}
\overline Q_{p,s}(\mu) & \leq t_1^{-1/s} \overline Q_{p,s}(\mu_1) +_p t_2^{-1/s} \overline Q_{p,s}(\mu_2); \\
\underline Q_{p,s}(\mu) & \leq t_1^{-1/s} \underline Q_{p,s}(\mu_1) +_p t_2^{-1/s} \overline Q_{p,s}(\mu_2).
\end{align*}
\end{lemma}
\begin{proof}\
\begin{itemize}
\item For the first inequality, let $N \in \N$ such that $N \geq \max\{t_1^{-1}, t_2^{-1}\}$, and set $N_i = N_i(N) := \lfloor t_i N \rfloor \geq 1$. By Proposition \ref{prop:enpprops},
\begin{align*} 
N^{1/s} e_{N,p}(\mu) &\leq N^{1/s} e_{N_1,p}(\mu_1) +_p N^{1/s} e_{N_2,p}(\mu_2) 
\\ & = \left(\frac{N_1}{N}\right)^{-1/s} N_1^{1/s} e_{N_1,p}(\mu_1) +_p \left(\frac{N_2}{N}\right)^{-1/s} N_2^{1/s} e_{N_2,p}(\mu_2).
\end{align*}
Then taking the limit superior of both sides, applying Lemma \ref{lem:psumlim} and noting that $\lim_{N\to\infty} \frac{N_i}{N} = t_i$ yields the first inequality.
\item For the second inequality, instead for $N_1 \geq \frac{t_1}{t_2}$ set $N_2 = N_2(N_1) := \lfloor \frac{t_2}{t_1} N_1 \rfloor \geq 1$ and $N_t = N_t(N_1) := N_1 + N_2(N_1)$. Again we have
\begin{align*} 
N_t^{1/s} e_{N_t,p}(\mu) &\leq N_t^{1/s} e_{N_1,p}(\mu_1) +_p N_t^{1/s} e_{N_2,p}(\mu_2) 
\\ & = \left(\frac{N_1}{N_t}\right)^{-1/s} N_1^{1/s} e_{N_1,p}(\mu_1) +_p \left(\frac{N_2}{N_t}\right)^{-1/s} N_2^{1/s} e_{N_2,p}(\mu_2).
\end{align*}
Taking the limit inferior and noting that $\lim_{N_1\to\infty} \frac{N_1}{N_t} = t_1$ and $\lim_{N_1\to\infty} \frac{N_2}{N_t} = t_2$ yields the result. \qedhere
\end{itemize}
\end{proof}

\begin{proof}[Proof of Proposition \ref{prop:subadd}]
We then minimize over all possible choices of $t_i$ using Lemma \ref{lem:optpsum}:
\[ \overline Q_{p,s}(\mu) \leq \inf_{\substack{t_1, t_2 \in (0,1) \\ t_1 + t_2 = 1}} \left[ t_1^{-1/s} \overline Q_{p,s}(\mu_1) +_p t_2^{-1/s} \overline Q_{p,s}(\mu_2) \right] = \overline Q_{p,s}(\mu_1) +_{p^\prime} \overline Q_{p,s}(\mu_2), \]
and likewise for $\underline Q_{p,s}(\mu_1)$, where $\frac{1}{p^\prime} = \frac{1}{p} + \frac{1}{s}$.
\end{proof}

\subsection{Finite superadditivity}\label{app:supadd}

We now prove the $p^\prime$-superadditivity of lower quantization coefficients on disjoint compact sets, stated above as Proposition \ref{prop:compsupadd}. We prove the following stronger statement:

\begin{prop}[Finite superadditivity, compact case]
Let $p \in [1,\infty]$, $\mu \in \M^p_+(X)$, $s \in (0,\infty)$. Set $\frac{1}{p^\prime} := \frac{1}{p} + \frac{1}{s}$.

Let $K_1, \ldots, K_m \subseteq X$ be disjoint compact sets. Let $\cS = (S_k)_{N_k}$ be an asymptotically optimal sequence of quantizers for $\mu$, and consider the sequence of proportions $(v_k)_k \subset [0,1]^m$ given by
\[ v_{k,i} := \frac{\#(S_k \mres K_i)}{N_k}, \quad i = 1, \ldots, m; \ k \in \N. \]
Then $\sum_{i=1}^m v_{k,i} \leq 1$ for $k$ sufficiently large, and for any limit point $\bar v \in [0,1]^m$ of the sequence $(t_k)_k$,
\begin{itemize}
\item For each index $i = 1, \ldots, m$, the following inequality holds:
\[ \underline{Q}_{p,s}^{(\mu)}(K_i; \cS \mres K_i) \leq \bar v_i^{1/s} \underline{Q}_{p,s}(\mu). \]
In particular, if $\underline{Q}_{p,s}(\mu) < \infty$, then $\bar v_i = 0$ implies $\underline{Q}_{p,s}^{(\mu)}(K_i; \cS \mres K_i) = 0$ thus also $\underline{Q}_{p,s}^{(\mu)}(K_i) = 0$. 
\item The following general inequality holds:
\[ \underline{Q}_{p,s}(\mu) \geq {}^p \sum_{i=1}^m \bar v_i^{-1/s} \underline{Q}_{p,s}^{(\mu)}(K_i; \cS \mres K_i) \geq {}^{p^\prime} \sum_{i=1}^m \underline{Q}_{p,s}^{(\mu)}(K_i), \]
where by convention the first summand is $0$ irrespective of $\bar v_i$ when $\underline{Q}_{p,s}^{(\mu)}(K_i; \cS \mres K_i) = 0$.
\end{itemize}
\end{prop}
\begin{proof}
By Proposition \ref{prop:admisseq}, since $e_p(\mu; S_k) \to 0$, $d(\cdot, S_k) \to 0$ uniformly on $K = \bigsqcup_{i=1}^m K_i$ compact. 
Since the sets $K_i$ are compact and disjoint, there exists $\veps > 0$ such that the open neighborhoods $K_i^\veps$ are also disjoint. 

Thus for $k$ sufficiently large, each $\max_{x \in K_i} d(x, S_k) < \veps$, and $S_k \mres K_i \subseteq S_k \cap K_i^\veps$, which are disjoint.
This shows that the sets $\{S_k \mres K_i\}_{i=1}^m$ are disjoint subsets of $S_k$ for $k$ sufficiently large, which implies that $\sum_{i=1}^m v_{k,i} \leq 1$.

We now characterize the limit points of the sequence $(v_k)_k$, which always exist by the compactness of $[0,1]^m$.
We can assume wlog that $v_k \to \bar v \in [0,1]^m$, since restricting to a subsequence preserves $Q_{p,s}(\mu; \cS) = \underline{Q}_{p,s}(\mu)$ and can only increase $\underline Q_{p,s}^{(\mu)}(K_i; \cS \mres K_i)$. 

For any index $i = 1, \ldots, m$, we first observe
\[ \#(S_k \mres K_i)^{1/s} e_p^{(\mu)}(K_i; S_k \mres K_i) = v_{k,i}^{1/s} N_k^{1/s} e_p^{(\mu)}(K_i; S_k) \leq v_{k,i}^{1/s} N_k^{1/s} e_p(\mu; S_k), \]
and taking the $\liminf$ of both sides implies the first inequality.

For the second inequality, assume wlog that $\underline{Q}_{p,s}(\mu) < \infty$ and that each $\underline{Q}_{p,s}^{(\mu)}(K_i; \cS \mres K_i) > 0$, so that also each $\bar v_i > 0$ (otherwise, we can simply exclude $K_i$ by the above argument). Since the sets $K_i$ are disjoint, Proposition \ref{prop:enpprops} (i) implies
\[ N_k^{1/s} e_p(\mu; S) \geq N_k^{1/s} e^{(\mu)}_p(K; S) = {}^p \sum_{i=1}^m N_k^{1/s} e^{(\mu)}_p(K_i; S) = {}^p \sum_{i=1}^m v_{k,i}^{-1/s} \#(S_k \mres K_i)^{1/s} e^{(\mu)}_p(K_i; S \mres K_i). \]
Taking the $\liminf$ of both sides and applying Lemma \ref{lem:psumlim} yields the second inequality. The final inequality then follows from Lemma \ref{lem:optpsum}.
\end{proof}

Under the assumption that $\mu$ is $(p,s)$-quantizable and the quantization coefficients of subsets exist, we can forgo the compactness assumption and combine both additivity statements, simultaneously deducing Theorem \ref{thm:stats}:

\begin{theorem}[Finite additivity] \label{thm:finadd}
Let $p \in [1,\infty)$, $s \in (0,\infty)$. Set $\frac{1}{p^\prime} := \frac{1}{p} + \frac{1}{s}$.

Let $\mu \in \M^p_+(X)$ be $(p,s)$-quantizable, $A \subseteq X$ Borel such that $Q^{(\mu)}_{p,s}(\partial A) = 0$ and $Q^{(\mu)}_{p,s}(A)$ exists. Then
\begin{itemize}
\item $\underline Q_{p,s}(\mu)^{p^\prime} = Q^{(\mu)}_{p,s}(A)^{p^\prime} + \underline Q^{(\mu)}_{p,s}(A^c)^{p^\prime}$.
\item Suppose $0 < \underline Q_{p,s}(\mu) < \infty$. Then for any asymptotically optimal sequence $\cS = (S_k)_{N_k}$ for $\mu$, we have
\[ \lim_{k \to \infty} \frac{\# (S_k \cap A)}{N_k} = \frac{Q^{(\mu)}_{p,s}(A)^{p^\prime}}{\underline Q_{p,s}(\mu)^{p^\prime}}; \quad \lim_{k \to \infty} \frac{\# (S_k \cap A^c)}{N_k} = \frac{\underline Q^{(\mu)}_{p,s}(A^c)^{p^\prime}}{\underline Q_{p,s}(\mu)^{p^\prime}}. \]
\end{itemize}
\end{theorem}
\begin{proof}
Proposition \ref{prop:subadd} immediately implies the inequality
\[ \underline Q_{p,s}(\mu)^{p^\prime} \leq Q^{(\mu)}_{p,s}(A)^{p^\prime} + \underline Q^{(\mu)}_{p,s}(A^c)^{p^\prime}. \]
For the converse inequality, let $\cS$ be an arbitrary asymptotically optimal sequence for $\mu$, and write
\[ v_{k,1} := \frac{\# (S_k \cap A)}{N_k}; \quad v_{k,2} := \frac{\# (S_k \cap A^c)}{N_k}. \]
Evidently $v_{k,1} + v_{k,2} \leq 1$ for arbitrary $k$. Restricting to a subsequence, we can assume wlog that $v_{k,i} \to \bar v_i \in [0,1]$.

Consider the disjoint open subsets $U_1 = \inter A$ and $U_2 = \inter A^c$. By the assumption that $Q^{(\mu)}_{p,s}(\partial A) = 0$, Proposition \ref{prop:subadd} implies that $Q^{(\mu)}_{p,s}(U_1) = Q^{(\mu)}_{p,s}(A)$ and $\underline Q^{(\mu)}_{p,s}(U_2) = \underline Q^{(\mu)}_{p,s}(A^c)$.

Let $\alpha > 1$ be arbitrary. By the assumption of $(p,s)$-quantizability, we can take compact subsets $K_i \subseteq U_i$ such that $\underline Q^{(\mu)}_{p,s}(K_i) \geq \alpha^{-1} \underline Q^{(\mu)}_{p,s}(U_i)$. Taking $\veps > 0$ such that each $K_i^\veps \subseteq U_i$, Proposition \ref{prop:admisseq} again implies that for $k$ sufficiently large, $S_k \mres K_i \subseteq S_k \cap U_i$ are disjoint, and
\[ t_{k,i} := \frac{\# (S_k \mres K_i)}{N_k} \leq \frac{\# (S_k \cap U_i)}{N_k} \leq v_{k,i}. \]
Restricting to a further subsequence, we can assume that $t_{k,i} \to \bar t_i \leq \bar v_i$. We then deduce from Proposition \ref{prop:compsupadd} that $\bar v_i \geq \bar t_i > 0$ whenever $\underline Q^{(\mu)}_{p,s}(U_i) > 0$, and that 
\[ \underline{Q}_{p,s}(\mu) \geq {}^p \sum_{i=1}^2 \bar t_i^{-1/s} \underline{Q}^{(\mu)}_{p,s}(K_i) \geq {}^p \sum_{i=1}^2 \bar v_i^{-1/s} \underline{Q}^{(\mu)}_{p,s}(K_i) \geq \alpha^{-1} \ {}^p \sum_{i=1}^2 \bar v_i^{-1/s} \underline{Q}^{(\mu)}_{p,s}(U_i). \]
We can then let $\alpha \to 1^-$ to obtain
\[ \underline{Q}_{p,s}(\mu) \geq \bar v_1^{-1/s} Q^{(\mu)}_{p,s}(A) +_{p} \bar v_2^{-1/s} \underline Q^{(\mu)}_{p,s}(A^c) \geq Q^{(\mu)}_{p,s}(A) +_{p^\prime} \underline Q^{(\mu)}_{p,s}(A^c). \]
This proves the desired equality, and shows also that $\bar v$ minimizes the above expression. When $0 < \underline Q_{p,s}(\mu) < \infty$, we have that $Q^{(\mu)}_{p,s}(A)$ and $\underline Q^{(\mu)}_{p,s}(A^c)$ are both finite and not both zero, thus by Lemma \ref{lem:optpsum}, we obtain that
\[ \bar v_1 = \frac{Q^{(\mu)}_{p,s}(A)^{p^\prime}}{\underline Q_{p,s}(\mu)^{p^\prime}}; \quad \bar v_2 = \frac{Q^{(\mu)}_{p,s}(A^c)^{p^\prime}}{\underline Q_{p,s}(\mu)^{p^\prime}}. \]
This shows that the limit point of the sequence $(v_k)_k$ is unique, yielding the desired convergence statement.
\end{proof}

One can likewise argue that for $A$ compact, the sequence $\cS \mres A$ is also asymptotically optimal for $\mu|_A$, by approximating only $A^c$ by compact subsets. 

The above argument only uses $(p,s)$-quantizability to ensure that the quantization coefficients of $A$ and $A^c$ can be approximated from below by compact subsets. 
This property can also be satisfied by appropriate sets in the case $p = \infty$, in which $(p,s)$-quantizability is not applicable, e.g.~by Jordan measurable subsets of $\R^d$.
One could then impose such assumptions on the set $A$ in order to prove Theorem \ref{thm:stats} also for $p = \infty$.

\section{Concentration inequalities}\label{app:conc}

If a measure $\nu$ satisfies concentration inequalities of the form $\nu(B_r(x)) \geq \vartheta r^s$ or $\nu(B_r(x)) \leq \vartheta r^s$, the quantization coefficients of a set $A$ with respect to $\nu$ can be controlled by $\vartheta^{-1} \nu(A)$. 
{These inequalities are quantitative statements of $\nu$ being at most resp.~at least $s$-dimensional, and yield upper resp.~lower bounds on the quantization coefficients of $\nu|_A$.}

The arguments we apply in this section are elementary and previously known, cf.~\cite[Sect.~12]{quantbook} and \cite[Sect.~4]{discapprox}; we nevertheless provide full proofs for the sake of completeness and precision.

\subsection{Lower bound}
We first prove the following lower bound, adapted from \cite[Prop.~4.2]{discapprox}:

\begin{lemma}[Lower bound]\label{lem:conclb}
Let $A \subseteq X$ be Borel with $\nu(A) < \infty$, and suppose there exist $\vartheta, \delta > 0$ such that
\[ \nu(A \cap B_r(x)) \leq \vartheta r^s \quad \text{for all } x \in X, r < \delta. \]
Then for each $p \in [1,\infty]$, the lower quantization coefficient of $A$ with respect to $\nu$ is bounded as follows:
\[ \underline{Q}^{(\nu)}_{p,s}(A)^{p^\prime} \geq \left( \frac{s}{s+p} \right)^{\frac{s}{s+p}} \vartheta^{-\frac{p}{s+p}} \nu(A), \] 
{where $\frac{1}{p^\prime} = \frac1p + \frac1s$.}
By convention, for $p = \infty$, $\left( \frac{s}{s+p} \right)^{\frac{s}{s+p}} := \lim_{t \to 0^+} t^t = 1$.
\end{lemma}
\begin{proof}
We prove the inequality for $p < \infty$; the case $p = \infty$ follows by letting $q < \infty$ arbitrarily large, bounding
\[ \underline{Q}^{(\nu)}_{q,s}(A) \leq \nu(A)^{\frac{1}{q}} \underline{Q}_{\infty,s}(A) \implies \underline{Q}_{\infty,s}(A)^{\frac{qs}{q+s}} \geq \nu(A)^{-\frac{s}{q+s}} \underline{Q}^{(\nu)}_{q,s}(A)^{\frac{qs}{q+s}} \geq \left( \frac{s}{s+q} \right)^{\frac{s}{s+q}} \vartheta^{-\frac{q}{s+q}} \nu(A)^{\frac{q}{s+q}} \]
and letting $q \to \infty$ so that $\frac{q}{s+q} \to 1$.

Suppose wlog that $A \subseteq \supp\nu$, noting that neither side of the inequality is affected by this restriction, and that $\nu(A) > 0$, since otherwise the inequality is trivial.
 Let $S \in \cS_N(X)$, $N \in \N$. By the layer cake representation, we have
\begin{align*} 
e_p^{(\nu)}(A; S)^p 
= \int_0^\infty p r^{p-1} \nu(A \cap \{ d(\cdot, S) \geq r \}) \di r
= \int_0^\infty p r^{p-1} \left[ \nu(A) - \nu(A \cap \bigcup_{x \in S} B_r(x) ) \right] \di r. 
\end{align*}
By the union bound and the assumption on $A$, for any $r < \delta$, we have
\[ \nu(A \cap \bigcup_{x \in S} B_r(x)) \leq \sum_{x \in S} \nu(A \cap B_r(x)) \leq N \vartheta r^s = \nu(A) \frac{r^s}{r_N^s } \quad \text{where we set} \quad r_N := \left( \frac{\nu(A)}{N \vartheta} \right)^{1/s}. \]
For $N$ large enough so that $r_N < \delta$, the following then holds for all $S \in \cS_N(X)$:
\begin{align*} 
e_p^{(\nu)}(A; S)^p 
& \geq \nu(A) \int_0^{r_N} p r^{p-1} \left[ 1 - \frac{r^s}{r_N^s} \right] \di r
= \left( 1 - \frac{p}{s+p} \right) r_N^p \nu(A)
= \frac{s}{s+p} N^{-p/s} \vartheta^{-p/s} \nu(A)^{1+p/s}. 
\end{align*}
Taking the infimum over all admissible $S$, rearranging and taking the $p$th root, we obtain
\[ N^{1/s} e_{N,p}^{(\nu)}(A) \geq \left( \frac{s}{s+p} \right)^{1/p} \vartheta^{-1/s} \nu(A)^{\frac{s+p}{sp}}. \]
Letting $N \to \infty$ yields the inequality for $p < \infty$. 
\end{proof}

\subsection{Upper bounds}
We obtain a converse upper bound by controlling the packing numbers of compact sets. First recall the duality between packings and covers:

\begin{recall*}
Let $A \subseteq X$ be compact. The packing and covering numbers of $A$ satisfy the following inequalities:
\[ N(A; 2r) \leq P(A; r) \leq N(A; r) \quad \text{for all} \quad r > 0. \]
The first inequality follows from maximality: since $P(A; r) < \infty$, there exists an $r$-packing $P \subseteq A$ such that $P(A; r) = \# P$. Then for any $x \in A \setminus P$, $P \sqcup \{x\}$ cannot be an $r$-packing of $A$, so there exists some $a \in P$ such that $B_r(x) \cap B_r(a) \neq \emptyset$ hence $x \in B_{2r}(A)$. Thus $P$ is a $2r$-cover of $A$.

The second inequality follows from the pigeonhole principle: given any $r$-packing $P \subseteq A$ and $r$-cover $S \subseteq X$ of $A$, each element $x \in S$ is contained in one of the balls $B_r(a)$, $a \in S$, and if two distinct elements $x, y \in P$ were contained in the same ball $B_r(a)$, that would imply that $a \in B_r(x) \cap B_r(y)$. Hence there always exists an injection from an $r$-packing to an $r$-cover. This shows in particular that $P(A; r) < \infty$ for $A$ compact, since there always exists a finite $r$-cover of $A$ for arbitrary $r > 0$.

Note moreover that for any cover of $A$ by $N$ open balls of radius $r$ with centers $S = \{x_i\}_{i=1}^N$, by compactness $e_\infty(A; S) = \max_{x \in A} d(x, S)$ is attained, hence strictly less than $r$. Hence for all $N < \# A$, we have that $e_{N,\infty}(A) > 0$ and $N(A; e_{N,\infty}(A)) > N$.

Since packings are always subsets of $A$, the packing numbers of $A$ are independent of the choice of the ambient domain $X$. While covers of $A$ need not be comprised of elements of $A$, every $r$-cover $S \subseteq X$ still induces a $2r$-cover $S^\prime \subseteq A$ with at most the same number of elements. Thus considering the covering problems separately on the metric spaces $A$ and $X$, we have $e_{N,\infty}^{[X]}(A) \leq e_{N,\infty}^{[A]}(A) \leq 2 e_{N,\infty}^{[X]}(A)$ for each $N$.
\end{recall*}

For the case $p = \infty$, we prove a more general statement {in terms of Minkowski contents (see Definition \ref{def:mincon})} which also encompasses sets of dimension less than $s$, cf.~\cite[Lem.~2.10]{covgrow}.

\begin{lemma}[Packing upper bound]\label{lem:packub}
Let $(X, \nu)$ be a metric measure space, $A \subseteq X$ compact such that for some $\vartheta, \delta, s > 0$,
\[ \nu(B_r(x)) \geq \vartheta r^s \quad \text{for all } x \in A,\ r < \delta. \]
Then for each $0 < m \leq s$,
\[ \overline{Q}_{\infty,m}^{[X]}(A)^m \leq \overline{Q}_{\infty,m}^{[A]}(A)^m \leq 2^m \omega_{s-m} \vartheta^{-1} \overline \cM^m_{(\nu)}(A), \]
where $\overline{Q}_{\infty,m}^{[A]}(A)$ denotes the upper $\infty$-quantization coefficient of $A$ within the compact metric space $A$, i.e.~when covers are constructed only from elements of $A$.
\end{lemma}
\begin{proof}
Assume wlog that $A$ is infinite, so that $e_N := e^{[A]}_{N,\infty}(A) > 0$ for all $N$. Proposition \ref{prop:enpconv} applied to the metric space $A$ implies that $e_N \to 0$ as $N \to \infty$.

Let $N \in \N$ be sufficiently large so that $e_N \leq \delta$, and take $0 < r < e_N/2$. 
Take a maximal $r$-packing $S$ on $A$. Then $S$ is a $2r$-cover of $A$ by elements of $A$, thus $\# S > N$ since $2r < e_N$. Since $S$ is an $r$-packing, we have
\[ \nu(A^r) \geq \nu\left( \bigsqcup_{x \in S} B_r(x) \right) = \sum_{x \in S} \nu(B_r(x)) \geq N \vartheta r^s. \]
Rearranging and taking the supremum over all such $r$,
\[ \overline{Q}_{\infty,m}^{[A]}(A)^m = \limsup_{N \to \infty} N e^{[A]}_{N,\infty}(A)^m = 2^m \limsup_{N \to \infty} \sup_{0 < r < e_N/2} N r^m \leq 2^m \limsup_{N \to \infty} \sup_{0 < r < e_N/2} \frac{\nu(A^r)}{\vartheta r^{s-m}}. \]
The desired inequality then follows from the fact that $e_N/2 \to 0$ and
\[ \overline \cM^m_{(\nu)}(A) = \inf_{r_0 > 0} \sup_{0 < r < r_0} \frac{\nu(A^r)}{\omega_{s-m} r^{s-m}} = \lim_{r_0 \to 0} \sup_{0 < r < r_0} \frac{\nu(A^r)}{\omega_{s-m} r^{s-m}}. \qedhere \]
\end{proof}

Note that we do not assume $X$ itself to be Polish; the compactness of $A$ still ensures that finite $r$-covers always exist. 

We can then deduce the following bound for $p < \infty$ in the case $m = s$:

\begin{lemma}[Upper bound]\label{lem:concub}
Let $A \subseteq X$ be compact, and let $\nu$ be a locally finite Borel measure on $X$ admitting constants $\vartheta, \delta, s > 0$ such that
\[ \nu(B_r(x)) \geq \vartheta r^s \quad \text{for all } x \in A,\ r < \delta. \]
Then for each $p \in [1,\infty]$, the upper quantization coefficient of $A$ with respect to $\nu$ is bounded as follows:
\[ \overline{Q}^{(\nu)}_{p,s}(A)^{p^\prime} \leq 2^{p^\prime} \vartheta^{-\frac{p}{s+p}} \nu(A), \] 
{where $\frac{1}{p^\prime} = \frac1p + \frac1s$, and} again the same upper bound holds whether quantizers are chosen from $X$ or only from $A$.
\end{lemma}
\begin{proof}
For $m = s$, Lemma \ref{lem:packub} yields
\[ \overline{Q}_{\infty,s}(A)^s \leq 2^s \vartheta^{-1} \nu(A), \]
which proves the statement for $p = \infty$. For $p < \infty$, {Proposition \ref{prop:coeffprops} (iii) implies that}
\[ \overline{Q}^{(\nu)}_{p,s}(A) \leq \nu(A)^{1/p} \overline{Q}_{\infty,s}(A) \leq 2 \vartheta^{-\frac{1}{s}} \nu(A)^{\frac{1}{p}+\frac{1}{s}}, \]
and exponentiating by $p^\prime$ yields the statement for $p < \infty$.
\end{proof}

\noindent {{\bf Acknowledgments.}}
The author expresses his warm gratitude to his advisor, Mikaela Iacobelli, for her constructive feedback and her continued academic and personal support.
The author is also very grateful to Urs Lang for his insightful comments and discussions,
{and the two anonymous referees for their careful reviews leading to substantial improvements in the presentation and scope of the paper.}

\bibliographystyle{alpha}
\raggedright
\bibliography{references}

\end{document}